\documentclass[11pt,a4paper,reqno]{amsart}

\usepackage[T1]{fontenc}
\usepackage{lmodern}
\usepackage{amsmath,amsthm,amssymb,amsfonts}
\usepackage[initials,msc-links]{amsrefs}
\usepackage{url}
\usepackage[margin=1.35in]{geometry}
\usepackage[dvipsnames]{xcolor}
\usepackage{enumerate}
\usepackage{stmaryrd} 
\usepackage{mathrsfs}
\usepackage{mathtools}
\usepackage{color, colortbl}
\usepackage{bbm}
\usepackage{bm}
\usepackage{tikz-cd}
\usetikzlibrary{matrix}
\usepackage[multiple]{footmisc}
\usepackage{enumitem}
\usepackage{filecontents}
\usepackage[dvipsnames]{xcolor}
\usepackage{cleveref}
\usepackage{hyperref}
\hypersetup{
	colorlinks,
	citecolor=black,
	filecolor=black,
	linkcolor=black,
	urlcolor=black
}

\usepackage{caption}
\usepackage{tikz}

\newtheorem{theorem}{Theorem}
\newtheorem{lemma}{Lemma}[section]
\newtheorem{proposition}{Proposition}[section]
\newtheorem{corollary}{Corollary}[theorem]

\theoremstyle{definition}
\newtheorem{definition}{Definition}
\newtheorem{remark}{Remark}[section]
\newtheorem{example}{Example}[section]

\newtheorem{notation}{Notation}

\newcommand{\Crm}{\mathrm{C}}

\newcommand{\Irm}{\mathrm{I}}

\newcommand{\Lrm}{\mathrm{L}}

\newcommand{\Wrm}{\mathrm{W}}

\newcommand{\Acal}{\mathcal{A}}
\newcommand{\Bcal}{\mathcal{B}}

\newcommand{\Dcal}{\mathcal{D}}

\newcommand{\Fcal}{\mathcal{F}}

\newcommand{\Hcal}{\mathcal{H}}
\newcommand{\Ical}{\mathcal{I}}

\newcommand{\Lcal}{\mathcal{L}}
\newcommand{\Mcal}{\mathcal{M}}

\newcommand{\Ocal}{\mathcal{O}}

\newcommand{\Scal}{\mathcal{S}}

\newcommand{\Bfrak}{\mathfrak{B}}

\newcommand{\Sbf}{\mathbf{S}}

\newcommand{\pbf}{\mathbf{p}}

\newcommand{\Abb}{\mathbb{A}}
\renewcommand{\Bbb}{\mathbb{B}}

\newcommand{\Kbb}{\mathbb{K}}
\newcommand{\Lbb}{\mathbb{L}}

\newcommand{\Pbb}{\mathbb{P}}

\DeclareMathOperator{\id}{id}

\DeclareMathOperator{\im}{Im}

\DeclareMathOperator{\curl}{curl}

\DeclareMathOperator{\rank}{rank}

\DeclareMathOperator{\spn}{span}

\newcommand{\set}[2]{\left\{\, #1 \  \textup{\textbf{:}}\  #2 \,\right\}}
\newcommand{\setn}[2]{\{\, #1 \  \textup{\textbf{:}}\  #2 \,\}}

\newcommand{\setBB}[2]{\biggl\{\, #1 \  \textup{\textbf{:}}\  #2 \,\biggr\}}

\newcommand{\dpr}[1]{\langle #1 \rangle}

\newcommand{\dprb}[1]{\bigl\langle #1 \bigr\rangle}

\newcommand{\cl}[1]{\overline{#1}}

\newcommand{\dd}{\;\mathrm{d}}

\newcommand{\N}{\mathbb{N}}
\newcommand{\R}{\mathbb{R}}
\newcommand{\C}{\mathbb{C}}

\newcommand{\loc}{\mathrm{loc}}
\newcommand{\sym}{\mathrm{sym}}

\newcommand{\todown}{\downarrow}

\newcommand{\embed}{\hookrightarrow}

\newcommand{\BigO}{\mathrm{\textup{O}}}

\newcommand{\sbullet}{\begin{picture}(1,1)(-0.5,-2)\circle*{2}\end{picture}}
\newcommand{\frarg}{\,\sbullet\,}
\newcommand{\BV}{\mathrm{BV}}
\newcommand{\BD}{\mathrm{BD}}

\newcommand{\eps}{\epsilon}

\newcommand{\proofstep}[1]{\textit{#1}}

\newcommand{\Leb}{\mathscr L}

\renewcommand{\eps}{\varepsilon}
\renewcommand{\phi}{\varphi}

\newcommand{\M}{\mathcal M}

\DeclareMathOperator{\Gr}{Gr}

\DeclareMathOperator{\tr}{tr}

\DeclareMathOperator{\Div}{div}

\newcommand{\mres}{\mathbin{\vrule height 1.6ex depth 0pt width
        0.13ex\vrule height 0.13ex depth 0pt width 1.3ex}}

\def\vint_#1{\mathchoice%
    {\mathop{\kern 0.2em\vrule width 0.6em height 0.69678ex depth -0.58065ex
            \kern -0.8em \intop}\nolimits_{\kern -0.4em#1}}%
    {\mathop{\kern 0.1em\vrule width 0.5em height 0.69678ex depth -0.60387ex
            \kern -0.6em \intop}\nolimits_{#1}}%
    {\mathop{\kern 0.1em\vrule width 0.5em height 0.69678ex depth -0.60387ex
            \kern -0.6em \intop}\nolimits_{#1}}%
    {\mathop{\kern 0.1em\vrule width 0.5em height 0.69678ex depth -0.60387ex
            \kern -0.6em \intop}\nolimits_{#1}}}

\newcommand{\aveint}[2]{\mathchoice%
    {\mathop{\kern 0.2em\vrule width 0.6em height 0.69678ex depth -0.58065ex
            \kern -0.8em \intop}\nolimits_{\kern -0.45em#1}^{#2}}%
    {\mathop{\kern 0.1em\vrule width 0.5em height 0.69678ex depth -0.60387ex
            \kern -0.6em \intop}\nolimits_{#1}^{#2}}%
    {\mathop{\kern 0.1em\vrule width 0.5em height 0.69678ex depth -0.60387ex
            \kern -0.6em \intop}\nolimits_{#1}^{#2}}%
    {\mathop{\kern 0.1em\vrule width 0.5em height 0.69678ex depth -0.60387ex
            \kern -0.6em \intop}\nolimits_{#1}^{#2}}}

\DeclareMathOperator{\proj}{Proj}

\title[Slicing and fine properties for $\BV^\Acal$]{Slicing and fine properties for \\ functions 
	 with Bounded $\Acal$-variation}

\date{\today}

\author[A. Arroyo-Rabasa]{Adolfo Arroyo-Rabasa}
\address{Mathematics Institute, University of Warwick, Coventry CV4 7AL, UK}
\email{\href{mailto:Adolfo.Arroyo-Rabasa@warwick.ac.uk}{Adolfo.Arroyo-Rabasa@warwick.ac.uk}}

\subjclass[2010]{49Q20, 26B30}
\keywords{approximate continuity, bounded $\Acal$-variation, elliptic operator, fine properties, jump set, rectifiability,  slicing, structure theorem}

\begin{document}
    
    \begin{abstract}
    	We study the slicing and fine properties of functions in
    	 $\BV^\Acal$, the space of functions with bounded $\Acal$-variation. Here, $\Acal$ is a  homogeneous linear
    	 differential operator with constant coefficients (of arbitrary order). Our main result is the characterization of all $\Acal$ satisfying the following one-dimensional structure theorem: every $u \in \BV^\Acal$ can be sliced into one-dimensional $\BV$-sections. Moreover, decomposing $\Acal u$ into an absolutely continuous part $\Acal^a u$,  a Cantor part $\Acal^c u$ and a jump part $\Acal^j u$, each of these measures can be recovered from the corresponding classical $D^a,D^c$ and $D^j$ $BV$-derivatives of its one-dimensional sections. By means of this result,  we are able  
    	 to analyze the set of Lebesgue points as well as the set of jump points where these functions have approximate one-sided limits. Thus, proving a  structure and fine properties theorem in $\BV^\Acal$. Our results extend most of the classical fine properties of $\BV$ (and all of those known for $\BD$). In particular, we establish a slicing theory and fine properties for $\mathscr {BV}^k$, $\BD^k$ and a whole class of $\BV^\Acal$-spaces that is not covered by the existing theory.

    \end{abstract}

    \maketitle
    
    \tableofcontents

    \section{Introduction}

Let $\Omega$ be an open set of $\R^n$. In this article we consider the space 
\[
	\BV^\Acal(\Omega) = \set{\Lrm^1\big(\Omega;\R^N\big)}{\Acal u \in \M\big(\Omega;\R^M\big)}
\] 
of \emph{functions with bounded $\Acal$-variation} over $\Omega$. Here, $\Mcal(\Omega;\R^M)$ denotes the space of $\R^M$-valued Radon measures on $\Omega$ and $\Acal : \Dcal'(\Omega;\R^N) \to \Dcal'(\Omega;\R^M)$ is a $k$\textsuperscript{th} order homogeneous system of linear partial differential operators
with constant coefficients. More precisely, the operator $\Acal$ acts on functions $u : \Omega \subset \R^n \to \R^N$  as
\begin{equation}\label{eq:A}
\Acal u  \, \coloneqq \, \sum_{|\alpha| = k} A_\alpha \partial^\alpha u,
\end{equation}
where  the coefficients $A_\alpha$ are matrices in $\R^M \otimes \R^N \cong \R^{M \times N}$, $\alpha = (\alpha_1,\dots,\alpha_n)$ is a multi-index in $\N_0^n$, and $\partial^\alpha$ denotes the  distributional partial derivative  $\partial_1^{\alpha_1}\cdots \partial_n^{\alpha_n}$ of order $|\alpha| = \alpha_1 + \cdots + \alpha_n$.

The purpose of this work is to give a comprehensive determination of the structural and fine properties of functions in $\BV^\Acal(\Omega)$, very much in the fashion of what is known for $\BV(\Omega;\R^N)$. A fundamental part of this endeavor is to construct a unified theory that circumvents  the use of tools that are exclusive to the theory of gradients (such as the co-area formula and the theory of sets of finite perimeter). In this regard, we have found that some of the core ideas used to establish the \emph{same} fine properties in $\BD$ remain valid in the $\BV^\Acal$-framework; particularly, the one-dimensional slicing theory (see~\cite{bellettini1992caratterizzazione,ambrosio1997fine-properties}). 
In its current form, the implementation of slicing techniques appeals directly to the unique structure  that  the gradient and the symmetric gradient possess. As such, little is known about the overall structural properties of operators admitting \emph{slicing} into lower dimensional elements. Our main contribution in this domain is the following: First, we introduce the notion of \guillemotleft~{$\rank_\Acal(w)$}~\guillemotright~for a vector $w$ in the target space \guillemotleft~$\R^M$~\guillemotright\; of the operator $\Acal$. This  concept extends the classical notion of rank  when $\R^M \cong \R^{m\times d}$ is a space of matrices, but is also sensible to $\Acal$ in a suitable algebraic way. For first-order operators (when $k = 1$), $\rank_\Acal$-one vectors $w$ can be formally defined as those vectors satisfying
%
\[
	w \cdot \Acal u = \partial_{\xi} (u \cdot e)   \quad \text{for all $u \in \Crm^\infty(\R^n;\R^N)$},
\]
for some direction $\xi \in \R^n$ and some coordinate $e \in \R^N$. In this case, one may think of $(\xi,e)$ as coordinates on which the operator $\Acal$ controls the partial derivative operator $\partial_\xi(\;)^e$. Naturally, the more $\rank_\Acal$-one tensors exist, the more individual partial derivatives are controlled by the operator $\Acal$.
In this regard, we show that the \emph{algebraic} mixing property
\[
	\bigcap_{\substack{\pi \, \le \, \R^n\\\dim(\pi) = n-1}} \spn 
	\set{\im \Abb^k(\xi)}{\xi \in \pi} = \{0\}, \qquad \Abb^k(\xi) \coloneqq \sum_{|\alpha| = k}\xi^\alpha A_\alpha,
\]
is equivalent to the existence of a family $\{w_1,\dots,w_M\} \subset \R^M$ such that
\[
	\spn\{w_1,\dots,w_M\} = \R^M, \qquad \rank_\Acal(w_j) \le 1.
\]  
Then, we  show that this spanning property is equivalent to the following \emph{functional} property: the space  $\BV^\Acal(\Omega)$ admits a definition by {slicing} into one-dimensional $\BV$-sections. We proceed to develop a slicing theory in $\BV^\Acal$, which is based on the notion of $\rank_\Acal$ and the understanding of the functional properties that stem from the mixing condition. These slicing methods are crucial for carrying the analysis of Lebesgue point properties, which,  we subsequently use to establish the structure and fine properties for $\BV^\Acal$-spaces. 

Before embarking in a formal discussion of the main theorems, let us bring some perspective to our results by briefly recalling the  slicing and fine properties of the classical $\BV$-theory.
The space of functions of bounded variation $\BV(\Omega;\R^N)$ consists of all functions $u \in \Lrm^1(\Omega;\R^N)$ whose distributional gradient  can be represented by an $\R^N \otimes \R^n$ tensor-valued Radon measure $Du \in \M(\Omega;\R^N \otimes \R^n)$. The theory surrounding this space of functions originated by the work of
\textsc{Caccioppoli} \cite{Caccioppoli_52_I,Caccioppoli_52_II}, \textsc{De Giorgi} \cite{de-giorgi-su-una-teoria-1954,de-giorgi-nuovi-teoremi-1955,de-giorgi-sulla-proprieta-1958,de-giorgi1961frontiere-orien} and \textsc{Federer} \cite{federer-note-on-gauss-green-1958,federer1969geometric-measu}, who studied a particular class of BV functions that consists of characteristic functions ({sets of finite perimeter}). Independently, \textsc{Fleming  \& Rishel}\cite{Fleming_Rishel_60} proved the co-area formula, which cast into the context of sets of finite perimeter evolved into the following well-known identity for gradient measures:
   \begin{equation*}\label{eq:coarea}
    |Du|(B) = \int_{-\infty}^\infty \Hcal^{n-1}(\partial^*\{u > t\} \cap B)\dd t.
    \end{equation*}
    The existence of such a decomposition into a family of $(n-1)$-dimensional rectifiable {sections} is an example of the \emph{structural properties} for BV-functions. It implies, among other things, that the total variation of a gradient measure vanishes on sets of zero $\Hcal^{n-1}$-measure. In fact, it implies the stronger bound $|Du| \ll \Ical^{n-1} \ll \Hcal^{n-1}$. Later on, \textsc{Federer} \cite[\S3.2.14]{federer1969geometric-measu} and \textsc{Vol'pert} \cite{Vol_pert_1967} showed that, for every $u \in \BV(\Omega;\R^M)$, the set $S_u$ of Lebesgue discontinuous points is $\Hcal^{n-1}$-countably rectifiable. 
    They also showed that the measure $Du$ can be decomposed, with respect to the $n$-dimensional Lebesgue measure $\Leb^n$, into a singular part $D^s u$ and an absolutely continuous part $D^a u = \nabla u \, \Leb^n$ with density given by the approximate differential $\nabla u : \Omega \to \R^N \otimes \R^n$ of $u$. Furthermore, the singular part $D^s u$ may be split into a \emph{Cantor part} and a \emph{jump part} as
\begin{equation*}\label{eq:dec}
\begin{split}
D^s u & = D^c u \, + \, D^j u \\
& = D^s u \mres (\Omega \setminus S_u) \, + \, (u^+ - u^-) \otimes \nu_u \, \Hcal^{n-1} \mres J_u,
\end{split}
\end{equation*}
 where the {Cantor part} $D^c u$ is the restriction of  $D^s u$ to the set $\Omega \setminus S_u$ of Lebesgue continuity points of $u$, and the {jump set} $J_u \subset S_u$ is the set of {approximate discontinuity points} $x \in \Omega$ where $u$ has one-sided limits $u^+(x) \neq u^-(x)$ with respect to a suitable direction $\nu_u(x)$ normal to $S_u$.   
Thus, $|Du|$ and each of the terms $D^a u, D^c u, D^ju$ possess particular geometrical properties that correlate directly with the Lebesgue continuity properties of  $u$. These properties ultimately conform the so-called \emph{fine properties} of functions of bounded variation. 

     Another remarkable fact is that $\BV$-functions can be characterized by their decomposition into one-directional sections (see~\cite[Remark~3.104]{AFP2000}). Namely, an integrable function $u$ lies in $\BV(\Omega;\R^N)$ if and only if, for any direction $\xi \in \Sbf^{n-1} = \set{\zeta \in \R^n}{|\zeta|= 1}$ and every coordinate $e \in \R^N$, its one-dimensional sections $u_{y,\xi}^e$ belong to $\BV(\Omega_y^\xi;\R^N)$ for $\Hcal^{n-1}$-almost every $y \in \pi_\xi$, and
    \begin{equation}\label{eq:bvslice}
    \int_{\pi_\xi} |D u_{y,\xi}^e|(\Omega_y^\xi) \dd\Hcal^{n-1}(y) < \infty.
    \end{equation}
    Here, $\Omega_y^\xi = \set{s \in \R}{y + s \xi \in \Omega}$, $\pi_\xi \le \R^n$ is the plane orthogonal to $\xi$ passing through the origin and $u_{y,\xi}^e(t) = u(y + t\xi) \cdot e$. As a matter of fact, the structure theorem extends to each of these one-dimensional sections:
    \begin{equation}\label{eq:BVslice}
    (D^\sigma u : e \otimes \xi) = \int_{\pi_\xi} D^\sigma u_{y,\xi}^e \dd\Hcal^{n-1} \qquad \text{for all $\sigma = a,c,j$.}
    \end{equation}

    Related decompositions and slicing techniques hold for  
    \[
    	\BD(\Omega) = \set{u \in \Lrm^1(\Omega;\R^n)}{Eu = Du + Du^T \in \Mcal(\Omega;(\R^n \otimes \R^n)_\sym)},
    \]
     the space of functions $u : \Omega \subset \R^n \to \R^n$ of \emph{bounded deformation} over $\Omega$. 
     In analogy with~\eqref{eq:bvslice}, the slicing properties of $\BD$ present a significant reduction of coordinates \guillemotleft~$e \in \R^N$ \guillemotright\; that are required to \emph{control} the symmetric gradient. This was observed by \textsc{Bellettini \& Coscia}~\cite{bellettini1992caratterizzazione}, who showed that  function $u : \Omega \subset \R^n\to \R^n$ has bounded deformation if and only if for a basis $\{\xi_1,\dots,\xi_n\}$ of $\R^n$, the one-dimensional sections $u_{y,\xi_i}^{\xi_i}$ belong to $\BV(\Omega_{\xi_i}^{\xi_i})$ for $\Hcal^{n-1}$-almost every $y \in \pi_\xi$ and
    \[
    \int_{\pi_\xi} |D u_{y,\xi_i}^{\xi_i}| \dd \Hcal^{n-1}(y) < \infty. 
    \]
   The full range of fine properties for $\BD$ is contained in the celebrated work of \textsc{Ambrosio, Coscia \& Dal Maso}~\cite{ambrosio1997fine-properties}, where the authors establish a crucial one-dimensional structure theorem of the flavor of~\eqref{eq:BVslice}. More precisely, they showed that if $u \in \BD(\Omega)$, then  
   \begin{equation}\label{eq:BDslice}
    (E^\sigma u : \xi \otimes \xi) = \int_{\pi_\xi} D^\sigma u_{y,\xi}^\xi \dd \Hcal^{n-1}(y) \qquad \text{for all $\sigma = a,c,j.$}
   \end{equation}
   In analogy with $\BV$, here, $Eu = E^a u + E^s u$ is the Radon--Nykod\'ym decomposition of $Eu$ with respect to to $\Leb^d$, $E^c u = E^s u \mres (S_u \setminus J_u)$ is the Cantor part of $E u$, and $E^j u = E^s u \mres J_u$ is the jump part of $Eu$.   Just as for the $\BV$-theory, it is shown that $|Eu|$-almost every point is either an approximate continuity point or an approximate jump point, i.e., 
    \[
    	|Eu|(S_u \setminus J_u) = 0.
    \] 
	Other interesting properties such as {approximate differentiability} are also discussed in~\cite{ambrosio1997fine-properties} (though this was formerly established by \textsc{Haj{\l}asz} ~\cite{hajlasz_on_appr_diff_bd_96} in a more general framework). Gathering these results into one single statement translates into the following structure theorem for functions of bounded deformation: if $u \in \BD(\Omega)$, then one may split $Eu$ into the mutually singular measures
	\[
	Eu = \sym(\nabla u) \, \Leb^n \, + \, E^s u \mres (\Omega \setminus S_u) \, + \, (u^+ - u^-) \odot \nu_u \, \Hcal^{n-1} \mres J_u,
	\] 
	where $\nabla u : \Omega \to \R^n \otimes \R^n$ is the approximate gradient map of $u$, $a \odot b \coloneqq \frac 12 (a \otimes b + b \otimes a)$ for vectors $a,b \in \R^n$, and $J_u$ is indeed an $\Hcal^{n-1}$-countably rectifiable set.

  Having revised the elements behind the slicing and fine properties of $\BV$ and $\BD$, we are now in position to give a brief account of the most important results presented in this work. In order to keep this preliminary exposition accurate and simple, we shall for now focus on the case when $\Acal$ is a first-order operator ($k = 1$).
  Our main assumption on $\Acal : \Crm^\infty(\R^n;\R^N) \to \Crm^\infty(\R^n;\R^M)$ is that it is an \emph{elliptic} operator, i.e., there exists a positive constant $c$ such that
\begin{equation}\label{eq:8}
	|\Abb(\xi)[v]| \ge c|\xi||v| \quad \text{for all $\xi \in \R^n, v \in \R^N$}, \qquad \Abb(\xi) \coloneqq \sum_{j = 1}^n \xi_j A_j.
\end{equation}
	A vector $w \in \R^M$ with $\rank_\Acal$-one satisfies the following identity in Fourier space: there exists a pair $(\xi,e) \in \R^n \times \R^N$ such that
	\[
		w \cdot \Abb(\eta)[v] = (\xi \cdot \eta)(e \cdot v) \qquad \text{for all $\eta \in \R^n$, $v \in \R^N$}.
	\]
  	Any such pair $(\xi,e)$ associated to a $\rank_\Acal$-one tensor is said to belong to \guillemotleft~$\partial\sigma(\Acal)$~\guillemotright\; the \emph{directional spectrum} of $\Acal$. Notice that, for the gradient and symmetric gradient, we have $\partial\sigma(D) = \R^n \times \R^N$ and $\partial\sigma(E) = \set{(\xi,\xi)}{\xi \in \R^n}$ respectively. These are precisely the pairs of directions and coordinates where the one-dimensional slicing holds for each of this operators (compare this with the identities~\eqref{eq:BVslice}-\eqref{eq:BDslice} and the fact that both $e \otimes \xi$, $\xi \otimes \xi$ are \emph{rank-one} tensors). 
	
  	In Theorem~\ref{thm:structure} we show that~\eqref{eq:8} and the mixing algebraic condition
    	\begin{equation}\label{eq:9}
    	\bigcap_{\substack{\pi \, \le \, \R^n\\\dim(\pi) = n-1}} \spn 
    	\set{\im \Abb(\xi)}{\xi \in \pi} = \{0\}
    	\end{equation}
   are equivalent to the following slicing representation of $\BV^\Acal(\Omega)$: a function $u$ belongs to $\BV^\Acal(\Omega)$ if and only if, for every pair $(\xi,e) \in \partial\sigma(\Acal)$, the one-dimensional sections
    \[
    \begin{tikzcd}[row sep=0pt,column sep=1pc]
    	u_{y,\xi}^e \colon \Omega_y^\xi \arrow{r} & \R \\
    	{\hphantom{f\colon{}}} \; \; \; t \arrow[mapsto]{r} &  u(y + t\xi) \cdot e
    \end{tikzcd}
    \] 
    belong to $\BV(\Omega_y^\xi)$ for $\Hcal^{n-1}$-almost every $y \in\pi_y$ and
    \[
    	 \int_{\pi_\xi} |D u_{y,\xi}^e|(\Omega_{\xi}^y) \dd \Hcal^{n-1}(y) < \infty.
    \]
   	Furthermore, in Theorem~\ref{thm:oned} we show that $\Acal u$ satisfies a one-dimensional \emph{structure theorem} in the following sense: if $w$ is a $\rank_\Acal$-one vector with an associated spectral pair $(\xi,e) \in \partial\sigma(\Acal)$, then  
   	\[
   	(w\cdot\Acal^\sigma u) = \int_{\pi_\xi}  D^\sigma u_{y,\xi}^e \dd \Hcal^{n-1}(y), \quad \text{for all } \sigma \in \{a,c,j\}.
   	\]
   	By exploiting this structural property, we are able to give a purely geometric proof that $|\Acal u|$ vanishes when projected on purely unrectifiable $\sigma$-finite $(n-1)$-dimensional sets (see Corollary~\ref{thm:dim}). This is, if $u \in \BV^\Acal(\Omega)$ and $B \subset \Omega$ is a Borel set satisfying 
   	\[
   	\Hcal^{n-1}\big(\pbf_\xi(B)\big) = 0 \quad \text{for $\Hcal^{n-1}$-almost every $\xi \in \Sbf^{n-1}$},
   	\]
   	where $\pbf_\xi : \R^n \to \pi_\xi$ is the canonical orthogonal projection,
   	then $|\Acal u|(B) = 0$ and in particular the Besicovitch--Federer Theorem implies the measure theoretical estimate
   	\[
   	|\Acal u| \ll \Ical^{n-1} \ll \Hcal^{n-1}.
   	\]
   	This accounts for rectifiability estimates of $\Acal$-gradient measures (see also~\cite{GAFA} for a proof of this result that relies on harmonic analysis techniques for elliptic operators).

   	 We also show that, for all $\Acal$ satisfying~\eqref{eq:8}-\eqref{eq:9}, most of the well-known fine properties of the $\BV$-theory extend to $u \in \BV^\Acal$ and its  parts $\Acal^a u, \Acal^c u, \Acal^j u$. Let us give a brief account of the most relevant ones, all of which are contained in Theorem~\ref{thm:2} (for $k=1$) and Theorem~\ref{thm:k} (for arbitrary order $k$);
   	 see also Theorem~\ref{thm:II}. 	Following the same principles of the $\BV$-splitting, one can decompose $\Acal u$, with respect to the $n$-dimensional Lebesgue, into a singular part $\Acal^s u$ and an absolutely continuous part $\Acal^a u = A(\nabla u) \Leb^n$ where $A$ is a linear map acting on the approximate differential $\nabla u$ of $u$ (this follows from a result for elliptic operators of \textsc{Alberti, Bianchini} and \textsc{Crippa}~\cite[Thm.~3.4]{Alberti_Bianchini_Crippa_14} and an observation of \textsc{Raita}~
   	 \cite{Raita_critical_Lp}). A priori, one can trivially split $\Acal^s u$ into a Cantor part, a \emph{diffuse discontinuous} part, and a jump part as
   	 \begin{equation}\label{eq:trivial}
   	 \begin{split}
   	 	\Acal^s u & \phantom{:}= \Acal^{c} u + \Acal^{d} u + \Acal^j u \\
   	 	& \coloneqq \Acal^s u \mres (\Omega \setminus S_u) \, + \, \Acal^s u \mres (S_u \setminus J_u) \, + \, \Acal^s u \mres J_u.
   	\end{split}
   	 \end{equation} 
   	 In~\cite{anna}, the author and \textsc{Skorobogatova} showed that 
   	 $J_u$ is $\Hcal^{n-1}$-countably recitifiable with a measure-theoretic  orientation normal $\nu_u : J_u \to \Sbf^{n-1}$ (see however the more general result in~\cite{nin}), and that the density of the jump part is characterized as
    \[
    \begin{split}
    \Acal^j u & = 
    (u^+ - u^-) \otimes_\Abb \nu_u \, \Hcal^{n-1} \mres J_u, \qquad v \otimes_\Abb \xi \coloneqq \Abb(\xi)[v].
    \end{split}
    \]
    Hence, in the $\BV^\Acal$-setting the representation of the jump part differs only in that the density $(u^+ - u^-) \otimes_\Abb \nu_u$  replaces the classical $\BV$-jump density $(u^+ - u^-) \otimes \nu_u$. This and other \emph{soft} fine properties have been also discussed in~\cite{anna} in a slightly more general framework. 
    
    In the realm of continuity properties and the classical structure theorem, we show that there is \emph{essentially} only one type of discontinuities by proving the estimate 
    \begin{equation}\label{eq:sorryfrancis}
    |\Acal u|(S_u \setminus J_u) = 0.
    \end{equation}
    This implies that only jump-type discontinuities are recorded by the total variation measure $|\Acal u|$. Notice that, then, the diffuse discontinuous part $\Acal^d u$ becomes a superfluous term of~\eqref{eq:trivial}.
    Therefore, from~\eqref{eq:sorryfrancis} we also conclude that $\Acal u$ satisfies the classical structure theorem decomposition
	\begin{equation}\label{eq:str}
    	\Acal u \, = \, A(\nabla u) \,\Leb^n \, + \, \Acal^s u \mres (\Omega \setminus S_u) \, + \,  (u^+ - u^-)  \otimes_\Abb \nu_u \, \Hcal^{n-1} \mres J_u,
    \end{equation}
    where $\nabla u : \Omega \to \R^M \otimes \R^n$ is the approximate differential of $u$ and $A$ is a linear map that is expressed solely in terms of the principal symbol $\Abb$. Parting from~\eqref{eq:sorryfrancis} and following by verbatim the proof of~\cite[Thm.~6.1]{ambrosio1997fine-properties}, one may upgrade~\eqref{eq:sorryfrancis} to the following statement: the set of Lebesgue discontinuity points of $u$ that are not jump points is negligible for all $\Acal$-gradient measures, i.e.,  
    \[
    	|\Acal v|(S_u \setminus J_u) = 0 \quad \text{for all $u,v \in \BV^\Acal(\Omega)$.}
    \]

    
    In the last section, we address the concepts and main results for operators of arbitrary order under the additional assumption of complex-ellipticity. We introduce a linearization principle that allows us to transform the PDE measure-constraint  $\Acal u$ into a first-order PDE constraint $d\Acal$ over $\nabla^{k-1} u$. This method seems to be new (see also~\cite{anna}) in the study of general elliptic operators and might be interesting in its own right.
    
    
    \subsection*{Acknowledgments} I would like to thank Anna Skorobogatova for many valuable conversations. {This project has received funding from the European Research Council (ERC) under the European Union's Horizon 2020 research and innovation programme, grant agreement No 757254 (SINGULARITY).

     \section{Statement of the main results} 	
     In its most general setting, we shall consider a $k$\textsuperscript{th} order homogeneous  linear partial differential operator with constant coefficients operator 
     acting on spaces of smooth functions as  
     $$\Acal : \Crm^\infty(\R^n;V) \to \Crm^\infty(\R^n;W).$$ Here, $V$ and $W$ are finite-dimensional euclidean spaces, of respective dimensions $N$ and $M$  (up to a linear isomorphism the reader may think of $V$ and $W$ as $\R^N$ and $\R^M$ respectively). More precisely, we shall consider operators $\Acal : \Crm^\infty(\Omega;V) \to \Crm^\infty(\Omega;W)$  acting on smooth maps $u \in \Crm^\infty(\Omega;V)$ as
     \begin{equation*}\label{eq:B}
     \Acal u \, = \, \sum_{|\alpha|=k} A_\alpha \partial^\alpha u \, \in \, \Crm^\infty(\Omega;W),
     \end{equation*}
     where the coefficients $A_\alpha \in W \otimes V^* \cong \mathrm{Lin}(V;W)$ are constant tensors. 

    \subsection{Slicing of first-order operators} We commence by discussing the so-called \emph{slicing} theory, which extends the known theory for gradients and symmetric gradients~\cite{ambrosio1997fine-properties,bellettini1992caratterizzazione}.

In all that follows we consider $\Omega \subset \R^n$ to be an open set  with Lipschitz boundary. For a non-zero vector $\xi \in \R^n$, we write $$\pi_\xi = \set{\eta \in \R^n}{\xi \cdot \eta = 0}$$ to denote the orthogonal hyper-plane to $\xi$ passing through the origin.   For a Borel set $B \subset \Omega$ and $y \in \pi_\xi$, we define the one-dimensional slice of $B$ in the $\xi$-direction and passing through $y \in \pi_\xi$ as $$B_y^\xi  = \set{s \in \R}{y + s \xi \in B}.$$ For a given covector $e \in V^*$ and a function $u : \Omega \to V$, we define a function of one variable $u_{y,\xi}^e : \Omega_y^\xi \to \R$ by setting
    $$u^e_{y,\xi}(t) \coloneqq \dpr{ e, u(y + t\xi) }.$$ 
    For stating our results it will be crucial to give an algebraic meaning to those partial derivative operators $\partial_\xi (\;)^e$ that are controlled by $\Acal$:
    \begin{definition}[Directional spectrum]
    	Let $\Acal : \Crm^\infty(\R^n;V) \to \Crm^\infty(\R^n;W)$ be a first-order homogeneous partial differential operator. We define the directional spectrum of $\Acal$ as the set of pairs $(\xi,e) \in \R^n \times V^*$ with the following property: there exists a covector $w \in W^*$ such that
    	\begin{equation}\label{eq:1}
    		\dpr{w , \Abb(\eta)v} = \dpr{\xi ,\eta} \dpr{e,v} \quad \text{for all $\eta \in \R^n$ and all $v \in V$}.
    	\end{equation}
    	We write $\partial\sigma(\Acal)$ to denote the directional spectrum of $\Acal$. 
    \end{definition}

The elements of $\partial\sigma(\Acal)$ have the following slicing property:
    \begin{proposition}\label{lem:1}
    	 Let $\Acal : \Crm^\infty(\R^n;V) \to \Crm^\infty(\R^n;W)$ be a first-order partial differential operator and let $u$ be a function in $\BV^\Acal(\R^n)$. Then, for every $(\xi,e) \in \partial\sigma(\Acal)$ and  $w \in W^*$ satisfying
    	 \begin{equation*}
    	 \dpr{w , \Abb(\eta)v} = \dpr{\xi ,\eta} \dpr{e,v} \quad \text{for all $(\eta,v) \in \R^n \times V$,}
    	 \end{equation*}
    	 it holds that $u_{y,\xi}^e \in \BV(\R)$ for $\Hcal^{n-1}$-almost every $y \in \pi_{\xi}$ and
    	 \[
    	 	\int_{\pi_\xi} |Du_{y,\xi}^e|(\R) \dd \Hcal^{n-1}(y) \le |\Acal u|(\R^n) \cdot |w|.
    	 \]
    	 Moreover,  
    	  \[
    	 \dpr{w ,\Acal u}(B) = \int_{\pi_\xi} Du_{y,\xi}^e(B_y^\xi) \dd \Hcal^{n-1}(y),
    	 \]
    	 \[
    	 		|\dpr{w,\Acal u}|(B) = \int_{\pi_\xi} |Du_{y,\xi}^e|(B_y^\xi) \dd \Hcal^{n-1}(y),
    	 \]
    	 for all Borel sets $B \subset \R^n$.
    \end{proposition}
Notice that the algebraic identity~\eqref{eq:1} is equivalent to requiring that $\dpr{w,\Abb} =\xi \otimes e$, when considering $\Abb$ as a bi-linear form from $\R^n \times V$ to $W$. This motivates us to think of such covectors $w$ as some sort of ``rank-one vectors'' with respect to $\Acal$. To formalize this idea, let us consider the linear map 
\[
\begin{tikzcd}[row sep=0pt,column sep=1pc]
	\bar f_\Acal \colon \R^n \otimes V \arrow{r} & W \\
	{\hphantom{\bar f_\Acal \colon{}}} \eta \otimes v \arrow[mapsto]{r} & \Abb(\eta)[v]
\end{tikzcd}.
\]  
For a covector $w$ in $W^*$, the composition of linear maps $w \circ \bar f_\Acal$ belongs to the tensor space $(\R^n \otimes V)^*$ and hence we may define the following notion of rank:
\begin{definition}[$\rank_\Acal$] Let $w \in W^*$, 
	we define
	\[
		\rank_\Acal(w) \coloneqq \rank(w \circ \bar f_\Acal), 
	\]
	where the latter is the canonical rank acting on $(\R^n \otimes V)^*$. 
\end{definition}

Our first main result is the following $\BV$ one-dimensional representation of  $\BV^\Acal(\Omega)$, which we show to be equivalent to an algebraic mixing condition  imposed on the principal symbol (this condition was introduced by 
\textsc{Spector} and \textsc{Van Schaftingen} \cite{spectorVS} as a {sufficient} condition to establish endpoint $\Lrm^{\frac{n}{n-1},1}$ Lorentz-type estimates for first-order elliptic systems).
    
	\begin{theorem}[Slicing representation]\label{thm:structure} Let $\Acal : \Crm^\infty(\R^n;V) \to \Crm^\infty(\R^n;W)$ be a  first-order homogeneous linear  elliptic differential operator. The following are equivalent:
		\begin{enumerate}\itemsep3pt
			\item[\textnormal{(\emph{1})}]\label{itm:mix} The principal symbol of $\Acal$ satisfies the algebraic mixing property
			\begin{equation}\label{eq:mix}
			\bigcap_{\substack{\pi \, \le \, \R^n\\\dim(\pi) = n-1}} \spn 
			\set{\Abb(\eta)[v]}{\eta \in \pi, v \in V} = \{0_W\}. 
			\tag{$\mathfrak m$}
			\end{equation}
			\item[\textnormal{(\emph{2})}] There exist covectors  $w_1,\dots,w_M$  in $W^*$ such that 
			\[
			\spn\{w_1,\dots,w_M\} = W^*, \qquad \rank_\Acal(w_i) \le 1.
			\]
			In particular, there exist directions $\xi_1,\dots,\xi_M$ in $\R^n$ and covectors $e_1,\dots,e_M$ in $V^*$ with the following property: 
			a function $u$ belongs to $\BV^\Acal(\Omega)$ if and only if, for every $i = 1,\dots,M$, the one-dimensional sections $u_{y,\xi_i}^{e_i}$ satisfy 
			\[
				\text{$u_{y,\xi_i}^{e_i}\in\BV(\Omega_y^{\xi_i})$ for $\Hcal^{n-1}$-almost every $y \in \pi_{\xi_i}$}
			\]
			  and
			\begin{equation*}
				\int_{\pi_{\xi_i}} |Du_{y,\xi_i}^{e_i}|(\Omega_y^{\xi_i}) \dd \Hcal^{n-1}(y) < \infty.
			\end{equation*}
		\end{enumerate}
		\end{theorem}
	

\begin{remark} For $\Omega = \R^n$, the slicing representation theorem also holds if one dispenses with the ellipticity assumption. Effectively, the ellipticity of $\Acal$ is only used to guarantee that the trivial extension by zero $E : \BV^\Acal(\Omega) \to \BV^\Acal(\R^n)$ is a bounded operator. This, in turn, allows for localization arguments up to the boundary.
\end{remark}

\begin{remark}\label{rem:unclear} If
	$\Acal$ is a first-order elliptic operator satisfying the mixing property~\eqref{eq:mix}, then one can show that every coordinate in $\R^n$ (or $V^*$) is the first (or second component) of a spectral pair, i.e.,
	\[
	\proj_{\R^n}\partial\sigma(\Acal)= \R^n, \qquad  \proj_{V^*} \partial\sigma(\Acal) = V^*.
	\]
	A related statement holds for $k$\textsuperscript{th} order elliptic operators satisfying the generalized mixing property~\eqref{eq:mixk} defined below.
\end{remark}
	
	\begin{notation}
		In analogy with $\BV$, we decompose $\Acal u = \Acal^a u + \Acal^s u$, where $\Acal^a u$ is the absolutely continuous part of $\Acal u$ with respect to $\Leb^n$ and $\Acal^s u$ is the singular part of $\Acal u$ with respect to $\Leb^n$. We also define the Cantor and jump parts of $\Acal^s u$ as follows: 
		\begin{align*}
			\Acal^c u \coloneqq \Acal^s u \mres (\Omega \setminus S_u), \qquad \Acal^j u \coloneqq \Acal^s u \mres J_u.
		\end{align*}
	\end{notation}  
			
		Our second main result is a one-dimensional structure theorem for first-order elliptic operators that satisfy the mixing property~\eqref{eq:mix}:
			
		\begin{theorem}[Sectional structure theorem]\label{thm:oned}  Let $\Acal : \Crm^\infty(\R^n;V) \to \Crm^\infty(\R^n;W)$ be a  homogeneous first-order linear elliptic partial differential operator satisfying the mixing condition~\eqref{eq:mix}. Let $(\xi,e) \in \partial\sigma(\Acal)$ and let $w \in W^*$ satisfy
			\[
			\dpr{w , \Abb(\eta)v} = \dpr{\xi , \eta}\dpr{e , v} \quad \text{for all $\eta \in \R^n$ and all $v \in V$}.
			\]
Then,  for every $u \in \BV^\Acal(\Omega)$, we have
	\begin{align*}
		\begin{split}
			\dpr{w , \Acal^\sigma u} & =  \int_{\pi_{\xi}}  D^\sigma u_{y,\xi}^{e} \dd \Hcal^{n-1}(y)\\ 
			|\dpr{w , \Acal^\sigma u}| & =  \int_{\pi_{\xi}}  |D^\sigma u_{y,\xi}^{e}| \dd \Hcal^{n-1}(y) 
			\end{split}\qquad \text{for all $\sigma \in \{a,c,j\}$,}
	\end{align*}
as measures over $\Omega$.
	\end{theorem}

%


\subsection{Rectifiability of $\Acal$-gradient measures} A \emph{weaker} mixing condition, closely related to~\eqref{eq:mix},
was introduced in~\cite{GAFA} (see also~\cite{elementary}) as a {sufficient} condition to establish $\Hcal^{n-1}$-rectifiability of $\Acal$-free measures. 
In our context, the results contained in~\cite{GAFA} imply that $|\Acal u| \ll \Ical^{n-1} \ll \Hcal^{n-1}$ as measures. 
 Theorem~\ref{thm:structure} and Remark~\ref{rem:unclear}
 allow us to give a \emph{purely geometric} proof of the same dimensional and rectifiability results (which we state for operators of arbitrary order):
\begin{corollary}[Rectifiability of $\Acal$-gradients]\label{thm:dim}
	Let $\Acal : \Crm^\infty(\R^n;V) \to \Crm^\infty(\R^n;W)$ be a $k$\textsuperscript{th} order  homogeneous   linear elliptic differential operator satisfying the  mixing condition
	\begin{equation}\label{eq:mixk}
		\bigcap_{{\substack{\pi \, \le \, \R^n\\\dim(\pi) = n-1}}} \spn 
		\set{\Abb^k(\eta)[v]}{\eta \in \pi, v \in V} = \{0_W\}  \tag{$\mathfrak m_k$}.
	\end{equation}
	If $u$ is a function in $\BV^\Acal(\Omega)$ and $B \subset \Omega$ is a Borel set satisfying 
	\[
	\Hcal^{n-1}\big(\pbf_\xi(B)\big) = 0 \quad \text{for $\Hcal^{n-1}$-almost every $\xi \in \Sbf^{n-1}$,}
	\]
	where $\pbf_\xi : \R^n \to \pi_\xi$ is the canonical linear orthogonal projection onto $\pi_\xi$,
	then
	\[
	|\Acal u|(B) = 0.
	\]
	In particular, $|\Acal u| \ll \Ical^{n-1} \ll \Hcal^{n-1}$ as measures and 
	\[
		\dim_\Hcal(|\Acal u|) \ge n-1.
	\]
\end{corollary}

\subsection{Structural and fine properties for first-order elliptic operators} 
    For the subsequent statements it will be necessary to define the set
    \[
    \Theta_u \coloneqq \setBB{x \in \Omega}{ 
    \limsup_{r \to 0^+} \, {r^{-(n-1)}}{|\Acal u|(B_r(x))} > 0}
    \] 
    of points with positive {Hausdorff $(n-1)$-dimensional upper density}.

    We are now in position to state the main {fine properties} result, which is  key towards establishing a {reliable} variational theory for functions of bounded $\Acal$-variation. It establishes that $\Acal$-gradients vanish when projected on purely unrectifiable $\sigma$-finite $(n-1)$-dimensional sets. It also says that, essentially all  approximate discontinuities are jump-type discontinuities and that essentially all Cantor points are Lebesgue continuity points.
    
    \begin{theorem}[Structural and fine properties]\label{thm:2}
    	Let $\Acal : \Crm^\infty(\R^n;V) \to \Crm^\infty(\R^n;W)$ be a first-order homogeneous linear elliptic differential operator satisfying 
    	\[
    					\bigcap_{{\substack{\pi \, \le \, \R^n\\\dim(\pi) = n-1}}} \spn 
    				\set{\Abb(\eta)[v]}{\eta \in \pi, v \in V} = \{0_W\}. 
    	\]
    If  $u \in \BV^\Acal(\Omega)$, then
    	$\Acal u$ decomposes into mutually singular measures as 
    	\begin{align*}
    	\Acal u \, & = \Acal^a  u\, + \, \Acal^c u  \, +  \, \Acal^j u \\
    	& = \Acal^a u \, + \, \Acal^s u \mres (\Omega \setminus S_u) \, + \, \Acal^s u \mres J_u\,, 
    	\end{align*}
    	and the following properties hold:
    	\begin{enumerate}\itemsep5pt\setlength{\itemindent}{-2pt}
    		\item[\textnormal{($i$)}] $\Acal^a u = A(\nabla u)\, \Leb^n$,    		where $\nabla u : \Omega \to V \otimes \R^n$ is the approximate gradient of $u$ and 
    		\[
    		A(P) \coloneqq \sum_{i = 1}^n A_i P[\mathbf e_i] \quad P \in V \otimes \R^n.
    		\]
    		
    		
    		\item[\textnormal{($ii$)}] The jump set $J_u$ is $\Hcal^{n-1}$-countably rectifiable and the jump part is characterized by the identity of measures
    		\[
    		\Acal^j u = \Abb(\nu_u)[u^+-u^-] \, \Hcal^{n-1}\mres J_u,
    		\]	
    		where $(u^+,u^-,\nu_u) : J_u \to V \times V \times \Sbf^{n-1}$ is the triplet Borel map associated to the jump discontinuities on $J_u$ (see Definition~\ref{def:jump}).
    		
    		\item[\textnormal{($iii$)}] The  set $S_u \setminus J_u$ is purely $\Hcal^{n-1}$-unrectifiable, $J_u \subset \Theta_u$, and 
    		$$\Hcal^{n-1}(\Theta_u \setminus J_u) = 0.$$
    		
    		\item[\textnormal{($iv$)}] The Cantor part vanishes on sets that are $\sigma$-finite with respect to $\Hcal^{n-1}$: 
    		\[
    		\text{$E \subset \Omega$ Borel with }\Hcal^{n-1}(E) < \infty \quad \Longrightarrow \quad |\Acal^c u|(E) = 0.
    		\]
    		In particular, $$|\Acal^c u|(\Theta_u) =0.$$
    	
    	\item[\textnormal{($v$)}] The set of Lebesgue discontinuity points that are not jump points  is negligible for all $\Acal$-gradient measures, that is,
    	\[
    	|\Acal v|(S_u \setminus J_u) = 0 \quad \text{for all $v \in \BV^\Acal(\Omega)$.}
    	\]
    	In particular, $$|\Acal u|(S_u \setminus J_u) = 0.$$
    	\end{enumerate}
    	\end{theorem}
    
    \begin{remark}\label{rem:ya} Property~(i) holds for elliptic operators (this follows from a result of \textsc{Alberti, Bianchini} and \textsc{Crippa}~\cite[Thm.~3.4]{Alberti_Bianchini_Crippa_14} and an observation of \textsc{Raita}~
    	\cite{Raita_critical_Lp}). Property~(ii) and the fact that $(S_u \setminus J_u)$ is $\Hcal^{n-1}$-purely unrectifiable also hold for complex-elliptic operators (see~\cite[Thm.~1.2]{anna}).
    \end{remark}

           If we restrict ourselves to operators $\Acal$ acting on scalar functions $u : \Omega \to \R$, we show that $\Acal$ is 
           {equivalent} to a rather simple elliptic operator (see Lemma~\ref{lem:equations}).
    This allows us to state a version of the Structure Theorem for first-order elliptic operators acting on scalar  maps, where we can dispense with any of the mixing assumptions:
    
    \begin{theorem}\label{thm:II} Let $\Acal : \Crm^\infty(\R^n) \to \Crm^\infty(\R^n;V)$ be a first-order homogeneous elliptic partial differential operator and let $u$ be a function in $\BV^\Acal(\Omega)$. Then, $\Acal u$ decomposes into mutually singular measures as 
    	\[
    	\Acal u \, = \, A(\nabla u) \, \Leb^n \, + \, \Acal^c u  \, +  \, \Abb(\nu_u)[u^+-u^-] \, \Hcal^{n-1}\mres J_u,
    	\]
    	and $u$ satisfies the properties (i)-(v) in Theorem~\ref{thm:2}.
    \end{theorem}

    \subsection{Statements for higher-order operators}\label{sec:h}
    For a $k$\textsuperscript{th} order elliptic operator $\Acal$, the classical Calder\'on--Zygmund theory implies the local embedding
    \begin{equation}\label{eq:emb}
    \BV^\Acal_\loc(\Omega) \embed \Wrm_\loc^{k-1,p}(\Omega;V), \qquad 1 \le p < \frac{n}{n-1}.
    \end{equation}
    Since the fine properties of Sobolev spaces are already well-understood (see for instance~\cite[Sec.~4.8]{evans1992}), we shall only focus on the fine properties of the $(k-1)$\textsuperscript{th}-order derivative map $\nabla^{k-1} u \in \Lrm^1(\Omega;V \otimes E_{k-1}(\R^n))$. Here, $E_m(\R^n)$ is the space of $m$\textsuperscript{th} order symmetric tensors on $\R^n$.  In analogy with the first-order case, we shall now give a name to those pairs $(\xi,E) \in \R^n \times (V^*  \otimes E_{k-1}(\R^n))$ such that
    \[
    	\dpr{w,\Acal u} = \partial_\xi (\dpr{E,D^{k-1} u}) \quad \text{for some $w \in W^*$.} 
    \]
    
	\begin{definition}[Tensor spectrum]
			Let $\Acal : \Crm^\infty(\R^n;V) \to \Crm^\infty(\R^n;W)$ be a $k$\textsuperscript{th} order homogeneous linear differential operator. The tensor spectrum of $\Acal$ is defined as the set of pairs $(\xi,E) \in \R^n \times (V^* \otimes E_{k-1}(\R^n))$ with the following property: there exists $w^* \in W$ such that
			\begin{equation}\label{eq:spectrumk}
				\dpr{w, \Abb^k(\eta)v} = \dpr{\xi,\eta}\dpr{E , v \otimes^{k-1} \eta} \quad \text{for all $\eta \in \R^n$, $v \in V$}.
			\end{equation}
			We write $(\xi,E) \in \partial\sigma(\Acal)$.
	\end{definition}

For a higher-order operator, the  $\rank_\Acal$ that we shall consider does not extend directly from the one for first-order operators. As already hinted above, this stems from the fact that we want to understand the slicing properties of $\nabla^{k-1} u$ as opposed to the ones of $u$. The rigorous definition of $\rank_\Acal$ requires of a linearization of the operator $\Acal$, but that will be postponed to Section~\ref{sec:k}. To make sense of the next statements, we say that $\rank_\Acal(w) = 1$ if and only if $w \in W^*$ satisfies~\eqref{eq:spectrumk} for a non-trivial pair $(\xi,E) \in \partial\sigma(\Acal)$.

In order to establish the one-dimensional structure theorem and the fine properties for higher order operators we require $\Acal$ to be complex-elliptic: there exists a positive constant $c > 0$ such that
\[
|\Abb^k(\xi)[v]| \ge c |\xi|^k|v| \quad \text{for all $\xi \in \C^n$ and all $v \in \C \otimes V$}.
\]

We are now in position to state the general version of the slicing theorem:
    
    %
    %

	\begin{theorem}\label{thm:slicek}
			Let $\Acal : \Crm^\infty(\R^n;V) \to \Crm^\infty(\R^n;W)$ be a $k$\textsuperscript{th} order homogeneous linear differential complex-elliptic  operator. The following are equivalent:
			\begin{enumerate}
				\item[\textnormal{(\emph{1})}] The principal symbol of $\Acal$ satisfies the mixing condition
				\[
					\bigcap_{{\substack{\pi \, \le \, \R^n\\\dim(\pi) = n-1}}} \spn\set{\Abb^k(\eta)[v]}{\eta \in \pi, v \in V} = \{0_W\}.
				\]
				\item[\textnormal{(\emph{2})}] There exist covectors $w_1,\dots,w_M \in W^*$ such that 
				\[
					\spn\{w_1,\dots,w_M\} = W^*, \qquad \rank_\Acal(w_i) \le 1.
				\]
				In particular, there exist directions $\xi_1,\dots,\xi_M$ in $\R^n$ and tensors $E_1,\dots,E_M$ in $V^* \otimes E_{k-1}(\R^n)$ with the following property: A function $u$ belongs to  $\BV^\Acal(\Omega)$ if and only if $U \coloneqq \nabla^{k-1}u \in \Wrm^{1,1}(\Omega)$ and, for every $i = 1,\dots,M$, \[
				 \text{$U_{y,\xi_i}^{E_i}\in\BV(\Omega_y^{\xi_i})$ for $\Hcal^{n-1}$-almost every $y \in \pi_{\xi_i}$}
				\] and
				\[
					\int_{\pi_{\xi_i}} |DU_{y,\xi_i}^{E_i}|(\Omega_y^{\xi_i}) \dd \Hcal^{n-1}(y) < \infty.
				\]
			\end{enumerate}
	\end{theorem}

    Lastly, we state the higher-order version of the fine properties Theorem:
    \begin{theorem} \label{thm:k} 	Let $\Acal : \Crm^\infty(\R^n;V) \to \Crm^\infty(\R^n;W)$ be a  $k$\textsuperscript{th} order homogeneous linear complex-elliptic differential operator satisfying 
    	\[
    	\bigcap_{{\substack{\pi \, \le \, \R^n\\\dim(\pi) = n-1}}} \spn\set{\Abb^k(\eta)[v]}{\eta \in \pi, v \in V} = \{0_W\}.
    	\]
If  $u \in \BV^\Acal(\Omega)$ and $U \coloneqq \nabla^{k-1} u$, then
$\Acal u$ decomposes into mutually singular measures as 
\begin{align*}
\Acal u \, & \phantom{:}= \Acal^a  u\, + \, \Acal^c u  \, +  \, \Acal^j u \\
& \coloneqq \Acal^a u \, + \, \Acal^s u \mres (\Omega \setminus S_U) \, + \, \Acal^s u \mres J_U 
\end{align*}
and the following properties hold:
\begin{enumerate}\itemsep5pt\setlength{\itemindent}{-3pt}
	\item[\textnormal{($i$)}] $\Acal^a u = A(\nabla^k u)\, \Leb^n$,    		where $\nabla^k u : \Omega \to V \otimes E_{k}(\R^n)$ is the approximate gradient of $U$ and 
	\[
	A(F) \coloneqq \sum_{|\alpha| = k} \frac{1}{\alpha!}A_\alpha \big[\dpr{F,\mathbf e_1^{\alpha_1} \odot \cdots \odot \mathbf e_n^{\alpha_n}}\big], \qquad F \in V \otimes E_{k}(\R^n).
	\]
	
	
	\item[\textnormal{($ii$)}] The jump set $J_{U} \subset \Theta_u$ is countably $\Hcal^{n-1}$-rectifiable and the jump part is characterized by the identity of measures
	\[
	\Acal^j u = \Abb^k(\nu_u)\llbracket \nabla^{k-1} u \rrbracket \, \Hcal^{n-1}\mres J_{U},
	\]	
	where $$\llbracket F \rrbracket \coloneqq \dpr{F^+ - F^+ , \otimes^{k-1} \nu_u}.$$
	\item[\textnormal{($iii$)}] The Cantor part $\Acal^c u$ vanishes on sets that are $\sigma$-finite with respect to $\Hcal^{n-1}$, that is,
	\[
	\text{$E \subset \Omega$ Borel with }\Hcal^{n-1}(E) < \infty \quad \Longrightarrow \quad |\Acal^c u|(E) = 0.
	\]
	\item[\textnormal{($iv$)}] The  set $S_{U} \setminus \Theta_u$ is purely $\Hcal^{n-1}$-unrectifiable and 
	$$\Hcal^{n-1}(\Theta_u \setminus J_{U}) = 0.$$
	\item[\textnormal{($v$)}] The set of Lebesgue discontinuity points that are not jump points  is negligible for all $\Acal$-gradient measures, that is,
	\[
	|\Acal v|(S_{U} \setminus J_{U}) = 0 \quad \text{for all $v \in \BV^\Acal(\Omega)$.}
	\]
\end{enumerate}
\end{theorem}

   \section{Preliminaries}  

\subsection{Disintegration into one-dimensional sections} We begin by recalling a basic concept of the slicing theory for Radon measures. Let $\Omega \subset \R^n$ be an open set and let $\xi \in \R^n$ be a non-zero vector. Suppose that for $\Hcal^{n-1}$-almost every $$y \in \Omega^\xi \coloneqq \set{z \in \pi_\xi}{\Omega_z^\xi \neq \emptyset},$$ we are given a measure $\mu_y$ in $\Mcal(\Omega_y^\xi)$. Further assume that, for every $\phi \in \Crm(\R)$, the assignment
\[
y \mapsto \int_{\Omega^\xi_y} \phi(t) \dd \mu_y(t)	
\] 
is well-defined and Borel measurable on $\Omega^\xi$. Lastly, assume that
\[
\int_{\Omega^\xi} |\mu_y|(\Omega_y^\xi) \dd \Hcal^{n-1}(y) < \infty.
\]
Then, for every Borel set $B \subset \Omega$, the assignment $y \mapsto \mu_y(B_y^\xi)$ is Borel measurable on $\Omega_y^\xi$ and the set function $\lambda : \Bfrak(\Omega) \to \R$ defined as 
\[
\lambda(B) \coloneqq \int_{\Omega^\xi} \mu_y(B_y^\xi) \dd \Hcal^{n-1} \quad \text{for all Borel sets $B \subset \Omega$},
\]
is a bounded Radon measure on $\Omega$, which we shall denote as
\[
\lambda = \int_{\Omega^\xi} \mu_y \dd \Hcal^{n-1}(y).
\]
In this case we say that $\lambda$ is disintegrated by $\Hcal^{n-1}$ into the one-dimensional sections $\mu_y$. By a density argument, it is easy to show that the total variation measure $|\lambda|$ coincides with the measure
\[
	\int_{\Omega^\xi} |\mu_y| \dd \Hcal^{n-1}.
\]

\subsection{The jump set and approximate continuity}Next, we give rigorous definitions and introduce notation belonging to the classical \emph{fine properties} theory. 
We begin by recalling the formal definitions of the {approximate jump set} and points of {approximate continuity}.
\begin{definition}[Approximate jump\label{def:jump}] Let $u \in \Lrm^1_\loc(\Omega;\R^M)$.
	We say that a point $x$ is an \emph{approximate jump point} of $u$ ($x \in J_u$) if there exist {distinct} vectors $a,b \in \R^M$ and a direction $\nu \in \Sbf^{n-1}$ satisfying 
	\begin{equation}\label{eq:jumps}
	\begin{cases} \displaystyle
	\lim_{r \todown 0} \aveint{B^+_r(x,\nu)}{} |u(y) - a| \dd y = 0,\\[12pt]
	\displaystyle
	\lim_{r \todown 0} \aveint{B^-_r(x,\nu)}{} |u(y) - b| \dd y = 0,
	\end{cases}
	\end{equation}
	where
$B^\pm_r(x,\nu) \coloneqq \set{y \in B_r(x)}{\pm\dpr{\nu,y} > 0}$ are 
the $\nu${-oriented half-balls} centered at $x$, where $B_r(x)$ is the open unit ball of radius $r > 0$ and centered at $x$. 
\end{definition}
	We refer to $a,b$ as the {one-sided limits} of $u$ at $x$ with respect to the orientation $\nu$. 
	Since the jump triplet $(a,b,\nu)$ is well-defined up to a sign in $\nu$ and a permutation of $(a,b)$, we shall write $(u^+,u^-,\nu_u) : J_u \to \R^M \times \R^M \times \Sbf^{n-1}$ to denote the triplet Borel map associated to the jump discontinuities on $J_u$, i.e., $x \in J_u$ if and only if~\eqref{eq:jumps} holds with $(a,b,\nu) = (u^+(x),u^-(x),\nu_u(x))$.
	
We now define what it means for a locally integrable function to be approximately continuous at a given point:
\begin{definition}[Approximate continuity]\label{def:appcont}
	Let $u \in \Lrm^1_\loc(\Omega;\R^M)$ and let $x \in \Omega$. We say that $u$ has an \emph{approximate limit} $z \in \R^M$ at $x$ if
	\[
	\lim_{r \todown 0} \aveint{B_r(x)}{} |u(y) - z| \dd y = 0.
	\]
	The set of points $S_u \subset \Omega$ is called the \emph{approximate discontinuity set}. 
\end{definition}

    \subsection{Symbolic calculus}\label{sec:preliminary_rank}In stating our results it will be fundamental to recall a couple of basic definitions for partial differential operators. For a Schwartz function $u \in \Scal(\R^n;V)$, the Fourier transform applied to $\Acal u$ gives 
    \[
    \widehat {\Acal u}(\xi) = (2\pi \mathrm i)^k \Abb^k(\xi)[\widehat u(\xi)], \qquad \Abb^k(\xi) \coloneqq \sum_{i=1} \xi^\alpha A_\alpha, \quad \xi^\alpha \coloneqq \xi_1^{\alpha_1} \cdot \cdots \cdot \xi_{n}^{\alpha_n}.
    \]
    The homogeneous polynomial $\Abb^k: \R^n \to W\otimes V^*$ is called the \emph{principal symbol} associated to $\Acal$ (when $k = 1$ we shall simply write $\Abb = \Abb^1$). 
    The \emph{image cone} of $\Acal$ (which contains all $\Acal$-gradients in Fourier space) and the \emph{essential range}  of $\Acal$ are respectively defined as  
    \[
    \Irm_{\Acal} \coloneqq \bigcup_{\xi \in \R^n} \im \Abb(\xi), \qquad W_\Acal \coloneqq \spn \Irm_\Abb \subset W.
    \]
     A simple Fourier transform argument (e.g., Sec.~2.5 in~\cite{arroyo2018lower}) shows that
    \[
    \Acal u(x) \in W_\Acal \quad \text{for all $u \in \Crm^\infty_c(\Omega;V)$.}
    \]
    Thus, by a standard approximation argument (e.g., Theorem~1.3 in~\cite{YMA}) we find that $\BV^\Acal(\Omega) \subset \Lrm^1(\Omega;W_\Acal)$. 
    One of the main structural assumptions on $\Acal$ will be to assume that it is an elliptic operator in the following sense: 
    \begin{definition}[Ellipticity]\label{def:ell} We say that an operator $\Acal$ as in~\eqref{eq:B} is elliptic if there exists a positive constant $c$ such that
    	\[
    	|\Abb^k(\xi)[v]| \ge c|\xi|^k|v| \quad \text{for all $(\xi,v) \in \R^n \times V$.}
    	\]
    \end{definition} 

    \subsection{Traces of complex elliptic operators}\label{sec:traces}The concept of complex-ellipticity was originally introduced by \textsc{Smith}~\cite{smith1,smith2}.  Recently, this concept has been re-visited by \textsc{Breit, Diening \& Gmeineder}~\cite{breit2017traces}, who have shown that complex-ellipticity is a necessary and {sufficient} condition for the existence of  trace operators on $\BV^\Acal$-spaces when $\Acal$ is a first-order operator. The precise definition is the following: 
    
    
    \begin{definition}We say that an operator $\Acal$ as in~\eqref{eq:B} is complex-elliptic  when the complexification of the principal symbol map $\Abb^k$ is injective, i.e., if there exists a positive constant $c$ such that
    	\[
    	|\Abb^k(\xi)v| \ge c |\xi||v| \qquad \text{for all $\xi \in \C^n$ and all $v \in \C\otimes V$.}
    	\]
    \end{definition}
    \begin{remark}For first-order elliptic operators, the mixing property is a sufficient condition for complex-ellipticity, i.e.,  
    \begin{equation}\label{eq:implication}
    \eqref{eq:mix} + \text{elliptic} \quad 
    \Longrightarrow \quad \text{complex-elliptic}.
    \end{equation}
However, complex-ellipticity is \emph{not} a sufficient condition for an operator to  satisfy the mixing condition (see Examples~\ref{ex:dev} and~\ref{ex:cauchy}). 
    \end{remark}
    We recall from~\cite[Corollary~4.21]{breit2017traces} the following trace properties for $u \in \BV^\Acal(\Omega)$ when $\Acal$ is complex-elliptic: If $\Omega$ is a Lipschitz domain, then there exists a continuous linear trace operator $\tr : \BV^\Acal(\Omega) \to  \Lrm^1(\partial\Omega;\Hcal^{n-1}),$
    satisfying $u = \tr u$ for all $u \in \Crm(\cl\Omega;V) \cap \BV^\Acal(\Omega)$. In particular,  the  extension by zero
    \[
    \bar u(x) = \begin{cases} u(x) & \text{if $x \in \Omega$},\\
    	0 & \text{if $x \in \R^n \setminus \Omega$}
    \end{cases}
    \]
    belongs to $\BV^\Acal(\R^n)$ and satisfies $\Acal u \mres \partial \Omega = \Abb(-\nu_\Omega)[\tr(u)] \, \Hcal^{n-1} \mres \partial \Omega$,
    where $\nu_\Omega$ is the outer unit normal of $\Omega$.   
\subsubsection{Characterization of elliptic equations}\label{sec:equations} For first-order equations in divergence form, there are no difference among the concepts of ellipticity, complex-ellipticity and the mixing property:
    \begin{proposition}[First-order equations]\label{ex:div_form} Let $n\ge 2$ and  let $R \in \R^M \otimes \R^n$. Consider the operator
    	\begin{align*}
    	\Acal_R u & \coloneqq \Div(Ru) = \bigg(\sum_{j=1}^n R_{ij} \partial_j u \bigg)_i \qquad i = 1,\dots,M\,,
    	\end{align*}
    	defined on scalar maps $u : \R^n \to \R$. The following are equivalent:
    	\begin{enumerate}
    		\item[(a)] $\rank(R) \ge n$,
    		\item[(b)] $\Acal_R$ is elliptic,
    		\item[(c)] $\Acal_R$ is complex-elliptic,
    		\item[(d)] $\Acal_R$ satisfies the the mixing condition~\eqref{eq:mix}.
    	\end{enumerate}
    \end{proposition}
    \begin{proof}
    	We shall see that (a)~$\Leftrightarrow$~(b),(c) and $\neg$(a)~$\Leftrightarrow$~$\neg$(d). Ellipticity is equivalent to the principal symbol $\Abb_R(\xi)$ being injective for all $\xi \in \R^n$, which is equivalent to $|R \cdot \xi| > 0$ for all non-zero $\xi \in \R^n$; this shows that (a)~$\Leftrightarrow$~(b). However, since $R$ is a tensor with real coefficients the same holds for all $\xi \in \C^n$; this shows (a)~$\Leftrightarrow$~(c). Lastly, (d) fails if and only if there exists a non-zero $\xi \in \R^n$ such that $R\xi \in R[\eta^\perp]$ for all $\eta \in \R^n$. However, this is equivalent to $R$ not being one-to-one, or equivalently, that $\rank(R) < n$. 
    \end{proof}

    We can now use the above result to show that, for first-order equations, ellipticity, complex-ellipticity, and the mixing condition are all equivalent:
    \begin{lemma}\label{lem:equations}
    	Let $\Acal : \Crm^\infty(\R^n) \to \Crm^\infty(\R^n;\R^M)$ be a first-order elliptic operator. 
    	There exists an full-rank tensor $R \in \R^n \otimes \R^n$ such that the operator
    	\begin{align*}
    	\Acal_R u & \coloneqq \Div(Ru), \quad u : \R^n \to \R, 
    	\end{align*}
    	is complex-elliptic, satisfies the mixing property, and 
    	\[
    	\|\Acal_R u\|_{\Lrm^1} \le \|\Acal u\|_{\Lrm^1} \le c\, \|\Acal_R u\|_{\Lrm^1}
    	\] 
    	for all $u \in \Crm_c^\infty(\R^n)$.
    \end{lemma}
    \begin{proof}
    	Since $\Acal$ is of of first order, the characteristic polynomial $p_j(\xi)$ corresponding to each scalar operator $P_j \coloneqq \mathbf e_j \cdot \Acal$ must be of the form $p_j(\xi) = \eta_j \cdot \xi$ for some $\eta_j \in \R^n$. Therefore, the complex-ellipticity of $\Acal$ is equivalent to the family $\{\eta_1,\dots,\eta_M\}$ possessing a basis of an $n$-dimensional subspace of $\R^M$. Hence, we may find a permutation $\sigma \in S_M$ such that
    	\[
    	E_\sigma \coloneqq \{\eta_{\sigma(1)},\dots,\eta_{\sigma(n)}\} \;  \qquad\text{spans an $n$-dimensional space.}
    	\]  
    	The matrix $R_{\ell i} \coloneqq [\eta_{\sigma(\ell)}]_i$ has rank $n$ and therefore the operator defined by
    	\begin{align*}
    	[\Acal_R u]_\ell \coloneqq \sum_{i = 1}^n R_{\ell i} \partial_i u   = [\Div(Ru)]_j, \qquad \ell = 1,...,n, \quad u \in \Crm^\infty_c(\R^n),
    	\end{align*}
    	is complex-elliptic. The $\Lrm^1$-bound $
    	\|\Acal_R u\|_{\Lrm^1} \le \|\Acal u\|_{\Lrm^1}$
    	follows immediately from the fact that $\Acal_R = p \circ \Acal$ where $p : \R^M \to \spn E_\sigma$ is the orthogonal projection onto $\spn E_\sigma$. It remains to show that we can again estimate $\Acal u$ by $\Acal_R u$ in~$\Lrm^1$. Since $E_\sigma$ forms a basis of $\R^n$, for each scalar operator $P_k$ we can write
    	\[
    	P_k u = \sum_{\ell=1}^n c^k_\ell \eta_{\sigma(\ell)} \cdot \nabla u = \sum_{i,\ell = 1}^n c^k_\ell R_{\ell i} \partial_{i} u = \sum_\ell c^k_\ell [\Acal_R u]_\ell
    	\]
    	for some choice of constants $c^k_\ell$. From this, we immediately obtain the desired estimate $\|\Acal u\|_{\Lrm^1} \le c\|\Acal_R u\|_{\Lrm^1}$.
    \end{proof}

\section{Slicing theory for first-order operators}\label{sct:slicing}The purpose of this section is to study first-order elliptic operators $\Acal$ that allow for the measure $\Acal u \in \Mcal(\Omega;W)$ to be disintegrated into one-dimensional sectional derivatives of its potential $u$. 
%
We aim to show that if a function $u$ belongs to $\BV^\Acal(\Omega)$, then there exist directions $\{\xi_1,\dots,\xi_r\} \subset \Sbf^{n-1}$ and covectors $\{e_1,\dots,e_r\} \subset V^*$ such that 
\[
u_{y,\xi_i}^{e_i} \in \BV(\Omega_{y}^{\xi_i}) \; \text{for $\Hcal^{n-1}$-almost every $y \in \Omega^{\xi_i}$}
\]
for every $i = 1,\dots, r$, and 
\begin{equation}\label{eq:slicing}
\Acal u =  \sum_{i=1}^r \, \bigg( \int_{\pi_{\xi_i}} Du_{y,\xi_i}^{e_i} \dd \Hcal^{n-1}(y) \bigg) P_i,
\end{equation}
for some family of vectors $\{P_1,...,P_r\} \in W$.
Let us begin by making the following observation about the  inherent relationship between the principal symbol of $\Acal$ and the slicing of $\Acal u$: Recall that the principal symbol $\Abb$ induces the linear map
\begin{align*}
\bar f_\Acal : \R^n &\otimes V \to W, \\
\qquad (\eta \otimes v) &\mapsto \Abb(\eta)[v].
\end{align*}
Let us assume for a moment that $\Omega = \R^n$ and that, for some covector $w \in W^*$, the composition map $w \circ \bar f_\Acal$ may be represented by a rank-one tensor (viewed as an element of $(\R^n \otimes V)^*$):
\begin{equation}\label{eq:tensor}
w \circ \bar f_\Acal  = \xi \otimes e \quad \text{for some $\xi \in \R^n$ and $e \in V^*$.}
\end{equation}
Here we have identified $\R^n$ with $(\R^n)^*$ by the canonical isomorphism.
It follows that
\[
\dpr{w,\bar f_\Acal(\eta \otimes v)}  = \dpr{\xi , \eta}\dpr{e , v} \qquad \text{for all $(\eta,v) \in \R^n \times V$.}
\]
Now, let $u \in \Scal(\R^n;V)$. We recall the following idea contained in the proof of Theorem~5 in~\cite{spectorVS}: applying the Fourier transform to the identity above we deduce the point-wise identity
\[
    \dpr{w,\Abb(\eta)[\widehat u(\eta)]}  = \dpr{\xi, \eta}\dpr{e , \widehat u(\eta)} \quad \text{for all $\eta \in \R^n$}.
\] 
Inverting the Fourier transform we discover that $\dpr{w,\Acal u} = \partial_\xi u^e$. We shall see later that, by using the differential $\pi_\xi$-independence of the right hand-side, it is relatively simple to show that
\begin{equation}\label{eq:1slice}
\dpr{w,\Acal u} = \int_{\pi_\xi} Du_{y,\eta}^e \dd \Hcal^{n-1}(y) \qquad \text{as measures in $\Mcal(\R^n)$}.
\end{equation}
The key feature of this this equality of measures is that the left hand-side can be estimated by $|\Acal u|$. This, in turn, allows one to apply standard smooth approximation methods to extend it to all functions $u \in \BV^\Acal(\R^n)$. In fact, it is easy to verify  that~\eqref{eq:tensor} is not only a sufficient, but a necessary condition for~\eqref{eq:1slice} to hold on arbitrary functions $u \in \BV^\Acal(\R^n)$.

\subsection{The rank-one property}\label{sec:rank} Motivated by the previous discussion, we next study those elements $w \in W^*$ for which~\eqref{eq:1slice} holds ---the $\rank_\Acal$-one vectors.
It will also be convenient to give a name to the cone of $\rank_\Acal$-$m$ covectors: for an integer $0 \le m \le \min\{n,\dim(V)\}$, we write
	$$\Acal^\otimes_m \coloneqq \set{w \in W^*}{\rank_\Acal(w) = m}$$
to denote the cone of $\rank_\Acal$-$m$ vectors.

\begin{definition}\label{def:spectrum}
	We say that $\Acal$ has the {rank-one property} if and only if
    \begin{equation*}
     \spn\Acal^\otimes_1 = (W_\Acal)^*,
    \end{equation*}
or equivalently,
\[
	\spn \{\Acal^\otimes_0 \cup \Acal^\otimes_1\} = W^*.
\]
\end{definition}
	
The following lemma shows that the rank-one property is  equivalent to the mixing condition introduced in the introduction.

\begin{lemma}\label{lem:slicemix}
The following are equivalent:
\begin{enumerate}
		\item[\textnormal{(\emph{1})}] $\Acal$ satisfies the rank-one property,
	\item[\textnormal{(\emph{2})}] $\Acal$ satisfies the mixing condition
	\[
	\bigcap_{\xi \in \Sbf^{n-1}} \spn \set{\im {\Abb}(\eta)}{\eta \in \pi_\xi} = \{0_W\}.
	\]
\end{enumerate}
\end{lemma}
\begin{proof}
		First, we show that (1)$\Rightarrow$(2). Let us first fix a $P \in (W_\Acal)^*$ with $\rank_{\Acal}(P) = 1$. By definition there exist vectors $\xi \in \R^n$ and $e \in V$ such that
	\[
	\dpr{P, \Abb(\eta)v} = \dpr{\xi,\eta}\dpr{e,v} \quad \forall\; (\eta,v) \in \R^n \times V. 
	\]
	In particular,
	\begin{equation}\label{eq:rank_perp}
	P\in \Bigg(\sum_{\eta \in \pi_\xi}\im {\Abb}(\eta)\Bigg)^\perp.
	\end{equation}
	Now, by assumption, we may find a family $\{P_j\}_{j=1}^r \subset W \cap \{\rank_\Acal = 1\}$ spanning $(W_\Acal)^*$. Write $\pi_j = \pi_{\xi_j} \in \Gr(n-1,n)$ to denote the hyper-plane for which~\eqref{eq:rank_perp} holds with $P = P_j$. Next, consider 
	\[
\displaystyle 	Q \in \bigcap_{\zeta \in \Sbf^{n-1}} \, \spn\set{\im \Abb(\eta)}{\eta \in \pi_\zeta} \subset W_\Acal. 
	\]
	Since $Q \in \spn_{\eta \in \pi_j} \{\im \Abb(\eta)\}$ for all $j = \{1,\dots,r\}$, we conclude from~\eqref{eq:rank_perp} with $P=P_j$ that $\dpr{P_j,Q} = 0$ for all $j =  \{1,\dots,r\}$. Recalling that $\{P_j\}_{j=1}^r$ spans ${W_\Acal}^*$, we conclude that  $Q$ must be the zero vector and (2) follows.
	
	We now show (2)$\Rightarrow$(1). Fix  a direction $\xi \in \Sbf^{n-1}$ and notice that
	\begin{equation}\label{eq:dual}
	   P \in \bigcap_{\eta \in \pi_\xi} \ker \Abb(\eta)^* \quad \Rightarrow \quad \rank_\Acal(P) \le 1.
	\end{equation}
	Here, $\Abb(\xi)^*$ is the adjoint operator of $\Abb(\xi)$.
	Indeed, $\dpr{Q,\Abb(\eta)v} = \dpr{\Abb(\eta)^*Q,v} = 0$
	for all $\eta \in \pi_\xi$ and all $v \in V$. Equivalently, $\dpr{Q,\Abb(\eta)v} = \dpr{\xi , \eta}\dpr{\Abb(\eta)^*Q,v}$ for all $\eta \in \pi_\xi$ and $v \in V$. The claim then follows by taking $e^* = \Abb(\eta)^*Q$ and observing that $Q \circ f_\Acal = \xi^* \otimes e^*$.
   Now, by an application of De Morgan's laws (for orthogonal complements) we obtain
	\[
		W^* = \spn\set{\displaystyle\bigcap_{\eta \in \pi_\xi} \ker \Abb(\eta)^*}{\xi \in \Sbf^{n-1}}. 
	\]
	The sufficiency then follows from~\eqref{eq:dual}.
\end{proof}

\subsection{Proof of the sectional representation theorem}

We begin this section with the proof of Proposition~\ref{lem:1}:

\begin{proof}[Proof of Proposition~\ref{lem:1}] By assumption we can find $w^* \in W^*$ such that
\begin{equation}\label{eq:spectrumguy}
	\dpr{w^*, \Abb(\eta)\widehat{\phi}(\eta)} = \dpr{\xi , \eta}\dpr{e, \widehat{\phi}(\eta)} \qquad \text{for every $\eta \in \R^n$},
\end{equation}
and all $\phi \in \Crm_c^\infty(\R^n;V)$.
Inverting the Fourier transform, we find that this is equivalent to the pointwise identity
$\dpr{w^*, \Acal \phi(x)} = \partial_\xi \phi^e(x)$ for all $x \in \R$.
This identity  implies that $w^* \circ \Acal = \partial_\xi (\;\;)^e$ {as distributional differential operators};
we shall  recall this identity as the proof develops.


Let $\rho_\eps \in \Crm_c^\infty$ be a standard mollifier at scale $\eps > 0$ and let $u_\eps \coloneqq u \star \rho_\eps$. By Fubini's Theorem and a change of variables, we have
\begin{align*}
	\int \dpr{w^*,\Acal u_\eps(x)} \dd x &= \int \partial_\xi (u_\eps)^e(x) \dd x \\
	&= \int_{\pi_\xi}\bigg( \int_{\R} D (u_\eps)^e_{y,\xi}(t)\dd t \bigg)\dd \Hcal^{n-1}(y). \\
\end{align*}
Clearly we have
\[
\int \dpr{w^*,\Acal u_\eps(x)} \dd x \longrightarrow \dpr{w^*,\Acal u}(\R^n) \qquad \text{as $\eps \to 0^+$}.
\]
From Young's inequality and the identity above we get the uniform bound
\[
|w^*||\Acal u|(\R^N) \ge |\dpr{w^*,\Acal u_\eps \, \Leb^n}|(\R^n) \ge \int_{\pi_\xi} V[(u_\eps)_{\xi,y}^e] \dd \Hcal^{n-1}(y),
\]
where, for $h:\R \to \R$, $Vh$ is the pointwise variation of $h$. 

By  Arguing as in the proof of \cite[Proposition 3.2]{ambrosio1997fine-properties}, passing to the limit $\eps \to 0^+$ we deuce  (from the lower semicontinuity properties of the pointwise variation, Fatou's lemma, and standard measure theoretic arguments) that

\[
|w^*||\Acal u|(\R^n) \ge \int_{\pi_\xi}  V[\tilde u^e_{\xi,y}](\R) \dd \Hcal^{n-1}(y) =  \int_{\pi_\xi}  |D u^e_{\xi,y}|(\R) \dd \Hcal^{n-1}(y),
\]
where \guillemotleft~$\tilde{\frarg}$~\guillemotright~denotes the Lebesgue representative of a locally integrable function.
This shows that if $u \in \BV^\Acal(\R^n)$, then $u^e_{\xi,y} \in \BV(\R)$ for $\Hcal^{n-1}$-almost every $y \in \pi_\xi$.  Now, let $\phi \in \Crm_c^\infty(\R^n)$ be an arbitrary test function. Using that the identity $\dpr{w^*,\Acal u} = \partial_\xi u^e$  in the sense of distributions,
we get
\begin{align*}
	\int \phi(x) \dd [\dpr{w^*,\Acal u}](x) & = \int \phi \dd [\partial_\xi u^e](x) \\
	& =  -\int \partial_\xi \phi_{\xi,y}(x) \, u^e(x) \dd x.
\end{align*}
Notice that $\partial_\xi \phi(y + t\xi) = D(\phi_{\xi,y})(t)$. Hence, by Fubini's theorem and a change of variables we further deduce 
\begin{align*}
	\dpr{w^*,\Acal u}(\phi) & = - \int_{\pi_\xi} \bigg(\int_\R D(\phi_{\xi,y})(t) \, u^e_{\xi,y}(t) \dd t \bigg) \dd \Hcal^{n-1}(y) \\
	& = \int_{\pi_\xi} \bigg( \int_\R\phi_{\xi,y}(t) \dd [Du_{y,\xi}^e](t)\bigg) \dd \Hcal^{n-1}(y),
\end{align*}
where in the last equality we have used that $u_{y,\xi}^e \in \BV(\R)$ for $\Hcal^{n-1}$-almost every $y \in \pi_\xi$. Herewith, a standard density argument implies the equalities of measures
\[
\dpr{w^*,\Acal u}(B) = \int_{\pi_\xi} Du_{y,\xi}^e(B_y^\xi)  \dd \Hcal^{n-1}(y),
\]
\[
|\dpr{w^*,\Acal u}|(B) = \int_{\pi_\xi} |Du_{y,\xi}^e|(B_y^\xi)  \dd \Hcal^{n-1}(y)
\]
for all Borel sets $B \in \R^n$.
\end{proof}

We are now ready to give the proof of slicing theorem:

\begin{proof}[Proof of Theorem~\ref{thm:structure}]
	First, observe that due to the nature of the boundary $\partial\Omega$, one can extend $u$ by zero to a function in $\BV^\Acal(\R^n)$ (see Section~\ref{sec:traces}). Therefore, there is no loss of generality in assuming that $\Omega = \R^n$. Also, by a localization argument we may assume that the support of $u$ is compact. Now, let us mollify $u$ at scale $\eps > 0$ ---denote this regularization by $u_\eps \in \Crm_c^\infty(\R^n;V)$.



%


Let us prove that (1) $\Rightarrow$ (2). By Lemma~\ref{lem:slicemix} we may suppose that $\Acal$ satisfies the rank-one property.  This means that we can find a family of $\rank_\Acal$-one covectors $\{P_1,\dots,P_r\}$ spanning $(W_\Acal)^*$ and such that
\[
\dpr{P_i , \Abb(\eta)v} = \dpr{\xi_i ,\eta}\dpr{e_i,v} \quad \text{for some } (\xi_i, e_i) \in \partial\sigma(\Acal),
\]
for all $i = 1,\dots,r$. We may complete this to a basis $\{P_1,\dots,P_r,P_{r+1},\dots,P_M\}$ of $W^*$ with $\rank_\Acal(P_i) = 0$ for all $i = r+1,\dots,M$. In a natural way, this basis induces a canonical isomorphism $W \to W^*$. Let $\{w_1,\dots,w_M\} \subset W$ be the pre-image of $\{P_1,\dots,P_M\}$ under this isomorphism and let  $u \in \BV^\Acal(\Omega)$ so that
\[
\Acal u = \sum_{i=1}^r \dpr{P_i,\Acal u} w_i.
\]
Now, invoking Proposition~\ref{lem:1} for each individual term $\dpr{P_i,\Acal u}$ ---each $(\xi_i, e_i)$ lies in the directional spectrum---  we conclude that
\begin{equation}\label{eq:slice}
	\Acal u = \sum_{i=1}^M \left(\int_{\pi_{\xi_i}} D u^{e_i}_{y,\xi_i} \dd \Hcal^{n-1}(y)\right) w_i.
\end{equation}     The sought assertion then follows from the triangle inequality.

The implication (2) $\Rightarrow$ (1) follows directly from Lemma~\ref{lem:slicemix}.
%
%
\end{proof}

\subsection{Algebraic constructions}  We shall see how, given a spectral pair $(\xi,e) \in \partial\sigma(\Acal)$, one can algebraically set the foundations of what the slice $\Acal_\xi^e$ of $\Acal$ should be: an operator that is invariant with respect to $\partial_\xi$ and with respect to the $e$-coordinate. 
 In all that follows and for the rest of this section let us be given a non-trivial pair (so that $\xi \otimes e$ is a non-zero tensor)
\[
	(\xi,e) \in \partial\sigma(\Acal).
\] 
Let us recall the notation $\bar f_\Acal(\xi\otimes v) = f_\Acal(\xi,v) \coloneqq \Abb(\xi)[v]$, defined in the introduction, and consider the pullback map 
\[
\begin{tikzcd}[row sep=0pt,column sep=1pc]
	g_\Acal \colon W^* \arrow{r} & (\R^n \otimes V)^* \\
	{\hphantom{g_\Acal \colon{}}} w^* \arrow[mapsto]{r} &  w^* \circ \bar f_\Acal
\end{tikzcd},
\] 
Notice that by the definition of essential image, the map $g_\Acal$ is injective when restricted to $(W_\Acal)^*$. 
We are now in position to start the construction of $\Acal_\xi^e$. The first step is to remove the coordinates of $V$ spanned by $e$:
\begin{definition} We define the subspace $V_e$ of $V$ as
	\[
	V_e \coloneqq \begin{cases}
	V & \text{if $\dim(V) = 1$},\\
	e^\perp \coloneqq \set{v \in V}{\dpr{e,v} = 0} &\text{if $\dim(V) \ge 2$}.
	\end{cases}
	\]
	We  write $\pbf^e : V \to V_e$ to denote the canonical linear projection onto $V_e$.
\end{definition} 
The second step is to remove all coordinates of $\Acal u$ (by separation with $W^*$) containing partial derivatives on directions that interact with $\xi$, as well as all the coordinates related to $u^e$.  
To this end, let us consider the subspace of $(\R^n \otimes V)^*$ defined by
\[
Y_{\xi}^e \coloneqq  \begin{cases}
	\spn\{\xi \otimes e\} & \text{if $\dim(V) = 1$},\\
	\R^n \otimes e  +  \xi \otimes V^*& \text{if $\dim(V) \ge 2$}.
\end{cases}
\]
Since we are targeting elements in $W^*$, it is natural to work with the pre-image of the isometry $g_\Acal : (W_\Acal)^* \to (\R^n \otimes V)^*$. We thus consider the subspace $X_{\xi}^e \coloneqq g_\Acal^{-1}[Y_\xi^e]$, which leads us to the following definition:
\begin{definition}
	We write $W_{\xi}^e$  to denote the subspace of $W$ defined by the property
\[
W_{\xi}^e \coloneqq (X_{\xi}^e)^\perp.
\]
In all that follows we shall write $\pbf_{\xi}^e : W \to W_\xi^e$ to denote the linear canonical linear projection from $W$ onto $W_\xi^e$. 
\end{definition}

The following result is fundamental for key posterior arguments. It guarantees that the  results contained in the forthcoming Lemma~\ref{lem:sub} and Corollary~\ref{cor:hyper} are non-trivial under our main assumptions:

\begin{proposition}\label{prop:nontrivial} Let $n \ge 2$ and let 
	$\Acal:\Crm^\infty(\R^n;V) \to \Crm^\infty(\R^n;W)$ be a first-order elliptic operator. If $(\xi,e) \in \partial\sigma(\Acal)$, then $W_\xi^e$ is non-trivial.
\end{proposition}
\begin{proof}
	Let us argue by a contradiction argument. If indeed $W_\xi^e$ was trivial, then
	$\dim(X_\xi^e) = \dim(W_\Acal)$ ---here we are using that $g_\Acal$ is one-to-one when restricted to $(W_\Acal)^*$. If $\dim(V) = 1$, then $\dim(X_\xi^e) = 1$ and by ellipticity we also have $\dim(W_\Acal) \ge n \ge 2$ ---thus, reaching a contradiction. Else, we may find non-zero vectors $\eta \in \pi_\xi$ and $a \in e^\perp$. Therefore, using that $\pbf_\xi^e \equiv 0$, we get $\Abb(\eta)[a] = (\id_W - \pbf_\xi^e) \circ \Abb(\eta)[a] =  0$. 
	Here, in reaching the last equality we have used the representation of linear maps and the fact that $X_\xi^e = g_{\Acal}^{-1}\set{\xi^* \otimes v + \zeta^* \otimes e^*}{\zeta \in \pi_\xi,v \in V}$. This poses a contradiction with the assumption that $\Acal$ is elliptic.
\end{proof}

\begin{notation} Let $X$ be a $\Kbb$-linear vector space, where $\Kbb = \R,\C$.
	For a vector $a \in X$. We write $\ell_a$ to denote the span of $a$ in $X$.
\end{notation}
Next, we show that both the $\xi$- and $e$-coordinates are algebraically invariant for the PDE when restricted to $W_\xi^e$. 
\begin{proposition}Let $\Acal : \Crm^\infty(\R^n;V) \to \Crm^\infty(\R^n;W)$ be a first-order operator and let $(\xi,e) \in \partial\sigma(\Acal)$. Then,
\begin{equation}\label{eq:vanishes}
\pbf_{\xi}^e \circ f_\Acal \equiv 0 \quad \text{on $(\R^n \times \ell_e) + (\ell_\xi \times V)$.}
\end{equation}
In particular, we obtain the equivalence of bi-linear forms
\begin{equation}\label{eq:restriction}
(\pbf_{\xi}^e \circ f_\Acal)|_{\pi_\xi \times V_e} \equiv f_\Acal|_{\pi_\xi \times V_e}.  
\end{equation}
\end{proposition}
\begin{proof} By definition $\pbf_\xi^e \circ f_\Acal \in (X_\xi^e)^\perp$. The sought assertion then follows directly from the definition of $X_\xi^e$.
%
%
%
\end{proof}

We have the following direct consequence:

\begin{proposition}\label{prop:restriction}
	Let  $\Kbb = \R,\C$ and let $\Acal:\Crm^\infty(\R^n;V) \to \Crm^\infty(\R^n;W)$ be a first-order $\Kbb$-elliptic operator. Assume that  $(\xi,e) \in \partial\sigma(\Acal)$, then  
	\[
		 \pbf_\xi^e \circ f_\Acal : \Kbb\pi_\xi \times (\Kbb \otimes V_e) \to (\Kbb \otimes W_\xi^e)
	\]
	defines a non-singular $\Kbb$ bi-linear form.
\begin{proof}
	The result follows form~\eqref{eq:restriction} and the fact that $f_\Acal$ is itself  a non-singular $\Kbb$ bi-linear form on the state domain of definition.
	%
\end{proof}
%
\end{proposition}

\subsection{Stability properties of co-dimension one slicing}
With the constructions we have done so far  we have now a toolbox to guarantee the existence of non-trivial slices \guillemotleft~$\Acal_\xi^e \mres \pi_\xi$~\guillemotright\; of $\Acal$, which will ultimately be the pivotal tool  towards the slicing on codimension-one planes and the proof of the structure and fine properties theorems. 

Let us introduce the restriction operator for differential operators:

\begin{definition}
	Let $\Acal : \Crm^\infty(\R^n;V) \to \Crm^\infty(\R^n;W)$ be a differential operator (of arbitrary order $k$) and let $\pi \in \Gr(n)$. We define the operator $$\Acal \mres \pi : \Crm^\infty(\pi;V) \to \Crm^\infty(\pi;W),$$ which is associated to the symbol
	\[
		(\Abb \mres \pi)^k \coloneqq \Abb^k|_{\pi}.
	\]
	This operator has the intrinsic property that
	\[
		\Acal \mres \pi (\phi) = \Acal (\phi \circ \pbf) \quad \text{for all $\Crm_c^\infty(\pi;V)$}. 
	\]	
\end{definition}

Let us recall the notation $\pbf_\xi : \R^n \to \pi_\xi$ for the canonical linear orthogonal projection from $\R^n$ onto $\pi_\xi$. The next Lemma studies the stability properties of the slices $\Acal_\xi^e$ with respect to its restriction on $\pi_\xi$.
\begin{lemma}[Sub-operators]\label{lem:sub}
Let $n \ge 2$ and let $\Acal:\Crm^\infty(\R^n;V) \to \Crm^\infty(\R^n;W)$ be a first-order operator. Suppose that $(\xi,e) \in \partial\sigma(\Acal)$ is a non-trivial pair and define the first-order operator 
	\[
		\Acal_\xi^e :\Crm^\infty(\R^n;V_e) \to \Crm^\infty(\R^n;W_\xi^e),
	\]
	which is associated to the principal symbol
	\[
		\Abb_\xi^e (\eta)[v] \coloneqq   \pbf_\xi^e \circ  \Abb(\eta)[v] \quad \text{for all $\eta \in \R^n$ and $v \in V_e$.}
	\]	
	Then,  
	\[
			(\Acal_\xi^e \mres \pi_\xi)( \phi) = \Acal (\pbf^e \circ \phi \circ \pbf_\xi)  \quad \text{for all $\phi \in \Crm^\infty(\pi_\xi;V)$}
	\]
	and the following implications hold:
\begin{enumerate}
	\item[(i)] $\Acal_\xi^e$ is $\partial_\xi$-invariant;
	\item[(ii)] $\Acal$ is $\Kbb$-elliptic $\Longrightarrow$ $\Acal_\xi^e\mres \pi_\xi$ is $\Kbb$-elliptic  (\,$\Kbb = \R,\C$\,);
	\item[(iii)] $\Acal$ satisfies~\eqref{eq:mix} $\Longrightarrow$  $\Acal_\xi^e \mres \pi_\xi$ satisfies~\eqref{eq:mix};
%
%
	\item[(iv)] The rank-one cones of these operators satisfy the set contention
	\[
		(\Acal_\xi^e \mres \pi_\xi)^\otimes_1 \subset \Acal^\otimes_1.
	\]

\end{enumerate}
\end{lemma}
\begin{proof} 

	The proof~(i) follows directly from~\eqref{eq:vanishes}, whilst ~(ii) follows from the previous proposition. To show (iii) we first notice that the case $n = 2$ is trivial: in this case $\Acal_\xi^e \mres \pi_\xi$ is morally an elliptic operator acting on functions of one variable and hence every non-zero element in $W_\xi^e$ is $\rank_\Acal$-one. This shows that $\Acal_\xi^e$ satisfies the rank-one property and therefore also~\eqref{eq:mix}. Now, assume that $n \ge  3$ so that $\pi \in \Gr(n-1,n)$ if and only if there exists $\gamma \in \R^n$ and a non-trivial $\tilde \pi \in \Gr(n-2,\pi_\xi)$ such that $\pi = \spn\{\gamma,\tilde\pi\}$. Using the identity $\pbf_\xi^e[f_\Acal] \equiv f_\Acal$ (as bi-linear forms on $\pi_\xi \times V_e$) we deduce that $\spn\setn{\Abb_\xi^e(\eta)[v]}{\eta \in \tilde \pi, v \in V_e}  \subset 	 \spn\set{\Abb(\eta)[v]}{\eta \in \pi, v \in V}$. 
	Moreover, if $P$ is in the intersection of the first linear space for all $\tilde \pi \in \Gr(n-2,\pi_\xi)$, then so it is in the intersections of the second space for all $\pi \in \Gr(n-1)$. Indeed, if there exists $\pi$ for which $P$ does not belong to it, then $P$ would also not belong to the first sum for any $\tilde \pi \in \Gr(n-2,\pi \cap \pi_\xi) \subset \Gr(n-2,\pi_\xi)$. This shows that
	\[
		\bigcap_{\tilde \pi \in \Gr(n-2,\pi_\xi)} \spn\set{ \Abb_\xi^e(\eta)[v]}{\eta \in \tilde \pi, v \in V_e} \subset \bigcap_{\pi \in \Gr(n-1)} \spn\set{\Abb(\eta)}{\eta \in \pi}. 	
	\] 
	We are thus in position to use that $\Acal$ satisfies~\eqref{eq:mix}, from where it directly follows that $\Acal_\xi^e \mres \pi_\xi$ must also satisfy the mixing property~\eqref{eq:mix} with $\pi_\xi$ in place of $\R^n$.
	Lastly, we show~(iv). Let $(\omega,h)$ be an arbitrary pair in $\partial\sigma(\Acal_\xi^e\mres \pi_\xi)$.
	We must check that indeed $(\omega,h) \in \partial\sigma(\Acal)$. 
	To this end let $w \in (W_\xi^e)^*$ be the vector satisfying $g_{\Acal_\xi^e \mres \pi_\xi}(w) = \omega^* \otimes  h^*$. Invoking~\eqref{eq:restriction} we find that
	\[
		w \circ f_\Acal(\eta,v) = w \circ f_{\Acal_\xi^e}(\eta,v) = \dpr{\omega^*,\eta}\dpr{h^*,v} \quad \text{for all $\eta \in \pi_\xi, v \in V_e$.}
	\]
	Thus, in order to show that $(\omega,h) \in \partial\sigma(\Acal)$, the standard representation of linear maps tells us that it suffices to show that $w \circ f_\Acal$ vanishes on $(\ell_\xi \times V) \cup (\R^n \times \ell_e)$. This however follows from~\eqref{eq:vanishes} and the fact that $w \in (W_\xi^e)^*$. 
\end{proof}

We are now in position to state the slicing on co-dimension one planes:

\begin{corollary}[Slicing on hyper-planes]\label{cor:hyper}
	Let $\Acal:\Crm^\infty(\R^n;V) \to \Crm^\infty(\R^n;W)$ be a first-order elliptic operator and let $(\xi,e)$ be a non-trivial pair in $\partial\sigma(\Acal)$ . Consider the restriction operator $\Bcal_\xi^e : \Crm^\infty(\pi_\xi;V_e) \to \Crm^\infty(\pi_\xi;W_\xi^e)$ defined by
	\[
	\Bcal_\xi^e \coloneqq \Acal_\xi^e \mres \pi_\xi,  
	\]
	where $\Acal_\xi^e$ is the sub-operator defined in Lemma~\ref{lem:sub}. Given $u : \R^n \to V$ and a vector $z \in \ell_\xi$, we define a function $v_z : \pi_\xi \to V_e$ as 
	\[
		v_z(y) = 
		\pbf^eu(z + y) 
	\] 
	where as usual $\pbf^e : V \to V_e$ is the linear orthogonal projection onto $V_e$.

	Then, for every $u \in \BV^\Acal(\R^n)$, the following holds:
	\begin{enumerate}\itemsep3pt
		\item[\textnormal{(\emph{1})}] The functions $v_z$ are well-defined and belong to $\BV^{\Bcal_\xi^e}(\pi_\xi)$ for $\Hcal^1$-almost every $z \in \ell_\xi$. Moreover, 
		\[
			\int_{\ell_\xi} |\Bcal_\xi^e v_z|(\pi_\xi) \dd \Hcal^1(z) < \infty.
		\] 
		\item[\textnormal{(\emph{2})}] For every Borel set $B \subset \R^n$ we have 
		\[
		(\pbf_\xi^e [\Acal u])(B) = \int_{\ell_{\xi}} \Bcal_\xi^e  v_z(B_z) \dd \Hcal^{1}(z), 
		\]   
		\[
			|\pbf_\xi^e [\Acal u]|(B) = \int_{\ell_\xi} |\Bcal_\xi^e v_z|(B_z) \dd \Hcal^{1}(z),
		\]
		where $B_z \coloneqq \set{y \in \pi_\xi}{y + z \in B}$ is the $\pi_\xi$-slice of $B$ at $z \in \ell_\xi$.
	\end{enumerate}
\end{corollary}
\begin{proof} Let $\varphi \in \Crm^\infty_c(\R^n;V)$ and let us first assume that $u \in \Crm^\infty(\R^n;V)$. We write $\phi_z(y) = \phi(z + y)$. By Fubini's Theorem and integration by parts we obtain
	\begin{align*}
		\dpr{\pbf_\xi^e [\Acal u],\varphi} & = \dpr{\Acal_\xi^e u,\varphi } \\
		& = \int_{\ell_\xi} \bigg( \int_{\pi_\xi} ([\Acal_\xi^e \mres \pi_\xi] u)(z + y) \phi(z+ y) \dd \Hcal^{n-1}(y) \bigg)\dd \Hcal^{1}(z) \\
		& = \int_{\ell_\xi} \bigg( \int_{\pi_\xi} \Bcal_\xi^e v_z(y) \phi_z(y) \dd \Hcal^{n-1}(y) \bigg)\dd \Hcal^{1}(z).
	\end{align*}
	Here, in the second equality we have used that $\Acal_\xi^e$ is $\partial_\xi$-invariant; in passing to the last equality we have used that $(\Acal_\xi^e \mres \pi_\xi) u(z + \frarg) = \Acal(\pbf_e u(z + \frarg) \pbf_\xi) = \Acal_\xi^e \mres \pi_\xi (v_z)$ (cf. Lemma~\ref{lem:sub}). Taking the supremum over all $\phi$ gives the inequality
	\[
		|\Acal u|(\R^n) \ge \int_{\ell_\xi}\bigg(\int_{\pi_\xi} |\Bcal_\xi^e v_z| \dd \Hcal^{n-1} \bigg) \dd\Hcal^1{(z)}.
	\]
	This estimate is stable under mollification: Let $\rho_\eps$ be a standard mollifier at $\eps$-scale. Using the standard notation $u_\eps \coloneqq u \star \rho_\eps$ and  Young's inequality we obtain the uniform bound
	\[
		  |\Acal u|(\R^n) \ge |\pbf^e_\xi[\Acal u_\eps]|(\R^n) \ge \int_{\ell_\xi}  |\Bcal_\xi^e \pbf^e [u_\eps(y + z)]|(\pi_\xi)\dd \Hcal^{1}(z).
	\]
	Fubini's theorem guarantees that, for $\Hcal^1$-almost every $z \in \ell_\xi$, we have the convergence $w_\eps \coloneqq u_\eps(z + \frarg) \to v_z$ in $\Lrm^1(\pi_\xi)$. This also means that $\Bcal_\xi^e w_\eps \to \Bcal_\xi^e v_z$ in the sense of distributions on $\pi_\xi$. Therefore, the map $w \mapsto |\Bcal_\xi^e w|(\pi_\xi)$ is lower semicontinuous with respect to $\Lrm^1(\pi_\xi)$ convergence. Then,  the estimate above and
	Fatou's lemma further imply that 
	\[
		|\Acal u|(\R^n) \ge \int_{\ell_\xi} |\Bcal_\xi^e v_z|(\pi_\xi) \dd \Hcal^1(z).
	\]
	This shows that if $u \in \BV^\Acal(\R^n)$, then~(1) holds. 
	
	Similarly to the identity for smooth functions, using Fubini's theorem one shows that the identity
	\begin{align*}
		\dpr{\pbf^e_ \xi[\Acal u],\phi}_{\R^n} & = \int_{\ell_\xi} \bigg( \int_{\pi_\xi} v_z \,  [\Bcal_\xi^e]^*\phi_z \dd \Hcal^{n-1}\bigg)\dd \Hcal^{1}(z) \\  
		& = \int_{\ell_\xi} \dpr{\Bcal_\xi^e v_z,\phi_z}_{\pi_\xi} \dd \Hcal^{1}(z).
	\end{align*}
	holds, in the sense of distributions for, for all $u \in \Lrm^1(\R^n;V)$. Here, the suffix $\dpr{\frarg,\frarg}_X$ indicates that the pairing is to be taken in $\Dcal'(X) \times \Dcal(X)$. The first statement  in~(2) then follows from a classical approximation argument for sets and~(1), which ensures that $\dpr{\Bcal_\xi^e v_z,\phi_z}_{\pi_\xi} = \int_{\pi_\xi} \phi_z \dd \Bcal_\xi^e v_z$ for $\Hcal^1$-almost every $z \in \ell_\xi$. The second statement in~(2) follows from characterization of the total variation of generalized product measures.  
	
This finishes the proof.
\end{proof}

\subsection{Slices of arbitrary co-dimension} The previous proposition extends to the following more general context: given a subspace $\mathscr V \in \Gr(\ell,n)$, there exists an operator $\Bcal_{\mathscr V} : \Crm^\infty(\mathscr V;V) \to \Crm^\infty(\mathscr V;W)$ and a linear projection $p_{\mathscr V} : W \to W_{\mathscr V}$ such that
	\[
		p_{\mathscr V}[\Acal u] = \int_{{\mathscr V}^\perp} \Bcal_{\mathscr V} v_z \dd \Hcal^{n- \ell}(z).
	\]
	In fact,  the proof of this statement only requires the rank-$\ell$ property $$\spn \Acal^\otimes_\ell = (W_\Acal)^*,$$ or equivalently (the proof is left to the reader),
	\[
		\bigcap_{{\substack{\pi \, \le \, \R^d\\\dim(\pi) = n-\ell}}} \spn\set{\im \Abb(\eta)}{\eta \in \pi} = \{0\}.
	\]
	However, since $\Acal$ may not have a self-similar algebraic design (as the gradient or the symmetric gradient have), in general there is no straightforward formula for $\Bcal_{\mathscr V}$ in terms of $\Acal$,  other than the hefty one given by iteration of slicing: 
	\[
		\Bcal_{\mathscr V} = (((\Acal_{\xi_1}^{e_1})_{\xi_2}^{e_2})\dots )_{\xi_{n-\ell}}^{e_{n-\ell}} \mres \mathscr V,
	\] 
	where $\{\xi_i,e_i\}_{i = 1}^{n - \ell} \subset \partial\sigma(\Acal)$ and $\{\xi_1,\dots,\xi_{n-\ell}\}$ is a basis of ${\mathscr V}^\perp$ (the  last step requires the projection result contained in Corollary~\ref{cor:projection} below).

\begin{remark}
	As we shall see later (in Section~\ref{sec:examples}), the deviatoric operator 
	\[
		Lu \coloneqq Eu - \frac{\id_{\R^n}}{n} \Div(u), \quad u : \R^n \to \R^n,
	\]
	not only does not satisfy the rank-one property, but its set of $\rank_L$-one tensors is empty. In light of the slicing theorem, it is clear that no coordinate of $Eu$ cannot be sliced into one-dimensional sections. However, it is easy to check that $L$ satisfies the rank-two property and therefore it can be sliced into two-dimensional slices.
\end{remark}

\subsection{Polarization properties} The purpose of this section is to verify that, under ellipticity and the rank-one condition, there exist sufficient transversal spectral pairs. The next lemma is inspired in a key transversality result of~\cite{ambrosio1997fine-properties}, in which the one-dimensional structure theorem there hinges on. 
\begin{proposition}[Polarization]\label{lem:polarization}
	Let $n \ge 2$ and let $\Acal : \Crm^\infty(\R^n;V) \to \Crm^\infty(\R^n;W)$ be a first-order elliptic operator satisfying the rank-one property. Assume that $(\xi,e)$ is a non-trivial pair in the spectrum $\partial\sigma(\Acal)$. Then there exists a non-trivial pair
	\[
		\text{$(\eta,f) \in \partial\sigma(\Acal)$ for some $\eta \in \pi_\xi$ and $f \in (V_e)^*$.}
	\] 
	Moreover, for any such pair, there exists a  direction  $v \in \ell_f$ with 
	\[
			(\xi + \eta,e + v), (\xi - \eta,e - v) \in \partial\sigma(\Acal)
	\]
and $\spn\{\xi,\eta\} \subset \proj_{\R^n} \partial\sigma(\Acal)$.
\end{proposition}
\begin{proof}
	The existence of $(\eta,f)$ is a direct consequence of Lemma~\ref{lem:sub} and Proposition~\ref{prop:nontrivial}. We shall therefore focus on the second statement:
	
	 \emph{1. Reduction to the case $n = 2$.}  
	Let us assume that $n \ge 3$ for otherwise there is nothing to show.
	By construction we get 
	\[
		(\xi,e) \in \partial\sigma(\Bcal), \quad \Bcal \coloneqq \Acal_\eta^f \mres \pi_\eta,
	\] 
	where $\Bcal$ is elliptic and satisfies the rank-one property.  
	Since $\dim(\pi_\eta) \ge 2$, then the slice $\Bcal_\xi^e \mres (\pi_\xi \cap \pi_\eta)$ is a non-trivial elliptic operator satisfying the rank-one property. This conveys the existence of a pair $(\omega,h) \in \partial\sigma(\Acal) \cap ((\pi_\xi \cap \pi_\eta) \times (V_e \cap V_\eta)^*)$
	and, in particular, 
	\[
		(\xi,e),(\eta,f) \in \partial\sigma({\Acal_\omega^h \mres \pi_\omega}) \subset \partial\sigma(\Acal).
	\]
	Once again, we observe that  $\Acal_\omega^h \mres \pi_\omega$ is elliptic and satisfies the rank property. Therefore, the statement of the proposition holds if and only if an analogous statement holds for the operator $\Acal_\omega^h \mres \pi_\omega$. An iteration of this argument tells us there is no loss of generality in assuming that $n = 2$. 
	
	
	\emph{2. The case when $\dim(V) = N > 2$.} Let us assume that $N > 2$ (and recall from the previous step that $n = 2$). As before, consider the slice $\Bcal = \Acal_\xi^e \mres \pi_\xi$, which in this case is an elliptic  differential operator in one variable (the $\eta$-variable). More precisely $	\Bcal : \Crm^\infty(\ell_\eta;V_e) \to \Crm^\infty(\ell_\eta;W_\xi^e)$. 
	It follows that $\Bcal$ contains a gradient operator and therefore 
	$\{\eta\} \times (V_e)^*\in \partial\sigma(\Bcal) \subset \partial\sigma(\Acal)$. Similarly, the slice $\Acal_{\eta}^f \mres \ell_\xi$ acts on $\Crm^\infty(\ell_\xi;V_f)$ and by the same reasoning above we get $\{\xi \}\times (V_f)^*\subset \partial\sigma({\Acal})$. 
	Since $\dim(V) \ge 3$, we also have $\dim(V_v \cap V_e) \ge N-2 \ge 1$.  In particular, using the bi-linearity of $f_\Acal$ we find that
	\begin{equation}\label{eq:mac}
		\spn\{\xi,\eta\} \times H^* \subset \partial\sigma(\Acal), \quad H \coloneqq V_e \cap V_f. 
	\end{equation}
	Up to a change of variables we may assume that $\xi + \eta$ and $\xi - \eta$ are orthogonal vectors. Up to multiplication by a constant, we may also assume that $|\xi| = |\eta| = 1$.
	Now, working with the slice $\Acal_{\xi + \eta}^h \mres \ell_\omega$ for some non-zero $h \in H^*$,  we find  (this slice must be a gradient) that $\ell_{\xi - \eta} \times \spn\{e,f\} \subset \partial\sigma(\Acal)$. 
	Thus, again by the bi-linearity of $f_\Acal$ and the fact that $(\xi,e)  \cup (\eta,f) \in \partial\sigma(\Acal)$, we conclude that $\spn\{\xi,\eta\} \times \spn\{e,f\} \subset \partial\sigma(\Acal)$,
	from where the sought assertion trivially follows.
	

	\emph{3. The case when $\dim(V) \le 2$.} The proof when $\dim(V) = 1$ is trivial since then~\eqref{eq:mac} holds trivially for the generating vector $f$ of $V^*$. We shall hence focus in  the case when $n = \dim(V) = 2$.
	The first observation is that
	\begin{equation}\label{eq:proj}
	\proj_{\R^2} [\partial\sigma(\Acal)] = \R^2, \quad \proj_{V^*} [\partial\sigma(\Acal)] = V^*.
	\end{equation}
	This means that every $\eta \in \R^2$ and every $e \in V^*$ are the first and second coordinates (respectively) of some element in the directional spectrum. The proof of this follows directly from the mixing property in two-dimensions: for a given $\eta \in \Sbf^{1}$, the image $\im \Abb(\eta_\perp)$ cannot be the whole of $W$ and therefore there exists a non-zero $w^* \in \Abb(\eta_\perp)^\perp$ such that
	$\dpr{w^*, \Abb(\eta_\perp)[e]} = 0$ {for all $e \in V$.} 
	By the representation of linear maps, this implies that
	\[
	\dpr{w^*, \Abb(\omega)[v]} = \dpr{\eta , \omega}\dpr{\Abb(\eta)^*[w^*] , v} \quad  \text{for all $(\omega,v) \in \R^n \times V$.}
	\]
	Therefore we obtain $(\eta,\Abb(\eta)^*[w^*]) \in \partial\sigma(\Acal)$. Since $\eta \in \Sbf^1$ was arbitrarily chosen, this proves the first claim. The second claim follows from a symmetric argument on the $V$-variable. 
	
	
	The second observation is that, qualitatively speaking, there are only two possible cases  (recall that $g_\Abb : (W_\Acal)^* \to (\R^2)^* \otimes V^*$ is one-to-one): $\dim(W_\Acal) \in \{3,4\}$. 
	Let us first understand the case when $\dim(W_\Acal) = 4$, which is easier. Clearly, this is the case when $W_\Acal^* \cong (\R^2)^* \otimes V^*$. In particular, $\R^2 \times V^* = \partial\sigma(\Acal)$. 
	Thus, the conclusion of the Proposition holds trivially. Let us now address the case when $\dim(W_\Acal) = 3$. Firstly, we claim that to each $[\zeta] \in \Pbb\R^1$ there corresponds one (and only one) $[v_\zeta] \in \Pbb V^*$ such that $(\zeta,v_\zeta) \in \partial\sigma(\Acal)$; moreover, this assignment is injective. The fact that to each line in $\Pbb\R^1$ corresponds at least one representative non-zero vector $v_\zeta \in V^*$ follows from~\eqref{eq:proj}. We are left to check that there cannot be more than one spectral $V^*$-coordinate attached to any direction $\zeta \in \Sbf^1$. If this was the case for some $\zeta \in \Sbf^1$, then a linearity argument would give $\{\zeta_0\} \times V^* \subset \partial\sigma(\Acal)$ (here we are using that $\dim(V) = 2$). However, by a similar linearity argument (now on the $\zeta$-variable), this would also imply that $\Sbf^1 \times \{v_0\} \subset \partial\sigma(\Acal)$ for some non-zero $v_0 \in V^*$. Hence, all four pairs $(\zeta_0,v_0),(\zeta_0,(v_0)_\perp),((\zeta_0)_\perp,v_0)((\zeta_0)_\perp,(v_0)_\perp)$ would belong to $\partial\sigma(\Acal)$. This however implies that $\dim(W_\Acal) \ge 4 > 3$; therefore reaching a contradiction. This proves that the assignment $\zeta \mapsto [v_\zeta]$ is well defined. That the map is one-to-one follows by inverting the roles of $\zeta$ and $v$.

	From this observation, we can give a basis for $g_\Acal(W_\Acal^*) \le (\R^2 \otimes V)^*$ conformed by the rank-one tensors $\{\xi \otimes e, \eta \otimes f, (\xi + \eta) \otimes (\alpha e + \beta f)\}$. 
	Here, $\alpha,\beta \in \R$ are chosen so that  $(\xi + \eta) \otimes (\alpha e + \beta f)$ belongs to $\partial\sigma(\Acal)$. Observe that since $[\zeta] \mapsto [v_\zeta]$ is one-to-one, then both $\alpha,\beta$ are non-zero reals. We are now in position to determine precisely the map $\Pbb\R^1 \to \Pbb V^* : [\eta] \mapsto [v_\eta]$. An equivalent basis to the one given above, corresponds to the elements
	$\{\xi \otimes e, \eta \otimes f,\alpha(\xi \otimes e) + \beta(\eta \otimes f)\}$. 
	Now, let $\zeta \in \Sbf^1$ be an arbitrary direction, which we may write as $\zeta = c\xi + d\eta$ for some reals $c,d \in \R$. Likewise we may write $v_\zeta = h e + g f$. Developing the tensorial product we discover that $
	\zeta \otimes v_\zeta = ch(\xi \otimes e) + dg(\eta \otimes f) + cg(\xi \otimes f) + dh(\eta \otimes e)$. 
	Since the left hand side is rank-one tensor, it must hold that $[(\alpha,\beta)] = [(cg,dh)] \in \Pbb\R^1$. Herewith,  we deduce that $[(h,g)] =  [(\alpha c,\beta d)]$ and thus the map $\nu \mapsto [v_\nu]$ is explicitly given by the assignment
$[c \xi + d \eta] \mapsto  [(\alpha c) e+ (\beta d) f]$. The sought polarization property follows by taking $c = \alpha^{-1}$, $d = \pm 1$ and $v = \pm\beta f$.
 This finishes the proof of the last possible case.
\end{proof}

\begin{corollary}\label{cor:projection} Let $\Acal : \Crm^\infty(\R^n;V) \to \Crm^\infty(\R^n;W)$ be a first-order operator satisfying the the rank-one property. Then,
	\[
		\R^n = \proj_{\R^n} \partial\sigma(\Acal) \quad \text{and} \quad V^* = \proj_{V^*} \partial\sigma(\Acal).
	\]
\end{corollary}
\begin{proof} If $n = 1$, then the proof is trivial. For $n = 2$, the proof follows from the last statement in the previous proposition. The case for $n \ge 3$ follows by induction (using Lemma~\ref{lem:sub}) and a simple geometric argument. The assertion for the projection onto $V$ follows by a symmetric argument.   
\end{proof}

Corollary~\ref{cor:projection} allows us to make an improvement in the dimensional estimates for the total variation measure of $\Acal u$. Namely, we pass from absolute continuity with $\Hcal^{n-1}$ to absolute continuity with respect to the $(n-1)$-dimensional integral-geometric-measure:

	\begin{corollary}\label{cor:dimensional1}
	Let $\Acal : \Crm^\infty(\R^n;V) \to \Crm^\infty(\R^n;W)$ be a first-order homogeneous elliptic linear  differential operator satisfying the the mixing condition~\eqref{eq:mix}. Let $u$ be a function in $\BV^\Acal(\Omega)$ and let $B \subset \Omega$ be a Borel set satisfying 
	\[
	\Hcal^{n-1}(\pbf_\xi(B)) = 0 \quad \text{for $\Hcal^{n-1}$-almost every $\xi \in \Sbf^{n-1}$.}
	\]
	Then,
	\[
	|\Acal u|(B) = 0.
	\]
\end{corollary}
\begin{proof}
	Let $B \subset \Omega$ satisfy the assumptions of the Corollary for all $\xi \in \Sbf$, where $\Sbf \subset \Sbf^{n-1}$ is a full $\Hcal^{n-1}$-measure subset. By virtue of the previous result and a continuity argument (the manifold of rank-one tensors is closed in $\R^n \otimes V$), we may find a family of coordinates $\{P_1,\dots,P_M\}$ spanning $(W_\Acal)^*$, and, for each $j \in \{1,\dots,M\}$, we may find covectors $\xi_j \in \Sbf$ and $e_j \in V^*$ such that
	\[
	\dpr{P_j, \Abb(\xi)[v]} = \dpr{\xi_j , \xi} \dpr{e_j , v}, \quad \text{for all} \; \xi \in \R^n, v \in V.
	\]
	Then, the Structure Theorem~\ref{thm:structure} gives
	\begin{align*}
		|\dpr{P_j, \Acal u}|(B) & = \int_{\pi_{\xi_j}} |D u_{y,\xi_j}^{e_j}|(B_{\xi_j}^y) \dd \Hcal^{n-1}(y) \\
		& = \int_{\pbf_{\xi_j}(B)} |D u_{y,\xi_j}^{e_j}|(B_{\xi_j}^y) \dd \Hcal^{n-1}(y) = 0 \qquad (\xi_j \in \Sbf).
	\end{align*}
	Since $\spn\{P_1,\dots,P_M\} = (W_\Acal)^*$, we conclude that $|\Acal u|(B) = 0$.
\end{proof}

\begin{remark}
For  future reference let us recall from~\cite[Thm,~1.1]{anna} that $J_u \subset \Theta_u$ whenever $\Acal$ is an elliptic operator satisfying the rank-one property.
\end{remark}
	\begin{corollary}\label{cor:iv} Let $\Acal$ as in the previous corollary and 
	let $u \in \BV^\Acal(\Omega)$. Then 
	\begin{enumerate}
		\item $|\Acal u|(\Theta_u \setminus J_u) = \Hcal^{n-1}(\Theta_u \setminus J_u) = 0$,
		\item $\Theta_u$ is $\Hcal^{n-1}$-countably rectifiable,
		\item $\Acal^c u$ vanishes on $\sigma$-finite sets with respect to the $\Hcal^{n-1}$-measure. 
	\end{enumerate}
\end{corollary}
\begin{proof} A standard covering argument implies that $\Theta_u$ is a $\sigma$-finite set with respect to $\Hcal^{n-1}$. We may split $S \coloneqq \Theta_u \setminus J_u$ into two disjoint Borel sets 
	\[
	G \cup F \coloneqq \{\theta^{*{(n-1)}}(|\Acal u|) \in (0,\infty) \} \cup \{\theta^{*{(n-1)}}(|\Acal u|) = \infty\}.
	\]
	Following a standard procedure, we proceed to split $G$ as 
	$R \cup U$ where $R$ is countably $\Hcal^{n-1}$-rectifiable, $U$ is $\Hcal^{n-1}$-purely unrectifiable, and $\Hcal^{n-1}(R \cap U) = 0$. Once again, a standard covering argument gives $\Hcal^{n-1}(F \cup (R \cap U)) = 0$. On the other hand,
	the Besicovitch--Federer Theorem implies that $\Ical^{n-1}(U) = 0$. By the dimensional estimates of the previous corollary we get $|\Acal u|(S) = 0$. By the definition of $\Theta_u$, it also holds $\Hcal^{n-1} \mres \Theta_u \ll |\Acal u| \mres \Theta_u$, whence we conclude that $\Hcal^{n-1}(S) = 0$. This proves~(1).
	
	Since $J_u$ itself is countably $\Hcal^{n-1}$-rectifiable and $\Hcal^{n-1}(\Theta_u \setminus J_u) = 0$, then $\Theta_u$ is also countably $\Hcal^{n-1}$-rectifiable and (2) follows.


	Let $B \subset \Omega$ be a $\sigma$-finite Borel set with respect to $\Hcal^{n-1}$ so that $B = \bigcup_{i = 1}^\infty B_i$ where $\Hcal^{n-1}(B_i) < \infty$. From~(1) we know that
	\[
	|\Acal^c u|(B) = |\Acal^c u|(B \setminus \Theta_u).
	\]
	Since at every point in $\Omega \setminus \Theta_u$, the $(n-1)$-dimensional upper density of $|\Acal u|$ vanishes, a standard covering argument implies that, for any $\eps > 0$ it holds
	\[
	|\Acal^c u|(B_i) \le \eps \Hcal^{n-1}(B_i) \quad \text{for  all $i = 1,2,\dots$.}
	\] 
	Letting $\eps \to 0^+$ and using the finiteness of $\Hcal^{n-1}(B_i)$ for all $i \in \N$, we find that $|\Acal^c u|(B) = 0$. This proves~(3), which finishes the proof.
	%
	%
\end{proof}

\section{Analysis of Lebesgue points} Now we focus on the Lebesgue continuity properties. We shall see, by the end of this section that the discontinuous diffusion part $\Acal^d u \coloneqq \Acal^s u\mres (S_u \setminus J_u)$ vanishes for elliptic operators satisfying the mixing property. Moreover, we establish that the only one form of approximate discontinuity for $\BV^\Acal$-functions is the jump-type discontinuity.

The results of this section hinge on the algebraic robustness of the rank-one property. Often, the main difficulty of the proofs will reside in finding the correct way to cast the algebraic structure of $\Abb$ into the well-established techniques developed for the symmetric gradient from~\cite{ambrosio1997fine-properties}. 

\begin{lemma}\label{lem:induction}
	Let $\Acal : \Crm^\infty(\R^n;V) \to \Crm^\infty(\R^n;W)$ be a first-order homogeneous linear elliptic differential operator satisfying the mixing condition~\eqref{eq:mix}. 
	
	Let $u \in \BV^\Acal(\R^n)$, $x \notin \Theta_u$, and $\xi \in \Sbf^{n-1}$. Assume that $\Hcal^{n-1}$-almost every point in the slice $x + \pi_\xi$ is a Lebesgue point of $u$. If $(\xi,e) \in \partial\sigma(\Acal)$ and the function $v : \pi_\xi \to V_e$ defined by $v(y) = \pbf^e u(x + y)$ has one-sided Lebesgue point limits at $0 \in \pi_\xi$ with respect to a suitable direction $\nu \in \pi_\xi$, then 
	\begin{enumerate}
		\item $0$ is a Lebesgue point of $v$ in $\pi_\xi$,
		\item $x$ is a Lebesgue point of $\pbf^e u$ in $\R^n$.
	\end{enumerate}
\end{lemma}

\begin{remark}
	In proving property (2), we use that if $\Acal$ is complex-elliptic, then every $u \in \BV^\Acal(\R^n)$ is quasi-continuous on the set $\Theta_u$ (see~\cite[Propisition~1.2]{anna}). 
\end{remark}
\begin{proof}
	\emph{Step~1. Preparations.} We may assume without loss of generality that $x = 0$. 
	Let $(a^+,a^-,\nu)$ be the triple describing the one-sided limits of $v$ at $0 \in \pi_\xi$. Clearly, we may assume that $|a^+ - a^-| > 0$ for otherwise the first statement is trivial (and we may pass to the next step). By construction we have $a \coloneqq a^+ - a^- \in V_e$. 
	 
	 We claim that there exists a non-trivial pair $(\eta,f) \in \partial\sigma(\Acal)$ satisfying
	 \begin{align}\label{eq:positive}
	 	(\eta,f)  \in \pi_\xi \times( V_e)^*, \qquad \dpr{\eta,\nu} \neq 0,\qquad
	 	\dpr{f,a} \neq 0,
	 \end{align}
 and a non-zero covector $g \in \ell_f$ such that
 \begin{align}\label{eq:positive2}
 	(\xi + \eta, e + g),	(\eta - \xi, g - e) \in \partial\sigma(\Acal).
 \end{align}
	 First, we use that $\proj_{\pi_\xi} \partial\sigma(\Acal_\xi^e \mres \pi_\xi) = \pi_\xi$ to find a non-zero covector $d \in (V_e)^*$ with $(\nu,d) \in \partial\sigma(\Acal)$. If $|\dpr{d , a}| > 0$, then we simply set $(\eta,f) = (\nu,d)$. If however $\dpr{d,a} = 0$, then we have to solve two further sub-cases: This time we use the identity $\proj_{(V_e)^*} \partial\sigma(\Acal_\xi^e \mres \pi_\xi) = (V_e)^*$ to find a direction $\omega \in \Sbf^{n-1} \cap \pi_\xi$ with $(\omega,a^*) \in \partial\sigma(\Acal)$, where $a^* \in (V_e)^*$ is any covector satisfying $\dpr{a^*,a} = 1$. If $|\dpr{\omega,\nu}| > 0$, then we set $(\eta,f) = (\omega,a^*)$. Else, we may use  Lemma~\ref{lem:polarization} to find a  covector $h \in \ell_d$ such that
	 \[
	 	(\nu + \omega,h + a^*) \in \partial\sigma(\Acal).
	 \]
	 This pair satisfies the sought properties of~\eqref{eq:positive}. Since this is the only other possible case,  this proves~\eqref{eq:positive}. Property~\eqref{eq:positive2} follows directly from an application of Lemma~\ref{lem:polarization} with $(\xi,e)$ and $(\eta,f)$. 
	 
	 \emph{Observation:} Slicing the operator $\Acal_\xi^e \mres \pi_\xi$ with respect to the pair $(\eta,f)$, we may find another pair $(\eta_2,f_2) \in \partial\sigma(\Acal)$ satisfying~\eqref{eq:positive}. Repeating the same argument yields a family $\{\eta_j,f_j\}_{j = 1}^\ell \subset \partial\sigma(\Acal)$ satisfying~\eqref{eq:positive} where the $V$-coordinates also satisfy
	 \begin{equation}\label{eq:span}
	 		\spn\{f_j\}_{j = 1}^\ell = (V_e)^*.
	 \end{equation}
	 This will be used in Step~3.

	\emph{Step~2. Approximate continuity on $\pi_\xi$.} The idea is to show first that $0 \in \pi_\xi$ is a Lebesgue point of $v$. 
		This part of the proof mimics the proof of Theorem~5.1 in~\cite{ambrosio1997fine-properties} into our context. Let $Q_r$ be the open cube of radius $r$ that is centered at $y$ and with two of its axes oriented by the $\xi$- and $\eta$-directions respectively. We write $C_r = Q_r \cap \pi_\xi$ and $C_r^\pm = \set{y \in C_r}{\pm y \cdot \nu \ge 0}$. Lastly, we write $A_r$ to denote the $(n-1)$-dimensional open ball in $\pi_\xi$ with radius $r(\nu\cdot\eta)$. With these conventions we have $A_r \pm r \eta \in C_{2r}^\pm$. 
		
		Let $Y$  be the set of all real numbers $\rho \ge 0$ such that both $y \pm \rho \eta$ are Lebesgue points of $u$ in $\R^n$ for $\Hcal^{n-1}$-almost every $y \in \pi_\xi$. We know that $\Hcal^1({\R^+ \setminus Y}) = 0$ and that, by assumption, also $0 \in Y$.
		
		Let us record for later use that the one-sided limits assumption implies
	\begin{equation}\label{eq:2side}
	\lim_{r \to 0^+} \frac{1}{r^{n-1}}\int_{C_{2r}^\pm} |v{(y)} - a^\pm| = 0.
	\end{equation}

	
	The triangle inequality,~\eqref{eq:positive}, and a change of variables yield the estimate 
	\begin{equation}\label{eq:a}
	\begin{split}
	|a^+ - a^-|& \lesssim \frac 1{(2r)^{n-1}} \bigg(\int_{C_{2r}^+}{} |v(y) - a^+| \dd \Hcal^{n-1}(y) \\
	& \quad + \int_{C_{2r}^-}{} |v(y) - a^-| \dd \Hcal^{n-1}(y) \\
	& \qquad + \int_{A_{r}}{} |u^{g}(y + r\eta) - u^{g}(y - r\eta)| \dd  \Hcal^{n-1}(y)\bigg).
	\end{split}
	\end{equation}
	The change of variables $\tilde y = y \pm r \eta$  and the one-sided  continuity~\eqref{eq:2side} give that the first two  terms of the right-hand side above are of order $\BigO(r)$. Therefore, we only need to show the last term vanishes as $r \to 0^+$.  
	In all that follows we write $\tilde u$ to denote the Lebesgue representative of $u$. 
	Let $\rho \in Y$. Using the polarization from Lemma~\ref{lem:polarization} we may decompose, for $\Hcal^{n-1}$-a.e. $y \in \pi_\xi$, the difference $\tilde u^{g}(y - \rho \eta) - \tilde u^{g}(y + \rho \eta)$ as
	\begin{equation}\label{eq:KEY}
	\begin{split}
	2[\tilde u^{g}(y + \rho \eta) -  \tilde u^{g}(y - \rho \eta)] & = \tilde u^{e+g}(y + \rho \xi) - \tilde u^{e+g}(t - \rho \eta) \\
	&\quad + \tilde u^{g-e}(y + \rho \eta) - \tilde u^{g-e}(y + \rho \xi) \\
	&\qquad + \tilde u^{g-e}(y - \rho \xi) - \tilde u^{g-e}(y - \rho\eta) \\
	& \quad \qquad + \tilde u^{g+e}(y + \rho \eta) - \tilde u^{g+e}(y - \rho \xi) \\
	& \qquad \qquad - 2[\tilde u^{e
	}(y + \rho \xi) - \tilde u^e(y-\rho \xi)].
	\end{split}
	\end{equation}
	Each of this terms may be estimated by a total variation term in a transversal direction to $\xi$. Following a measure theoretic argument as in the proof of Theorem~5.1 in~\cite{ambrosio1997fine-properties}, we may (for $\Hcal^{n-1}$-almost every $y \in \pi_\xi$ and $\rho \in Y$) estimate $| \tilde u^g(t - \rho \eta) -  \tilde u^{g}(t + \rho \eta)|$, up to a multiplicative constant, by the sum
	\begin{equation*}
	V\tilde u_{\eta+\xi,y}^{g+e}([-\rho,\rho])   + V\tilde u_{\eta - \xi,y}^{g - e}([-\rho,\rho])   + V_{\xi,y}^e\tilde u([-\rho,\rho]),
	\end{equation*}
	where the latter is the one-dimensional total variation in terms of difference quotients (see, e.g.,~(2.8) in~\cite{ambrosio1997fine-properties}).
	Returning to the first estimate~\eqref{eq:a}, we then deduce from the bound $V\tilde h \le |Dh|$ for functions $h$ of one variable that
	\begin{align*}
	|a^+-a^-| & \lesssim  \frac c{r^{n-1}} \bigg( \int_{A_r} |D u^{g + e}_{\eta + \xi,y}|([-r,r]) \dd \Hcal^{n-1}(y) \\
	& \quad + \int_{A_r} |D u^{g-e}_{\eta - \xi,y}|([-r,r]) \dd \Hcal^{n-1}(y) \\
	& \qquad + \int_{A_r} |D u^{e}_{\xi,y} |([-r,r]) \dd \Hcal^{n-1}(y) \bigg)  + \BigO(r).
	\end{align*}
	Since all pairs $(\eta + \xi,e+g),(\eta - \xi,e-g),(\xi,e)$ all  belong to $\partial\sigma(\Acal)$, we may control each of the terms on the right hand side by $|\Acal u|(Q_{2r})$. The first part of the proof and the fact that $0 \notin \Theta_u$ give
	\[
	|a^+ - a^-| \lesssim \limsup_{r \to  0^+} \bigg( \frac C{r^{n-1}} |\Acal u|(Q_{2r}) + \BigO(r) \bigg) = 0.
	\]
	This shows that $0 \in \pi_\xi$ is a Lebesgue point of $v$.

	
	\emph{Step 3. Proof of continuity by transversality.} Now, we use the continuity proved in Step~2, and another suitable transversality argument to show that $0 \in \R^n$ is a Lebesgue point of $\pbf^e u$ in $\R^n$. This part of the proof follows by verbatim the arguments contained in the proof of Theorem~6.4 in~\cite{ambrosio1997fine-properties}. Since $0 \notin \Theta_u$, the quasi-continuity established in \cite[Prop.~1.2]{anna} implies that there exist a family of reals $\{d_r\}_{r \in \R} \subset V$ such that
	\begin{align}\label{eq:aquasi}
	\frac 1{r^n}\int_{Q_r} |u^{\mathbf v}(y) - d_r^{\mathbf v}| \dd y \le
 \frac{1}{r^n} \int_{Q_r}|u- d_r|= \BigO(r) \quad \forall \, \mathbf v \in (V_e)^*,
	\end{align}
	where we have used the short-hand notation $d_r^{\mathbf v} \coloneqq \dpr{\mathbf v, d_r}$. Notice that a priori this is not enough to ensure Lebesgue continuity since the sequence $\{d_r\}$ may not be convergent as $t \to 0^+$. Applying Fubini's Theorem we further obtain
	\[
		\lim_{r \to 0^+} \,  \frac 1{r^{n-1}} \int_{C_r} \bigg( \frac{1}{r} \int_{-r}^r  |u^{\mathbf v}(y + t\xi) - d_r^{\mathbf v}| \dd t   \bigg) \dd \Hcal^{n-1}(y) = 0,
	\]
	for all $\mathbf v  \in (V_e)^*$.
	Now, fix $r > 0$ an arbitrary radius. A standard measure theoretic argument applied to the previous estimate yields the existence of $\rho = \rho(r) \in (r/2,r)$ such that
	\begin{equation}\label{eq:select}
		\frac{1}{r^{n-1}} \int_{C_r} |\tilde u^{\mathbf v}(y + \rho \xi) - d_r^{\mathbf v}| + |\tilde u^{\mathbf v}(y - \rho \xi) - d_r^{\mathbf v}| \dd \Hcal^{n-1}(y) \le 2\BigO(r).  
 	\end{equation}

	We know that $v$ is approximately continuous at $0$ when restricted to $\pi_\xi$ (cf. Step~2). Let us write $d \coloneqq \tilde {v}(0)$ to denote its Lebesgue point at $0 \in \pi_\xi$. Let $j \in \{1,\dots,\ell\}$ where $\ell$ is the integer in~\eqref{eq:span}. Similarly to the previous step, we find non-zero covectors $g_j \in \ell_{f_j}$ with $(\xi \pm \eta_j,e \pm g_j) \in \partial\sigma(\Acal)$ so that the difference $[\tilde u^{g_j}(y + \rho \xi) - \tilde u^{f_j}(y - \rho \xi) - 2d^{g_j}]$ may be written as a sum of ``{good transversal}'' differences:
	\begin{align*}
		\tilde u^{g_j}(y + \rho \xi) - & \tilde u^{g_j}(y - \rho \xi) - 2d^{g_j} \\
		& = \tilde u^{g_j + e} (y + \rho \xi) - \tilde u^{g_j + e} (y - \rho \eta_j)  \\
		& \quad + \tilde u^{g_j - e} (y - \rho \xi) - \tilde u^{g_j - e} (y - \rho \eta_j) \\
		& \qquad + \tilde u^e(y - \rho \xi) - \tilde u^e(y + \rho \xi) + \underbrace{2[\tilde u^{g_j}(y - \rho \eta_j) - d^{g_j}]}_{\text{app. cont. difference over $\pi_\xi$}},
	\end{align*}
	Integrating both sides over $C_r$, we are now in position to use a suitable transversality reasoning as in the previous step (see also~(6.1) in~\cite{ambrosio1997fine-properties}) where we exploit that $(\eta_j,g_j) \in \partial\sigma(\Acal)$: 
	\begin{align*}
		\frac 1{r^{n-1}} \int_{C_r} & |\tilde u^{g_j}(\rho + t\xi) - \tilde u^{g_j}(y-\rho\xi) - 2d^{g_j}| \dd \Hcal^{n-1}(y)  \\
		& \lesssim \frac{1}{r^{n-1}}|\Acal(Q_{2r})| + \frac{1}{r^{n-1}} \int_{C_r} |\tilde u^{g_j}(y - \rho \eta_j) - d^{g_j}| \dd \Hcal^{n-1}(y).
	\end{align*}
	Then, the continuity of $p_e u$ at $0 \in \pi$ and the fact that
	 $0 \notin \Theta_u$ imply
	\[
		\lim_{r \to 0^+} \frac 1{r^{n-1}} \int_{C_r} |\tilde u^{g_j}(\rho + t\xi) - \tilde u^{g_j}(y-\rho\xi) - 2d^{g_j}| \dd \Hcal^{n-1}(y) = 0. 
	\]
	This estimate and~\eqref{eq:select} imply that 
	$d_r^{g_j}$ converges to $d^{g_j}$ as $r \to 0^+$. Since $\{g_j\}_{j = 1}^\ell$ spans $(V_e)^*$, we deduce that $\pbf^e[d_r] \to \pbf^e[d]$ as $r \to 0^+$. We conclude from~\eqref{eq:aquasi} that
	\[
		\lim_{r \to 0^+} 
		\frac 1{r^{n-1}}\int_{C_r} |\pbf^eu(y) - \pbf^e[d]| \dd y = 0,
	\]
	which proves that $0$ is indeed a Lebesgue point of $\pbf^e u$ at $0 \in \R^n$. This finishes the proof.
\end{proof}

\begin{corollary}\label{cor:F}
	 $|\Acal u|$-almost every $x \in \Omega \setminus  \Theta_u$ is a Lebesgue point of $u$. In particular,
	\[
		|\Acal u|(S_u \setminus J_u) = 0.
	\]
\end{corollary}
\begin{proof} With all the algebraic structure we have developed now, the proof should follow closely the proof of~\cite[Proposition~6.8]{ambrosio1997fine-properties}.
	
	 We may assume that $n \ge 2$ for otherwise the proof follows from the one-dimensional classical $\BV$-theory. Since the claim is local, we may as well assume that $u \in \BV^\Acal(\R^n)$. Our proof is based on an induction argument over the dimension $n$. We assume that continuity assertion holds for all elliptic operators satisfying the rank-one property that act on spaces of functions of $(n-1)$-variables, and we prove the result for  operators acting on functions spaces of $n$-variables. The step of induction is clear since the case $n - 1 = 1$ is covered by the $\BV$-theory.
	 
	 \emph{Claim 1.} Let us write $\mu = \Acal u$ and $\mu_\xi^e \coloneqq \dpr{w,\mu}$ where $g_\Acal(w) = \xi \otimes e$ so that $\mu_\xi^e = {\partial_\xi u^e}$. The first task will be to show that $|\mu_\xi^e|$-a.e. point in $\Omega$ must either be contained in $\Theta_u$ or be a Lebesgue point of $u^e$. To this end, let us choose a pair $(\eta,f) \in \partial\sigma(\Acal)$ satisfying $\eta \in \pi_\xi$ and $f \in (V_e)^*$. In particular, this ensures that
	\begin{equation}\label{eq:coo}
		e \in (\pbf^f[V])^* = (V_f)^*, \quad |\mu_{\xi}^e| \lesssim |\pbf_\eta^f[\mu]|.
	\end{equation}
The codimension-one slicing tells us that $|\pbf_\eta^f[\mu]|$ is concentrated on points of the form $y + z \eta$, where $y \in \pi_\eta$ and $z \in S$ for some $S \subset \ell_\eta$ with $\Hcal^1(\ell_\eta \setminus S) = 0$. Moreover, the slices $\pi_\eta \ni y \mapsto v_z(y) \coloneqq \pbf^{f} u(y + z)$  are well-defined and all belong to $\BV^{\Bcal}(\pi_\xi)$ for every $z \in S$, where $\Bcal = \Acal_\eta^{f} \mres \pi_\xi$ is elliptic and satisfies the rank-one property. The hypothesis of induction then ensures  that, at every $z \in S$, $|\Bcal  v_z|$-almost every point in $\pi_\eta$ has one-sided Lebesgue points (with respect to some direction). The slicing identity
\[
|\pbf_\eta^{f}[\mu]|(B) = \int_S |\Bcal v_z|(B_{z}) \dd \Hcal^{1}(z).
\]
and Lemma~\ref{lem:induction} further imply that every such point either belongs to $\Theta_u$ or it must be a Lebesgue point of $\pbf^f \circ u$ in $\R^n$.  In sight of~\eqref{eq:coo} this shows that 
\[
	\text{$|\mu_\xi^e|$-almost every $x \in \Omega \setminus \Theta_u$ is a Lesbegue point of $u^e$}.
\]
This proves the first claim.

The next step is to use this property of $u$ and $\Theta_u^\complement$ to actually lift the Lebesgue continuity to all the coordinates of $u$, at $|\mu|$-almost every $x \in \Omega \setminus \Theta_u$. In light of the rank-one property, it shall be enough to show that $u^e$ is Lebesgue continuous at $|\mu_\eta^f|$-almost every point in $\Omega \setminus \Theta_u$, where $(\eta,f)$ is an arbitrary pair in $\partial\sigma(\Acal)$. Clearly, we may assume that $f \notin \ell_e$ for otherwise the claim above implies the desired result. By a change of variables we may further assume that $f \in (V_e)^* = 0$. 
Let $\xi \in \pi_\eta$ be such that $(\xi,e) \in \partial\sigma(\Acal)$.  In view of the polarization result in Lemma~\ref{lem:polarization} we may find a  direction $\omega \in \ell_\xi$ and a non-zero coordinate $v \in \ell_e$ for which
\[
	(\eta + \omega,f  + v), (\eta - \omega,f - v) \in \partial\sigma(\Acal).
\] 
On the other hand, the linearity of (distributional) differentiation ensures that
\begin{align*}
\mu_{\eta \pm \omega}^{f \pm v} & =\phantom{:} \partial_{\eta \pm \omega} u^{f \pm v} \\
& =\phantom{:} \partial_{\xi} u^f \pm(\partial_\eta u^v + \partial_\omega u^f) + \partial_\omega u^v \\
& =\phantom{:} \mu_\eta^f  \pm(\partial_\eta u^v + \partial_\omega u^f) + \mu_\omega^v \\
& \eqqcolon \mu_\eta^f \pm \sigma + \mu_\omega^v.
\end{align*}
In particular the distribution $\sigma$ is a measure, which is absolutely  continuous with respect to $\Lambda$. let $\Ocal \subset \Omega$ be a full $|\mu|$-measure set where the density $\mu / |\mu|$ exists and consider the measure $\Lambda \coloneqq |\mu_\eta^f| + |\sigma| + |\mu_\omega^v| \ll |\mu|$. There exist Borel functions $\lambda_1,\lambda_2,\lambda_3 : \Ocal \to \R$ such that 
\[
	 \mu_\eta^f = \lambda_1 \Lambda, \quad \mu_{\omega}^v = \lambda_2 \Lambda, \quad \sigma = \lambda_3 \Lambda.
\]
We split $(\R^n - \Theta_u)$ into the  disjoint Borel sets 
\begin{gather*}
A_1 \coloneqq \set{x \in \R^n - \Theta_u}{\lambda_1(x) + \lambda_2(x) = 0},\\
A_2 \coloneqq \set{x \in \R^n - \Theta_u}{\lambda_1(x) + \lambda_2(x) \neq 0}.
\end{gather*}
On $A_1$ the total variation measures of $\mu_{\eta}^f$ and $\mu_\xi^v$ are identical and therefore (recall that $v \in \ell_e$) $u^e$ is Lebesgue continuous $|\mu_\eta^f|$-almost everywhere on $A_1$. 
 On the other hand, the measure $|\mu_{\eta + \omega}^{f + v}| + |\mu_{\eta - \omega}^{f - v}|$ is strictly positive in $A_2$. This implies that $|\mu|$-almost every point in $x\in A_2$ is either a Lebesgue point of $u^{f + v}$ or a Lebesgue point of $u^{f - v}$. By linearity, we conclude that  $u^e$ is Lebesgue continuous $|\mu_\eta^f|$-almost everywhere in $A_2$. This  shows that $u^e$ is Lebesgue continuous $|\mu_\eta^f|$-almost everywhere on $\R^n \setminus \Theta_u$, as desired. 
Lastly, since both $e \in V^*$ and $(\eta,f) \in \partial\sigma(\Acal)$ were chosen in an arbitrary manner, the rank-one property implies that (recall that $J_u \subset \Theta_u$ and $|\Acal u|(\Theta_u \setminus J_u) = 0$)
\[
	|\Acal u|(S_u \setminus J_u) =  0.
\]
This finishes the proof.
\end{proof}

Arguing as in~\cite[Thm. 6.1]{ambrosio1997fine-properties} we obtain the following more general statement:
\begin{corollary}\label{cor:v}
	Let $\Acal$ be as in the previous Corollary and let $u,v$ be functions in $\BV^\Acal(\Omega)$. Then
	\[
		|\Acal u|(S_v \setminus J_v) = 0.
	\]
\end{corollary}

\begin{corollary}\label{cor:help}
	Let $u \in \BV^\Acal(\R^n)$ and let $(\xi,e) \in \partial\sigma(\Acal)$ be a non-trivial pair. Set $\Bcal \coloneqq \Acal_\xi^e \mres \pi_\xi$, and, for a vector $z \in \ell_\xi$, write $v_z : \pi_\xi \to V_e$ to denote the function defined by
	\[
	v_z(y) \coloneqq \pbf^e u(z + y).
	\] 
	There exists a full $\Hcal^1$-measure set $Z \subset \ell_\xi$ with the following property: if  $z \in Z$, then
	\begin{enumerate}
		\item $\Hcal^{n-1}$-almost every point in $z + \pi_\xi$ is a Lebesgue point of $u$,
		\item $v_z \in \BV^\Bcal(\pi_\xi)$,
		\item $|\Bcal^j v_z|(\pi_\xi \setminus (\Theta_u)_z) = 0$. 
	\end{enumerate} 
\end{corollary}
\begin{proof}
	By Fubini's theorem, there exists a full $\Hcal^1$-measure set $Z_1 \subset \ell_\xi$ such that $\Hcal^{n-1}$-almost every point of $z + \pi_\xi$ is a Lebesgue point of $u$. On the other hand, the slicing on co-dimension one planes yields another full $\Hcal^1$-measure set $Z_2 \subset \ell_\xi$ where $v_z \in \BV^\Bcal(\pi_\xi)$. Set $Z \coloneqq Z_1 \cap Z_2$. Then $Z$ is a full $\Hcal^1$-measure set of $\ell_\xi$. {Moreover, since $\Bcal$ is elliptic and satisfies the rank-one property, $v_z$ has one-sided Lebesgue limits at $|\Bcal^j v_z|$-almost every $y \in \pi_\xi$ (for all $z \in Z$)}. Therefore, the first assertion of the previous lemma implies that, for every $z \in Z$, the function $v_z$ is Lebesgue continuous $|\Bcal v_z|$-almost everywhere in $\pi_\xi \setminus (\Theta_u)_z$. 
\end{proof}

	\section{Proof of the one-dimensional structure theorem} 
	
	We have now all the necessary tools to give a proof of the one-dimensional structure theorem (Theorem~\ref{thm:oned}).
	
	\begin{proof}[Proof of Theorem~\ref{thm:oned}] It suffices to show that
	\begin{equation}\label{eq:sliceacj}
	\left[\int_{\pi_{\xi}} D u_{y,\xi}^{e} \dd \Hcal^{n-1}(y) \right]^\sigma = \int_{\pi_{\xi}} D^\sigma u_{y,\xi}^{e} \dd \Hcal^{n-1}(y), \qquad \sigma = a,c,j,
	\end{equation}
	For ease of notation, let
	\[
	\nu = \int_{\pi_{\xi}} D u_{y,\xi}^{e} \dd \Hcal^{n-1}(y).
	\]
	Then, we have
	\[
	\nu^a - \int_{\pi_{\xi}} D^a u_{y,\xi}^{e} \dd \Hcal^{n-1}(y) = \int_{\pi_{\xi}} D^s u_{y,\xi}^{e} \dd \Hcal^{n-1}(y) - \nu^s.
	\]
	Fubini's Theorem yields that the left-hand side is absolutely continuous with respect to $\Leb^n$, meanwhile the right-hand side is concentrated on a set of $\Leb^n$-measure zero. Thus, both sides must vanish, and so we conclude that \eqref{eq:sliceacj} indeed holds for $\sigma = a$.
	
	Now, let us write $\nu^j \coloneqq \nu\mres J_u$ and $\nu^c \coloneqq \nu \mres (\Omega \setminus S_u)$ so that
	\begin{equation}\label{eq:slicesingular}
	\begin{split}
	\nu^j - \int_{\pi_{\xi}} D^j u_{y,\xi}^{e} \dd \Hcal^{n-1}(y) & = \nu^c - \int_{\pi_{\xi}} D^c u_{y,\xi}^{e} \dd \Hcal^{n-1}(y) \\
	(\nu^j - \mu^j & \triangleq  \nu^c - \mu^c).
	\end{split}
	\end{equation}
	Once again, we aim to show that both sides vanish. For a start, consider the restriction of both sides to the jump set $J_u$. We know that $J_u$ is countably $\Hcal^{n-1}$ rectifiable (see~\cite{nin,anna}), and hence an elementary measure-theoretic argument shows that for $\Hcal^{n-1}$-almost every point $y \in \pi_\xi$, the line $\set{y+t\xi}{t \in \R}$ and $J_u$ intersect at most on a countable set. 
	However, $D^c u_{y,\xi}^{e}$ is non-atomic, so it vanishes on this intersection. In conclusion, we have $\mu^c \mres J_u = 0$.
	Let us recall from Corollary~\ref{cor:iv} that $\Hcal^{n-1}(\Theta_u \setminus J_u) = |\Acal u|(\Theta_u \setminus J_u) = 0$. Therefore,
	\[
		\nu^c \mres (\Theta_u \setminus J_u) = \mu^c \mres (\Theta_u \setminus J_u) = 0.
	\]
%
%
%
%
%
Moreover, $\nu^c\mres J_u = 0$, since this is a component of $\Acal^c u \mres J_u \equiv 0$ in some direction,
	we conclude that the left-hand side of \eqref{eq:slicesingular} vanishes when restricted to $\Theta_u$. We also know that the essential support of $\nu^j$ is contained in $J_u$, so it remains to show that the other term $\mu^j$ is concentrated purely on  $\Theta_u$. Indeed, if this was the case, then  both sides must vanish identically, whence  we obtain $\nu^j \equiv \mu^j$ and $\nu^c \equiv \mu^c$ as desired.
	


	Let us prove that $\mu^j$ vanishes in $\Omega \setminus \Theta_u$. The proof we give here follows closely the $n$-dimensional induction-based proof in~\cite{ambrosio1997fine-properties}. The main difficultly, however, lies in circumventing the lack of a well-defined structure of elliptic operators satisfying the slicing property.

	\proofstep{Step~1. The step of induction $n = 2$.}  The measure $\mu^j$ is concentrated on slices $\ell_\xi + y$ where $y \in Y$ and $\Leb^1(\pi_\xi \setminus Y) = 0$. By Fubini's theorem, we may assume without loss of generality that, for all $y \in Y$, $\Hcal^1$-almost every point of $\R$ is a Lebesgue point of $v_y \coloneqq u_{y,\xi}^e$. Now, let $(\eta,f) \in \partial\sigma(\Acal_\xi^e \mres \pi_\xi)$ so that 
	\[
		(\xi,e) \in \partial\sigma(\Acal_\eta^f \mres \pi_\eta).
	\]
	Set $\Bcal \coloneqq \Acal_\eta^f \mres \pi_\eta$ and consider the isometry $\psi : \ell_\xi \to \R : t\xi \mapsto t$. By construction we have $\ell_\xi = \pi_\eta$ and we can find a positive constant $c > 0$ such that (since $n = 2$, we have $\ell_\eta = \pi_\xi$)
	\begin{equation}\label{eq:papi}
		|D u_{y,\xi}^e| \le c \psi_\# |\Bcal (\pbf^f u)_y| \quad \text{as measures over $\R$, for $\Leb^1$-a.e. $y \in \pi_\xi$}.
	\end{equation}
	Notice also that $(\pbf^f u)_y(\psi(t)) \coloneqq \pbf^f u(y + \psi(t)) = \pbf^f u_{y,\xi}(t)$ for all $t \in \R$. Thus, using Corollary~\ref{cor:hyper} we verify that 
	\begin{align*}
		\int_{\pi_\xi} |D u_{y,\xi}^e|  \dd \Hcal^1(y) & \le c \int_{\ell_\eta} \psi_\#|\Bcal (\pbf^f u)_y| \dd \Hcal^1(y),
	\end{align*}
	where the inequality holds in the sense of  (bounded) measures on $\R^2$. Therefore, up to adding a set of $\Leb^1$-measure zero to $Y$,
	we may assume that the measure on the right hand side is concentrated on slices of the form $\ell_\xi + y$, with $y \in Y$. Moreover, since $\Bcal$ is an elliptic on functions of one variable,  it must be that $(\pbf^f u)_y \in \BV(\R)$ for all $y \in Y$. Therefore, the one-dimensional $\BV$-theory implies that, for all $y \in Y$, the map $(\pbf^f u)_y$ has one-sided Lebesgue points for  $\psi_\#|\Bcal(\pbf^f u)_y|$-a.e. $t \in \R$ and all $y \in Y$. Applying Lemma~\ref{lem:induction} to $u$ and $(\eta,f)$ we thus infer from~\eqref{eq:papi} that $(\pbf^f u)_y \circ \psi$ is Lebesgue continuous
	at $|D u_{y,\xi}^e|$-a.e. $t \in \R$  such that $y+t\xi \notin \Theta_u$. Since also $e \in (V_f)^*$, we conclude that
	\[
		 \bigg(\int_{\pi_\xi} |D^j u_{y,\xi}^e| \dd \Hcal^{n-1} \bigg)\mres \Theta_u^\complement \equiv 0. 
	\]  
	 This proves that $\mu^j$ vanishes on $\Omega \setminus \Theta_u$ as desired.

	\proofstep{Step~3. The induction argument.} Let $n \ge 3$ and assume that the the conclusion of the Theorem holds for all operators satisfying the mixing condition in dimensions $m \le n- 1$. Let $(\eta,f) \in \partial\sigma(\Acal)$ be a pair satisfying 
	\[
		(\xi,e) \in \partial\sigma(\Acal_\eta^f \mres \pi_\eta). 
	\]
	Let $\Bcal \coloneqq \Acal_\eta^f \mres \pi_\eta$ and write $v_z(y) \coloneqq \pbf^f u(z + y)$. The slicing on co-dimension one planes implies the existence of a set $Z \subset \ell_\eta$ with $\Hcal^1(\ell_\eta \setminus Z) = 0$, and such that the slices $v_z$ belong to   $\BV^\Bcal(\pi_\eta)$ for all $z \in \ell_\eta$. Moreover, the hypothesis of induction applied to the operator $\Bcal$ and $v_z$ for all $z \in Z$ yield the identity  
	\begin{equation}\label{eq:estimate}
	\begin{split}
		I^\sigma  \coloneqq \int_{\ell_\eta} \bigg(\int_{\pi_\xi \cap \pi_\eta}  |D^\sigma (v_z)_{\xi,\tilde y}^e| \dd \Hcal^{n-2}(\tilde y)\bigg) \dd \Hcal^1(z)  \lesssim \int_{\ell_\eta} |\Bcal^\sigma v_z| \dd \Hcal^1(z)
	\end{split}
	\end{equation}
	for all $\sigma = a,c,j$. Observe that $I^\sigma$ is not immediately trivial since we have assumed that $n > 2$.
	
	 Next, recall that by construction  $e \in (V_f)^*$. This observation and Fubini's theorem imply the equality of functions 
	\[
		(v_z)_{\xi,\tilde y}^e \equiv (v_z)^e(\tilde y + \frarg\, \xi) \equiv (\pbf^f u)^e(z + \tilde y + \frarg \, \xi) \equiv u_{\xi,z + \tilde y}^e \quad\text{in $\BV(\R)$}
	\]
	 for $\Hcal^{n-2}$-almost every $\tilde y \in (\pi_\xi \cap\pi_\eta)$ and  every $z \in Z$. Therefore, applying Fubini's theorem and appealing to the classical structure theorem for functions of bounded variation we further deduce that
	\begin{equation}\label{eq:nomames}
		I^\sigma = \int_{\pi_\xi}  |D^\sigma u_{y,\xi}^e| \dd \Hcal^{n-1}(y) = |\mu^j|, \qquad \sigma \in \{a,c,j\}.
	\end{equation}
	On the other hand, from Corollary~\ref{cor:help} we find that $|\Bcal^j v_z|((\pi_\xi \setminus \Theta_u)_z) = 0$ for $\Hcal^1$-almost every $z \in \ell_\eta$. This proves that
	\[
		\bigg(\int_{\ell_\eta} |\Bcal^j v_z| \dd \Hcal^1(z) \bigg) \mres \Theta_u^\complement \equiv 0 \quad \text{as measures over $\R^n$}.
	\]
	To conclude we recall the estimate~\eqref{eq:estimate}, which in light of the previous equality of measure implies that $I^j \mres (\R^n \setminus \Theta_u)$ must be the zero measure. The sought assertion $\mu^j \mres (\R^n \setminus \Theta_u) \equiv 0$ follows from~\eqref{eq:nomames}.\end{proof}

\section{Proof of the fine properties statements} With all the analysis developed so far, we are now able to collect the proofs of the statements for the fine properties:
    
    \subsubsection*{Proof of Theorem~\ref{thm:2}} In light of Remark~\ref{rem:ya}, we shall only discuss the structure theorem decomposition and statements (iii)-(v). 
    The assertions contained in (iii)-(iv) follow directly from the results of Corollary~\ref{cor:iv}, while assertion~(v) follows from Corollary~\ref{cor:v}.
    As discussed in the introduction, one can trivially decompose $\Acal u$ into mutually singular measures 
    \[
    	\Acal u = \Acal^a u + \Acal^c u + \Acal^d u + \Acal^j u,
    \]
   	where $\Acal^d u = \Acal^s u \mres (S_u \setminus J_u)$. By virtue of~(v) with $v = u$, we deduce that $|\Acal^d u| \equiv 0$ whenever $\Acal$ is elliptic and satisfies the mixing property. This finishes the proof.\qed
   	
   	\subsubsection*{Proof of Theorem~\ref{thm:II}} The proof of this result follows directly from Theorem~\ref{thm:2} and the results contained in Section~\ref{sec:equations}.\qed

\section{Notions and proofs for operators of arbitrary order}\label{sec:k}

 The task of the following lines is to extend the slicing theory that we have already introduced for first order operators. As it has already been discussed in the introduction, due to the ellipticity assumption, we shall focus on the fine properties of the $(k-1)$\textsuperscript{th} gradients of functions in $\BV^\Acal$-spaces. Therefore, we also look for a slicing theory in terms of $\nabla^{k-1} u$ rather than $u$ itself. This leads us to consider a notion of $\rank_\Acal$ that extends the one for first-order operators and whose $\rank_\Acal$-one elements satisfy 
\[
	\dpr{w,\Acal u} = \partial_\xi(\dpr{E, \nabla^{k-1} u}) \eqqcolon \partial_\xi U^E
\] 
for some $\xi \in \R^n$ and some $E \in V^{k-1} \coloneqq V^*\otimes E_{k-1}(\R^n)$. We are naturally led  to define, at least formally, the $\rank_\Acal$-one vectors as those $w \in W^*$ such that  
\begin{equation}\label{eq:candidate}
\dpr{w,{\Abb^k}(\eta)v} = \dpr{\xi , \eta}\dpr{E , v \otimes^{k-1} \eta} \qquad \text{for all $\xi \in \R^n$, $E \in V \otimes E_{k-1}(\R^n)$} 
\end{equation}
for some $\xi \in \R^n$ and $E \in V^{k-1}$.
\begin{definition}
	Let $w \in W^*$. We say that $w \in \Acal^\otimes_1$ if and only if $w$ satisfies~\eqref{eq:candidate} for some $(\xi,E) \in \partial\sigma(\Acal)$.
\end{definition}
Notice that, such covectors $w$, are precisely those with the property that 
\[
\dpr{w, \Abb^k(\eta)\widehat u(\eta)} 
= \dpr{\xi \cdot \eta} \dpr{E , \Fcal(\nabla^{k-1} u)(\eta)},
\]
so that inverting the Fourier transform we obtain the desired identity $\dpr{w,\Acal u} = \partial_\xi U^E$ for all $u \in \Scal(\R^n;V)$.
If slicing is to be useful, we need that $\Acal$ controls a sufficient number of partial derivatives $\partial_\xi(\;\;)^E$. 
Now, observe that if $\rank_\Acal(w) = 1$, then
\begin{equation}\label{eq:star}
	w \in \bigcap_{\eta \in \pi_\xi} \ker \Abb^k(\eta)^*.
\end{equation}
Hereby we deduce that the mixing property~\eqref{eq:mixk} is a necessary condition for the slicing of $\Acal u$ and $\nabla^{k-1}u$. The remaining question is whether this is also a sufficient condition. The first issue at hand is that $\Abb^k$ is no longer a linear map, but a $k$-homogeneous map. Our first step will be to deal with the $k$-homogeneity of $\Acal$: 

\begin{definition}[linearized symbol]Let $\Acal : \Crm^\infty(\R^n;V) \to \Crm^\infty(\R^n;W)$ be a $k$\textsuperscript{th} order partial differential operator. Let us recall that the principal symbol $\Abb^k$ is a $k$-homogeneous map $\Abb^k : \R^n \to W \otimes V^*$, and therefore there exists a (uniquely determined) linear map
	\[
	\cl{\Abb^k} : V \otimes E_k(\R^n) \to W
	\]
	satisfying
	\[
	\cl{\Abb^k}[v \otimes^k \xi] =  \Abb^k(\xi)[v] \quad \text{for all $\xi \in \R^n$, $v \in V$}.
	\]
\end{definition}
The main difference is that we are considering a linear map. 
Of course,  the cost to pay is that we are adding a considerable amount of $k$-order tensors. 
Now, we lift the algebraic linearization  to a linearization of $\Acal$, which will allow us to make use the better part of the theory we have developed for first-order operators. In order to preserve ellipticity, the idea is to add a $\curl$-operator to compensate for the addition of $(k-1)$\textsuperscript{th} order tensors occurring in the linearization process of $\Abb^k$\,:  

\begin{definition}[Linearized operator] We define the linearization of $\Acal$ to be the  first-order operator given by
$$d\Acal \coloneqq \cl{\Abb^k}(DU) \times \curl_{k-1}U, \qquad U : \R^n \to V_{k-1} \coloneqq V \otimes E_{k-1}(\R^n),$$
 where, for a positive integer $m$, 
	\[
	\curl_m U \coloneqq (\partial_i U_{\mathbf e_j + \beta}^\ell - \partial_j U_{\mathbf e_i + \beta}^\ell)_{i,j,\beta}, \qquad |\beta| = m-1, \, \ell = 1,\dots,\dim(V),
	\] 
	is the generalized curl operator on $V_m$-valued tensor fields.
\end{definition}

It is well-known (see~\cite[Example~10.3(d)]{FM99}) that $\ker \curl_m(\xi) = \set{v \otimes^{m} \xi}{v \in V}$. Therefore, by definition, we have
\begin{equation}\label{eq:sss}
d\Abb(\xi)[v \otimes^{k-1} \xi] = (\Abb^k(\xi)[v],0) \qquad \text{for all $\xi \in \R^n$, $v \in V$.}
\end{equation}
This shows that $d\Acal$ indeed behaves as a differential of $\Acal$, in the sense that it is a linear operator that acts on $(k-1)$-order gradients as $\Acal$, i.e., 
\begin{equation}\label{eq:nomas}
d\Acal(\nabla^{k-1} u) = ({\Acal}u,0), \qquad u : \R^n \to V.
\end{equation}

The linearization satisfies the following crucial properties:

\begin{lemma}\label{prop:10}Let $\Acal : \Crm^\infty(\R^n;V) \to \Crm^\infty(\R^n;W)$ be a $k$\textsuperscript{th} order homogeneous
	
	 partial differential operator.
	The following holds:
	\begin{enumerate}
		\item $d\Acal$ is $\Kbb$-elliptic if and only if $\Acal$ is $\Kbb$-elliptic for $\Kbb = \C,\R$,
		\item if $\Acal$ is elliptic, then $u \in \BV^\Acal(\Omega)$ if and only if $\nabla^{k-1} u \in \BV^{d\Acal}(\Omega)$,
		\item $\set{w \in W^*}{\text{$w$ satisfies~\eqref{eq:star}}} = \mathrm{Proj}_{W^*} (d\Acal)^\otimes_1 \subset \Acal^\otimes_1$,
		\item $\partial\sigma(\Acal) = \partial\sigma(d\Acal)$.
	\end{enumerate}
\end{lemma}

\begin{remark}[Non-stability of the mixing property under linearization]
	In general it is not true that if $\Acal$ is elliptic and satisfies~\eqref{eq:mixk}, then its linearization $d\Acal$ satisfies~\eqref{eq:mix}. In particular, we cannot expect the $k$\textsuperscript{th} order theory to follow trivially from linearization. 
\end{remark}

\begin{proof}[ Proof of Proposition~\ref{prop:10}] The proof of (1) and (2) is contained in~\cite[Sec.~5]{anna}. 
	
	Let us prove (3) and (4). Let $(w,h) \in (d\Acal)^\otimes_1$. By definition, we may find $(\xi,E)$ such that
	\[
		\dpr{w , \cl{\Abb^k}[M \odot \eta]} + \dpr{h , \curl(\eta)[M]} = \dpr{\xi,\eta}\dpr{E,M}.
	\]
	Taking $M = v \otimes^{k-1} \eta$ above, we deduce that 
	\[
		\dpr{w , \Abb^k(\eta)[v]} = \dpr{\xi,\eta}\dpr{E,v \otimes^{k-1} \eta}.
	\]
	This shows that $w \in \Acal^\otimes_1$ and therefore $w$ satisfies~\eqref{eq:star}. The same argument shows that $\partial\sigma(d\Acal) \subset \partial\sigma(\Acal)$. The challenging part is to show the other contention of the equality. Let us assume that $w$ and $(\xi,E)$ satisfy~\eqref{eq:star}. 
	Using the linearity of $\cl {\Abb^k}$ and a polarization argument, we find that
	\[
		\dpr{w , \cl{\Abb^k}[Q] }= 0 \qquad\text{for all $Q \in V \otimes E_{k}(\pi_\xi)$}.
	\]
	%
	Let $\{\xi,\zeta_2,\dots,\zeta_n\}$ be an orthonormal basis of $\R^n$ so that
	\[
		\set{e^\alpha \coloneqq (\alpha!)^{-1}(\otimes^{\alpha_1}  \xi) \odot (\otimes^{\alpha_2}\zeta_2) \odot \cdots \odot (\otimes^{\alpha_n} \zeta_n)}{|\alpha| = k}
	\]
	is an orthogonal basis of $E_k(\R^n)$. Notice that the orthogonal complement of $E_{k}(\pi_\xi)$ in $E_k(\R^n)$ is given by $\spn\set{e^\alpha}{\alpha_1 \ge 1, |\alpha| = k}$. 
	Then, by the representation of linear maps, we may find reals $\gamma_\alpha^\ell \in \R$ such that 	
	\[
		w \cdot \cl{\Abb^k}[M \odot \eta] = \sum_{\substack{\ell=1,\dots,N.\\ \alpha_1^\ell \ge 1,\\ |\alpha^\ell|=k}} \gamma_\alpha^\ell\dpr{\mathbf v^\ell \otimes e^\alpha,M \odot \eta}, \qquad M \in V_{k-1},
	\]
	where $\{\mathbf v^1,\dots,\mathbf v^N\}$ is any orthonormal basis of $V^*$. 
	Now, writing $\eta = \eta_1 \xi + \eta_2 \zeta_2 + \dots +\zeta_n \eta_n$, 
	and $(M)_{\ell,\beta} = m_\beta^\ell$, we find that
	\[
		M\odot \eta = \sum_{\substack{j=1,\dots,n,\\ \ell = 1,\dots,N,\\|\beta|=k-1}} \eta_j m_\beta^\ell(\mathbf v_\ell \otimes e^{\beta + \mathbf e_j}).
	\]
	This, in turn, yields the identity
	\[
		\dpr{\mathbf v^\ell \otimes e^\alpha,M\odot \eta} = \eta_j m_\beta^\ell \delta_{(\beta + \mathbf e_j)\alpha}, \qquad |\alpha| = k.
	\]
	Next, we exploit the representation of linear maps by elements in the complement of their kernel to deduce that
	\[
		\dpr{w , \cl{\Abb^k}[M \odot \eta]} = \sum_{\substack{j=1,\dots,n\\ \ell=1,\dots,N}}  \sum_{\substack{\beta + \mathbf e_j = \alpha,\\|\alpha| = k,\\
		\alpha_1 \ge 1}}  \eta_j m_\beta^\ell \gamma_{\alpha}^\ell = I_1 + \cdots + I_n,
	\]
	where 
	\[
		I_m(\eta)[M] \coloneqq \sum_{\substack{j=1,\dots,n\\ \ell=1,\dots,N}}\sum_{\substack{\beta + \mathbf e_j = \alpha,\\|\alpha| = k,\\
				\alpha_1 = m}}  \eta_j m_\beta^\ell \gamma_{\alpha}^\ell.
	\]
	The goal now is to prove that $I_m(\eta)[M] = \dpr{\xi,\eta}\dpr{E_m , M} + \dpr{h_m, \curl_{k-1}(\eta)[M]}$ for some constant tensors $E_m,h_m$. Since the argument for arbitrary $m$ is analogous to the one for $m = 1$, we shall focus on proving the sought representation for $I_1$.
	We may write
	\[
		I_1 = II_1 + III_1 \coloneqq \sum_{\substack{\ell=1,\dots,N,\\|\beta| = k-1}} \eta_1 m_\beta^\ell \gamma_{\beta + \xi}^\ell + \sum_{\substack{j=2,\dots,n\\ \ell=1,\dots,N}}\Bigg(\sum_{\substack{\xi + \omega + \zeta_j = \alpha,\\|\omega| = k-2,\\}}  \eta_j m_{\omega + \xi}^\ell \gamma_{\alpha}^\ell\Bigg).
	\]
	Now, the first term $II_1$ can be written as $\dpr{\xi,\eta}\dpr{E_1,M}$ where $E_1 \in E_{k-1}(\R^n)$ is expressed as in coordinates as $E_1 = (\gamma_{\beta+\mathbf e_1}^\ell)_{\ell,\beta}$.
	Thus, we only have to check that the second term on the right-hand side can be expressed as a linear combination of terms of $\curl_{k-1}(\eta)[M]$. First, let us calculate the $(k-1)$-curl operator on simple tensors for $j \ge 2$:
	\begin{align*}
		\curl_{k-1}(\zeta_j)[\mathbf v_\ell \otimes e^{\omega + \xi}] 
		& =\phantom{:} \mathbf v_\ell (\delta_{jp} \delta_{(\omega + \xi)(\beta + \zeta_q)} - \delta_{jq} \delta_{(\omega +\xi)(\beta + \zeta_p)})_{p,q,\beta}\\
		& =\phantom{:} \mathbf v_\ell \sum_{\substack{p,q = 1,\dots,n\\|\beta| = k-2}} \zeta_j \wedge (\zeta_q \odot e^{\omega + \xi - \zeta_q}) + \zeta_j \wedge (\mathbf e_p \odot e^{\omega + \xi - \zeta_p}) \\
		& \eqqcolon \mathbf v_\ell \otimes \mathbf m_{j,\omega}.
	\end{align*}
Notice that $|\mathbf m_{j,\omega}| \ge 1$ for all $j =2,\dots,n$ and all $|\omega| = k-2$. This follows since $\mathbf \zeta_j \wedge \xi \odot e^\omega$ is a non-zero  tensor. Now, consider the tensor
	\[
		h_1 \coloneqq  - \sum_{\substack{j = 2,\dots,n,\\ \ell=1,\dots,N,\\
				 |\omega| = k-2,\,\omega_q \ge 1}} \gamma_{\omega + \xi + \zeta_j}^\ell \bigg(\frac{\mathbf v^\ell \otimes \mathbf m_{j,\omega}}{|\mathbf m_{j,\omega}|^2}\bigg).
	\]
By construction,  we discover that (here we use that $|\mathbf m_{j,\omega}| > 0$)
	\[
		h_1 \cdot \curl(\eta)[M] = - III_1.
	\]
	We have thus have found $h_1$ such that
	\[
		I_1(\eta)[M] + \dpr{h_1, \curl_{k-1}[M]} = \dpr{\xi,\eta}\dpr{E_1,M}.
	\]
	As mentioned beforehand, the process of showing that 
	\[
		I_m(\eta)[M] + \dpr{h_m, \cdot \curl_{k-1}[M]} =  \dpr{\xi,\eta}\dpr{E_m,M}
	\]
	is analogous and therefore the details are left to reader for its verification. Summing over $m = 1,\dots,k$ we find that
	\[
		\dpr{(w,h) , d \Abb(\eta)[M]} = \dpr{\xi,\eta}\dpr{\tilde E,M}
	\]
	where we have set $h \coloneqq h_1 + \cdots + h_k$ and $\tilde E \coloneqq E_1 + \cdots + E_k$.  This proves that $(w,h) \in (d\Acal)^{\otimes}_1$. It also implies that $\tilde E = E$, which shows that $\partial\sigma(\Acal) \subset \partial\sigma(d\Acal)$. 
	
	This proves properties~(3) and (4).
\end{proof}

\begin{proposition}
	Let $\Acal : \Crm^\infty(\R^n;V) \to \Crm^\infty(\R^n;W)$ be a $k$\textsuperscript{th} order homogeneous linear elliptic differential operator satisfying~\eqref{eq:mixk}. Then
	\[
		\proj_{\R^n} \partial\sigma(d\Acal) = \R^n, \qquad \spn \Big\{\proj_{V^{k-1}} \partial\sigma(d\Acal)\Big\} = V^{k-1}.
	\]
	Moreover, $d\Acal$ satisfies the polarization property contained in Proposition~\ref{lem:polarization}.
\end{proposition}
\begin{remark}
	In general and in the context of the previous assumptions, it does not hold that
	\[
		\proj_{V^{k-1}} \partial\sigma(d\Acal) = V^{k-1}.
	\]
	In particular, one \emph{cannot} expect the linearization $d\Acal$ to satisfy the rank-one property. As an example, consider the operator
	\[
		\Dcal^ku = (\partial_1^k u ,\dots,\partial_n^k u), \quad k \ge 3.
	\]	
	It is easy to verify that the only pure derivatives $\partial_\eta\partial^{k-1}_\xi$ controlled by $\mathscr D^k$ are the ones where $\eta = \xi$ are elements of the canonical axis. In particular, if $\xi\in \Sbf^{n-1}$ does not belong to the canonical axis, then the tensor $\otimes^{k-1} \xi$ cannot be the second coordinate of a spectral pair of $d\Acal$. 
\end{remark}
\begin{proof}
	The mixing property~\eqref{eq:mixk} and the previous lemma imply that there exists a non-trivial pair $(w,h) \in \partial\sigma(d\Acal)$. Let $(\xi,E) \in \R^n \times V^{k-1}$ be a spectral direction associated to this pair and consider the slice $\Bcal \coloneqq (d\Acal)_\xi^E \mres \pi_\xi$. 
	
	Invoking Proposition~\ref{prop:nontrivial} and Lemma~\ref{lem:sub} we find that $\Bcal$ is a non-trivial elliptic operator.  We claim that $\Bcal$ is the linearization of a $k$\textsuperscript{th} order elliptic operator from $V$ to $W$ with variables in $\pi_\xi$ that satisfies the mixing property~\eqref{eq:mixk}. First, notice that the projection $\pbf_\xi^E$ induces a projection $\pbf : W \to X$, where $X = \pbf_\xi^E[W\times \{0\}]$. Now, let us consider the operator $\Lcal$ that is associated to the principal symbol $$\Lbb^k(\zeta)[v] = \pbf\circ\Abb^k(\zeta)[v] \quad \text{for all $\zeta \in \pi_\xi$ and $v \in V$.}$$
	This construction conveys the identity $\cl{\Lbb^k} = \pbf \circ \cl{\Abb^k}$, whereby we obtain
	\[
		d\Lbb = (\pbf\circ \cl{\Abb^k \mres \pi_\xi},\curl_{k-1}).
	\]
	Testing this identity with ($\curl_{k-1}$)-free tensors yields
	\begin{align*}
		d\Lbb(\zeta)[v\otimes^{k-1}\zeta] & = (\pbf \circ \Abb^k(\zeta)[v],0) \\
		&  = \pbf_\xi^E(d\Abb(\zeta)[v \otimes^{k-1} \zeta]) \\
		& = \Bbb(\xi)[v \otimes^{k-1}\zeta].
	\end{align*}	
	This proves that $\Bbb$ is indeed the linearization of $\Lcal$ (that $\Lcal$ is elliptic follows from the fact that $\Bcal$ is elliptic). Moreover, by a similar argument to the one used to prove point~(4) of Lemma~\ref{lem:sub}, one can show that $\Lcal$ also satisfies the mixing property~\eqref{eq:mixk} as an operator with variables in $\pi_\xi$. This proves the claim.
	
	Thus, $\Lcal$ satisfies the very same assumptions that $\Acal$ in the hypotheses of this lemma. In other words, we can iterate the slicing until we find a one-dimensional (elliptic) slice. This yields (notice that $n \le \dim(V_{k-1})$)
	\begin{enumerate}
		\item an orthonormal basis $\{\xi_1 \coloneqq \xi,\xi_2,\dots,\xi_n\}$ of $\R^n$,
		\item a basis $\{E_1,\dots,E_{n-1},E_{n},\dots,E_r\}$ of $V_{k-1}$
	\end{enumerate}
	such that $(\xi_i,E_i)$ is a spectral pair of the $i$\textsuperscript{th} slice of $d\Acal$ for all $i = 1,\dots,n-1$, and $\ell_{\xi_n} \times \spn\{E_n,\dots,E_{n+r}\}$ is a subset of the directional spectrum of the last (one-dimensional and elliptic) slice (and therefore containing a gradient). This, and point (4) in the previous lemma yield
	\[
	\spn\Big\{	\proj_{\R^n} \partial\sigma(d\Acal) \Big\}= \R^n, \qquad \spn \Big\{\proj_{V^{k-1}} \partial\sigma(d\Acal)\Big\} = V^{k-1}.
	\]		
	The iteration can be re-engineered in a way that, for any distinct $p,q \in \{1,\dots,n\}$, we slice with respect to $(\xi_i,E_i)$ for all $i \notin\{p,q\}$. This gives a slice of $d\Acal$  defined on the $2$-plane $\{\xi_p,\xi_q\}$. At this point of the proof, one may follow by verbatim the arguments in the proof of Proposition~\ref{lem:polarization} to show that $\spn\{\xi_p,\xi_q\} \subset \proj_{\R^n} \partial\sigma(d\Acal)$. Moreover, since the argument is independent of the initial pair $(\xi,E)$, we may also follow the same ideas of that proof to show the polarization property.
%
\end{proof}

\subsection{Proofs of the main results}
The idea will be to discuss, in chronological order, suitable versions of the main propositions, lemmas and theorems that are valid for first-order operators. Since most of the ideas remain largely similar, we shall mainly focus on those details and adaptations which are non-trivial.  

\emph{The slicing theorem.} If $(w,h) \in (d\Acal)^1_\otimes$ is a $\rank_{d\Acal}$-one tensor with an associated spectral pair $(\xi,E) \in \partial\sigma(d\Acal)$, then
$\dpr{w , \Acal u} = \dpr{(w,h) , d\Acal(\nabla^{k-1} u)} = \partial_\xi (\dpr{E , \nabla^{k-1} u})$. 
Invoking Proposition~\ref{lem:1} and~\eqref{eq:nomas} we get
\[
\dpr{w , \Acal u}= \int_{\pi_\xi} D(\nabla^{k-1} u)_{y,\xi}^E \dd \Hcal^{n-1}(y),
\]
\[
|\dpr{w , \Acal u}| = \int_{\pi_\xi} |D(\nabla^{k-1} u)_{y,\xi}^E| \dd \Hcal^{n-1}(y),
\]
as long as $u \in \BV^\Acal(\R^n)$. The same holds for $\Omega$ instead of $\R^n$ when $d\Acal$ is complex-elliptic, or equivalently, that $\Acal$ is complex-elliptic (see Proposition~\ref{prop:10}).  
Applying the rules for orthogonal complements we find that the mixing condition~\eqref{eq:mixk} is equivalent to the existence of a family $\{w_1,\dots,w_M\}$ spanning $(W_\Acal)^*$ and satisfying~\eqref{eq:star}. Invoking Lemma~\ref{prop:10}, we find that these covectors are, in fact, elements of $\Acal^\otimes_1$ that can be extended to elements of $d\Acal^\otimes_1$ by adding an appropriate coordinate.  Hence, the same concluding argument in the proof of Theorem~\ref{thm:structure} serves just as well as a proof for Theorem~\ref{thm:slicek}.

\emph{Dimensional estimates and fine properties I.} The proof of Corollary~\ref{thm:dim} follows from the first conclusion of the previous lemma and the same geometric argument used in the proof of Corollary~\ref{cor:dimensional1}. As a consequence, we also obtain a suitable version of Corollary~\ref{cor:iv} for higher order operators ---thus proving that $|\Acal u| \ll \Ical^{n-1}$, that $\Hcal^{n-1}(\Theta_u \setminus J_U) = 0$, and that $\Acal^c u$ vanishes on $\Hcal^{n-1}$ $\sigma$-finite sets.

\emph{Analysis of Lebesgue points on $\Theta_u$ (Lemma~\ref{lem:induction}).} The basis of the one-dimensional structure theorem and also of the fact that $|\Acal u|(S_u\setminus J_u)$, in the first-order case, hinges on on the first and second statements of Lemma~\ref{lem:induction} respectively. The first statement relies purely on slicing and the polarization properties, which we have now established in Proposition~\ref{prop:10} in the general case. The second statement requires slightly more: firstly, we need that $d\Acal$ is complex-elliptic in order to be able to use the quasi-continuity property at $\Theta_u$ (see~\cite[Prop.~1.2]{anna}); this is covered by the properties of the linearization and the assumption that $\Acal$ is complex-elliptic. The other necessary tool is that there are sufficient spectral $V^{k-1}$-coordinates to span $V^{k-1}$, which we have also established in Proposition~\ref{prop:10} for the general case ---in despite that $d\Acal$ may not satisfy the rank-one property. 

\emph{One-dimensional structure theorem  and fine properties II.}  The one-dimensional structure theorem on one-dimensional $BV$-sections follows by using $d\Acal$ and the identity $d\Acal U = (\Acal u,0)$ from the analysis of Lebesgue points and the aforementioned version of Corollary~\ref{cor:iv}. The latter also conveys a suitable version of Corollary~\ref{cor:F}, which then proves that $|\Acal u|(S_U \setminus \Theta_u) = 0$, or equivalently, that $\Acal^d u \coloneqq \Acal^s u \mres (S_U \setminus J_U)\equiv 0$. This covers the fine properties of the Cantor part $\Acal^c u$. The characterization of $\Acal^a u$ follows from the use of the linearization properties and the existing theory (see Remark~\ref{rem:ya}).
Lastly, the  characterization of $\Acal^j u$ follows from the results contained in~\cite{anna}. 
This proves Theorem~\ref{thm:k}.\qed

   \section{Applications}\label{sct:ex}
In this section we review some well-known operators that satisfy our main assumptions. In particular, we revise the details of some interesting cases, which are not covered by the $\BV$-theory.

\subsection{Gradients} The gradient operator
	\[
	D u = (\partial_1 u, \dots, \partial_n u), \qquad u: \R^n \to \R^N,
	\]
	is complex-elliptic operator. Indeed,  the symbol associated to $D$ is simply
	$D(\xi)[a] = a \otimes \xi$,
	which has no complex non-trivial zeros. Clearly, $D$ also satisfies the mixing property since $D^\otimes_1 = \Irm_D = \setn{a \otimes \xi}{\xi \in \R^n, a \in \R^N}$, and, in particular, $\partial\sigma(D) = \R^n \times \R^N$.

\subsection{Higher Gradients} The $k$\textsuperscript{th} gradient operator
	\[
	D^k u = \bigg(\frac{\partial^k u^j}{\partial x_{i_1}\cdots\partial x_{i_k}}\bigg)_{i_1,\dots,i_k}^{j}, \qquad u : \R^n \to \R^N,
	\]
	is complex-elliptic since its symbol is given by $D^k(\xi)[a] = a \otimes^k \xi$. It also satisfies the rank-one property or the mixing condition since every element in the image cone is clearly rank-one. 
	
	If $u$ belongs to the space  $$\BV^k(\R^n;\R^N) \coloneqq \setn{u \in \Lrm^1(\R^n)}{D^ku\in \Mcal(\R^n;\R^N \otimes E_{k}(\R^n))},$$ then $\nabla^{k-1} u \in \BV(\R^n;\R^N \otimes E_{k-1}(\R^n))$. The classical $\BV$-theory implies that $\nabla^{k-1} u$ and $D^k u$ satisfy the fine properties. Moreover, the structure theorem takes the form
	\begin{align*}
		D^k u & = \nabla^k u\, \Leb^n \, + \,  D^c(\nabla^{k-1} u)  \, + \,  \llbracket \nabla^{k-1} u\rrbracket  \, \otimes^{k} \nu_u \, \Hcal^{n-1} \mres J_{\nabla^{k-1} u},
	\end{align*}
	where $$\llbracket \nabla^{k-1} u \rrbracket \coloneqq \dprb {(\nabla^{k-1} u)^+ - (\nabla^{k-1} u)^- , \underbrace{\nu_u,\dots,\nu_u}_{\text{$(k-1)$-times}}} \in \R^N.$$

The following example is particularly interesting, since it does not follow from the $\BV$-theory. It says that it suffices to control the pure derivatives $\partial_1^k,\dots,\partial_n^k$ in order to deduce a structure theorem and fine properties for all the lower order derivatives:

\subsection{Fine properties of ${\mathscr{BV}^k}$-functions} Consider the diagonal of the $k$\textsuperscript{th} gradient 
	\[
	u \mapsto \mathscr D^k \coloneqq \mathrm{diag}(D^k u) = (\partial^k_1 u,\dots,\partial^k_n u), \quad u : \R^n \to \R.
	\]
	The principal symbol of $\mathscr D^k$  given by the map $\xi \mapsto  (\xi_1^k,\dots,\xi_n^k)$, whence we verify that $\mathscr D^k$ is complex-elliptic. Moreover, $\mathscr D^k$ satisfies~\eqref{eq:mixk}. Indeed, every element of the canonical basis $\{\mathbf e_1,\dots,\mathbf e_n\}$ of $\R^n$ is a $\rank_{\mathscr D^k}$-one tensor. 
	
	We conclude that if $u$ belongs to the space $$\mathscr{BV}^k(\R^n)  \coloneqq \setn{u \in \Lrm^1(\R^n)}{\mathscr D^ku \in \Mcal(\R^n;\R^n)},$$ then
	$\nabla^{k-1} u $ is integrable and approximately differentiable,
	the jump set  of $\nabla^{k-1}u$ is countably $\Hcal^{n-1}$-rectifiable, 
	$|\mathscr D^k u| \ll \Ical^{n-1} \ll \Hcal^{n-1}$, and 
	$\mathscr D^k u$ decomposes into its absolutely continuous, Cantor, and jump parts as
	\begin{align*}
		\mathscr D^k u & = \mathrm{diag}{(\nabla^k u)} \, \Leb^n  \, + \, 
		(\mathscr D^k)^s u \mres (\R^n \setminus S_{\nabla^{k-1} u}) \\
		& \qquad \qquad + \,  \llbracket \nabla^{k-1} u\rrbracket \, \big((\nu_1)_1^k,\dots,(\nu_u)_n^k\big) \, \Hcal^{n-1} \mres J_{\nabla^{k-1} u}.
	\end{align*}
	Moreover,  $\nabla^{k-1} u$ satisfies the fine properties $(i)$-$(v)$ contained in Theorem~\ref{thm:k}. 

\subsection{Fine properties of $\BD$-functions} For vector-valued map $u: \Omega \subset \R^n \to \R^n$ we define its symmetric gradient
	\begin{align*}
		Eu & \coloneqq \frac12 (Du + Du^t),
	\end{align*}
	which takes values on the space $E_2(\R^n)$ of symmetric bilinear forms of $\R^n$. One readily verifies that $E$ is elliptic since 
	$\mathrm I_E = \set{a \odot \xi}{a,\xi \in \R^n}$. We have $E^\otimes_1 = \set{\xi \otimes \xi}{\xi \in \R^n}$ and $\partial\sigma(E)=\set{(\xi,\xi)}{\xi\in \R^n}$. 
	A standard polarization argument shows that the family $\set{\xi \otimes \xi}{\xi \in \R^n}$ is a spanning set of $E_2(\R^n)$, which further implies that that $E$ has the rank-one property. The structure theorem in $\BD(\R^n)$, which is well-known, reads
	\[
	Eu = \sym(\nabla u) \, \Leb^n \, + \, E^s u \mres (S_u \setminus J_u) \, + \, \llbracket u \rrbracket \odot \nu_u \, \Hcal^{n-1} \mres J_u.
	\]

More generally, one may consider the symmetrization of the gradient of a symmetric $k$-tensor field:

\subsection{Fine properties of $\BD^k$-functions}
	Let $k \in \N$. For a symmetric $k$-tensor $v \in E_k(\R^n)$, we define the operator with symbol $E^k(\xi)[v]$ satisfying
	\[
	E^k(\xi)[v]a = \sym^{k+1}(\xi)[v](a_1,...,a_{k+1}) \coloneqq \frac{1}{(k+1)!}\sum_{\sigma \in S_{k+1}} (\xi\cdot a_{\sigma(k+1)})v(a_{\sigma(1)},...,a_{\sigma(k)}),
	\]
	for all $(k+1)$-tuples $\R^n$-vector-fields $(a_1,\dots,a_{k+1})$.
	Thus, for a $k$-tensor-field $u : \R^n \to E_k(\R^n)$, we have $E^ku = \sym^{k+1}(D u)$. 
	We verify that this operator is complex-elliptic by the pointwise definition: let $\xi \in \C^n$ be a non-zero vector and let $v \in E_{k}(\C^n)$. Then $\sym^k(\xi)v = 0$ implies 
	$\sym^{k+1}(\xi)[v](a,...,a) = 0$ for each $a \in \C^n$. This however implies
	$v(a,\dots,a) = 0$ for all $a \in \C^n$ with $a \cdot \xi \neq 0$.
	Now, consider any $b \in \pi_\xi$, and take $a = b + \xi$. Applying $v$ to this choice of $a$, we see by a polarization argument that $v$ vanishes on $E_{k+1}(\C^n)$  and hence $v=0$. This shows that $E^k$ is complex-elliptic. 
	Moreover,
	\[
	({E^k})^\otimes_1 = \set{\otimes^{k+1} \xi}{\xi \in \R^n}, \qquad \partial\sigma(E^k) = \set{(\xi,\otimes^{k} \xi)}{\xi \in \R^n}.
	\] This family of $k$-tensors can be seen to generate $E_{k+1}(\R^n)$ and therefore $E^k$ satisfies the rank-one property. Thus, for a function in the space $$\BD^k(\R^n) = \setn{u \in \Lrm^1(\R^n;E_k(\R^n))}{E^ku \in \Mcal(\R^n;E_{k+1}(\R^n))},$$ the structure theorem takes the form
	\[
	Eu = 	\sym^k(\nabla u) \, \Leb^n \, + \, (E^k)^s u \mres (S_u \setminus J_u) \, + \, \sym^k(\llbracket u \rrbracket \otimes \xi) \, \Hcal^{n-1} \mres J_u
	\]
	and furthermore $u$ satisfies the fine properties established in Theorem~\ref{thm:2}.


\section{Counterexamples}\label{sec:examples}
We review a number of well-known differential operators, which fail to satisfy the main mixing property.  
\subsection{Insufficiency of complex-ellipticity} The following two examples show that complex-ellipticity is not a sufficient condition  to ensure the validity of the mixing property~\eqref{eq:mix}:
\begin{example}[Deviatoric operator]\label{ex:dev} Consider the operator  that measures the shear part of the symmetric gradient:
	\[
	Lu = Eu - \frac{\Div(u)}{n}\id_{\R^n}, \qquad u:\R^n \to \R^n.
	\]	
	This operator satisfies the following structural properties:
	\begin{enumerate}
		\item $L$ is elliptic,
		\item $L$ is complex-elliptic for $n \ge 3$,
		\item $L$ does not satisfy the rank-one property for all $n \ge 2$ and its $\rank_\Lbb$-one cone is trivial, i.e., $\Lbb^\otimes_1 = \{0\}$.
		In particular, $L$ does not satisfy~\eqref{eq:mix}.
	\end{enumerate}
	\end{example}	

\begin{proof} The principal symbol is given by $\Lbb(\xi)[a] = a \odot \xi - n^{-1} (a \cdot \xi) \id_{\R^n}$. Therefore, $\Lbb(\xi)[a] = 0$ if and only if  
	$n(a \odot \xi) = (a \cdot \xi) \id_{\R^n}$. When $(a \odot \xi)$ is non-trivial, it is either a rank-one matrix or a rank-two matrix with eigenvalues of discordant sign. This shows that $L$ is elliptic. That $L$ is complex-elliptic follows from the fact that it has a finite dimensional kernel if and only if $n \ge 3$.
	
	Notice that $Lu$ takes values on the space $\sym_0(n)$ of trace-free symmetric $(n\times n)$-matrices. The spectral theorem ensures that $\sym_0(n)$ contains no rank-one elements, whereby we conclude  that $\Lbb^\otimes_1 = \{0\}$. 
\end{proof}


\begin{example}\label{ex:cauchy}
	Let $\Bcal : \Crm^\infty(\R^{n+1};\R^{N+1}) \to \Crm^\infty(\R^n;\R^{N + n+1})$ be the operator associated to the principal symbol
	\[
		\Bbb(\xi_0,\dots,\xi_n)[v_0,\dots,v_N] \coloneqq \bigg(\sum_{r+s = m} \xi_r v^s\bigg)_m, \qquad m=0,1,\dots,n+N.
	\]
	Then,
	\begin{enumerate}
		\item $\Bcal$ is complex-elliptic
		\item $\Bcal$ does not satisfy the rank-one property for all $n + N \ge 3$. 
	\end{enumerate} 
\end{example}
\begin{proof} A symmetry argument shows that there is no loss of generality in assuming that $n \ge N$ (otherwise, we simply reverse the roles of $\R^n$ and $V$).
	We only show the failure of the mixing property since the complex-ellipticity follows directly from a simple induction argument. The associated principal symbol of the operator satisfies
	\[
	(\Bbb(\xi)[v])_m = z_m(\xi,v) \coloneqq \sum_{r + s = m} \xi_r v_s = \sum_{r + s = m} (\mathbf e_r \cdot \xi)(\mathbf v_s \cdot v),
	\]
	where $\{\mathbf e_0,\dots,\mathbf e_n\}$ are $\{\mathbf v_0,\dots\mathbf v_n\}$ are orthonormal basis of $\R^{n+1}$ and $V$.
	Therefore, $z_m$ may be identified (under a suitable isometry) with a matrix of the form
	\[
	\tilde z_m = \begin{bmatrix}
		& & & & & 1 & &  \\
		& & & & 1 &    & & \\
		& & & 1 & &   & & \\
		& & 1 & & & & & & \\
		& 1 & & &  & & & \\
		1 & &  &  &  & & &\\
		& &  & & & &  & \\
	\end{bmatrix} \in \R^{N+1} \otimes \R^{n+1},
	\]
	where the vector in the first row only has a non-zero coordinate in its $(m+1)$\textsuperscript{{th}} entry. Notice that $\rank(\tilde z_m) \ge m+1$. More generally, for a vector $P = (P_0,\dots,P_{p+q})$, we discover that the bi-linear form $P \cdot f_\Bcal$ may be identified (under the same isometry) with a matrix of the form
	\[
	\tilde P = \begin{bmatrix}
		a & b & c & d & e & f & g   \\
		b & c & d & e & f & g & h \\
		c & d & e & f & g & h & i \\
		d & e & f & g & h & i & j \\
		e & f & g & h & i & j & k \\
		f & g & h & i & j & k & l \\
		g & h & i & j & k & l & m
	\end{bmatrix},
	\]
	and accordingly $\rank(\tilde P) = 1$ only when considering multiples of $ (1,0,\dots,0)$, $(0,\dots,0,1)$, or $(1,1,\dots,1)$. Since clearly these three vectors do not span $\R^{n + N +1}$ (when $n + N \ge 3$), this shows that $\Acal$ does not satisfy the rank-one property in the conjectured range.
\end{proof}

\subsection{Non-canceling operators}The following are some relevant examples of elliptic operators that \emph{fail} to satisfy the cancellation property
\[
	\bigcap_{\xi \in \pi_\xi} \im \Abb^k(\xi) = \{0\},
\]
introduced by \textsc{Van Schaftingen} in~\cite{van-schaftingen2013limiting-sobole} to establish limiting Sobolev inequalities on $\BV^\Acal(\R^n)$. Every operator satisfying the mixing property~\eqref{eq:mixk} is clearly canceling Therefore, the operators discussed next, all  fail to satisfy the rank-one property (for more details we refer the reader to~\cite{anna,van-schaftingen2013limiting-sobole} and references therein):

%
\begin{example}\label{ex:div_curl}Let $n \ge 2$. The div-curl operator 
	\begin{align*}
		\Fcal_n u & \coloneqq (\Div \times \curl) u. 
	\end{align*}
	defined for vector-fields $u : \R^n \to \R^n$.
	%
\end{example} 

An equivalent formulation of the previous example comes from the \emph{del-bar} operator or the Cauchy-Riemann equations, as well as conformal gradients:

\begin{example}[Cauchy-Riemann equations] A function $u : \R^2 \to \R^2$ satisfies
	\[
	\bar \partial (u^1 + \mathrm i u^2) \coloneqq (\partial_1 - \mathrm i \partial_2) (u^1 + \mathrm i u^2) = 0
	\]
	if and only if $w(x + \mathrm iy) \coloneqq u^1(x,y) + \mathrm{i} u^2(x,y)$ is holomorphic. The $\bar \partial$-operator in two dimensions is equivalent to the div-curl operator: if we set $\psi = ( u^1,-u^2)$, then
	\[
	(\Div \times \curl) \psi = 0.
	\]
	
	\begin{example}[Conformal maps]The same conclusions apply to the differential inclusion
		\[
		Du(x) \in K \coloneqq \setBB{\begin{pmatrix}
				a & b \\
				-b & a
		\end{pmatrix}}{a,b\in \R},
		\]
		where $K$ the is the set of conformal $(2 \times 2)$ matrices.
	\end{example}
	
	
\end{example}

\begin{example}[Compensated compactness]
	For $n \ge 3$ and let $m \in \{1,..., n-1\}$, consider the first order operator $(d,d^*)$, whose symbol is given by
	\[
	[(d,d^*)(\xi)]v \coloneqq \left(\xi \wedge v, *(\xi \wedge *v)\right), \qquad v \in \Lambda^m(\R^n).
	\]
	%
\end{example}

\subsection{Scalar-valued elliptic operators} All elliptic and scalar-valued operators, in dimensions $n \ge 2$, are non-canceling In particular the following well-known examples do not satisfy our assumptions:


\begin{example}[Laplacian]\label{ex:laplacian} The easiest example of a second-order operator that is elliptic but fails to the mixing condition is the Laplacian
	\[
	\Delta u = \sum_{i =1}^n \partial^2_iu, \qquad n \ge 2.
	\]
	%
	Observe also that there is no hope for either dimensional to hold for Laplace's operator for all $n \ge 2$:
	There exists a fundamental solution $\Phi \in \Crm^\infty(\R^n \setminus \{0\})$ satisfying $\Delta \Phi = \delta_0$ on $\R^n$.
	Accordingly, $\dim_\Hcal(\Delta \Phi) = 0$. 
\end{example}

%
%

\begin{example}[The $\Acal$-Laplacian operator] For any homogeneous elliptic partial differential operator $\Acal$ from $V$ to $W$, we may consider the operator
	\[
	\Delta_\Acal u \coloneqq (\Acal^*\circ \Acal) u. 
	\]
	%
	Such operators are elliptic and possess a fundamental solution $\Delta_\Acal \Phi_\Acal = \delta_0$.
\end{example}

\begin{example}[Laplace-Beltrami operator]
	The Laplace-Beltrami operator
	\[
	\Delta \coloneqq	dd^* + d^*d 
	\]
	is elliptic and scalar-valued and fails to satisfy the mixing property.
\end{example}

    \appendix


\begin{bibdiv}
	\begin{biblist}
		
		\bib{Alberti_Bianchini_Crippa_14}{article}{
			author={Alberti, Giovanni},
			author={Bianchini, Stefano},
			author={Crippa, Gianluca},
			title={On the {$L^p$}-differentiability of certain classes of
				functions},
			date={2014},
			ISSN={0213-2230},
			journal={Rev. Mat. Iberoam.},
			volume={30},
			number={1},
			pages={349\ndash 367},
			url={https://0-doi-org.pugwash.lib.warwick.ac.uk/10.4171/RMI/782},
			review={\MR{3186944}},
		}
		
		\bib{ambrosio1997fine-properties}{article}{
			author={Ambrosio, L.},
			author={Coscia, A.},
			author={Dal~Maso, G.},
			title={Fine properties of functions with bounded deformation},
			date={1997},
			ISSN={0003-9527},
			journal={Arch. Rational Mech. Anal.},
			volume={139},
			number={3},
			pages={201\ndash 238},
			url={https://0-doi-org.pugwash.lib.warwick.ac.uk/10.1007/s002050050051},
			review={\MR{1480240}},
		}
	
		
		\bib{AFP2000}{book}{
			author={Ambrosio, L.},
			author={Fusco, N.},
			author={Pallara, D.},
			title={{Functions of bounded variation and free discontinuity
					problems}},
			series={Oxford Mathematical Monographs},
			publisher={The Clarendon Press, Oxford University Press, New York},
			date={2000},
			review={\MR{1857292}},
		}
	

        \bib{elementary}{article}{
            AUTHOR = {Arroyo-Rabasa, Adolfo},
            TITLE = {An elementary approach to the dimension of measures satisfying
                a first-order linear {PDE} constraint},
            JOURNAL = {Proc. Amer. Math. Soc.},
            FJOURNAL = {Proceedings of the American Mathematical Society},
            VOLUME = {148},
            YEAR = {2020},
            NUMBER = {1},
            PAGES = {273--282},
            ISSN = {0002-9939},
            MRCLASS = {28A78 (35F35 49Q15)},
            MRNUMBER = {4042849},
            URL = {https://0-doi-org.pugwash.lib.warwick.ac.uk/10.1090/proc/14732},
        }
    
    
    
    

\bib{YMA}{article}{
	author = {Arroyo-Rabasa, Adolfo},
	title = {Characterization of generalized Young measures generated by $\mathcal A$-free measures},
	journal = {arXiv e-prints},
	keywords = {Mathematics - Analysis of PDEs, Primary 49J45, 49Q15, Secondary 46G10, 35B05},
	year = {2019},
	month = {aug},
	eid = {arXiv:1908.03186},
	pages = {arXiv:1908.03186},
	archivePrefix = {arXiv},
	eprint = {1908.03186},
	primaryClass = {math.AP},
	adsurl = {https://ui.adsabs.harvard.edu/abs/2019arXiv190803186A},
	adsnote = {Provided by the SAO/NASA Astrophysics Data System}
}

    
    		
    \bib{GAFA}{article}{
    	author={Arroyo-Rabasa, Adolfo},
    	author={De~Philippis, Guido},
    	author={Hirsch, Jonas},
    	author={Rindler, Filip},
    	title={Dimensional estimates and rectifiability for measures satisfying
    		linear {PDE} constraints},
    	date={2019},
    	ISSN={1016-443X},
    	journal={Geom. Funct. Anal.},
    	volume={29},
    	number={3},
    	pages={639\ndash 658},
    	url={https://0-doi-org.pugwash.lib.warwick.ac.uk/10.1007/s00039-019-00497-1},
    	review={\MR{3962875}},
    }
    
    \bib{arroyo2018lower}{article}{
   author={Arroyo-Rabasa, Adolfo},
author={De Philippis, Guido},
author={Rindler, Filip},
title={Lower semicontinuity and relaxation of linear-growth integral
	functionals under PDE constraints},
journal={Adv. Calc. Var.},
volume={13},
date={2020},
number={3},
pages={219--255},
issn={1864-8258},
review={\MR{4116615}},
}	

		\bib{anna}{article}{
				author = {{Arroyo-Rabasa}, Adolfo}
				author = {{Skorobogatova}, Anna},
				title = {On the fine properties of elliptic operators},
				journal = {arXiv e-prints},
				keywords = {Mathematics - Analysis of PDEs, 49Q20, 26B30},
				year = {2019},
				month = {nov},
				eid = {arXiv:1911.08474},
				pages = {arXiv:1911.08474},
				archivePrefix = {arXiv},
				eprint = {1911.08474},
				primaryClass = {math.AP},
		}
		
		\bib{bellettini1992caratterizzazione}{article}{
			title={Una caratterizzazione dello spazio $\BD(\Omega)$ per sezioni unidimensionali},
			author={Bellettini, G.},
			author={Coscia, A.},
			journal={Seminario di Analisi Matematica, Dip. Mat. Univ. Bologna},
			pages={1992--93},
			year={1992}
		}
			
    \bib{breit2017traces}{article}{
    	author={Breit, Dominic},
    	author={Diening, Lars},
    	author={Gmeineder, Franz},
    	title={On the trace operator for functions of bounded $\mathbb A$-variation},
    	journal={Anal. PDE},
    	volume={13},
    	date={2020},
    	number={2},
    	pages={559--594},
    	issn={2157-5045},
    	review={\MR{4078236}},
    }
		
	
		
		\bib{Caccioppoli_52_I}{article}{
			author={Caccioppoli, Renato},
			title={Misura e integrazione sulle variet\`a parametriche. {I}},
			date={1952},
			ISSN={0392-7881},
			journal={Atti Accad. Naz. Lincei Rend. Cl. Sci. Fis. Mat. Nat. (8)},
			volume={12},
			pages={219\ndash 227},
			review={\MR{47757}},
		}
		
		\bib{Caccioppoli_52_II}{article}{
			author={Caccioppoli, Renato},
			title={Misura e integrazione sulle variet\`a parametriche. {II}},
			date={1952},
			ISSN={0392-7881},
			journal={Atti Accad. Naz. Lincei Rend. Cl. Sci. Fis. Mat. Nat. (8)},
			volume={12},
			pages={365\ndash 373},
			review={\MR{49990}},
		}
		
	

		
		\bib{de-giorgi1961frontiere-orien}{book}{
			author={De Giorgi, Ennio},
			title={Frontiere orientate di misura minima},
			series={Seminario di Matematica della Scuola Normale Superiore di Pisa,
				1960-61},
			publisher={Editrice Tecnico Scientifica, Pisa},
			date={1961},
			review={\MR{0179651}},
		}
		
		\bib{de-giorgi-su-una-teoria-1954}{article}{
			author={De Giorgi, Ennio},
			title={Su una teoria generale della misura {$(r-1)$}-dimensionale in uno
				spazio ad {$r$} dimensioni},
			date={1954},
			ISSN={0003-4622},
			journal={Ann. Mat. Pura Appl. (4)},
			volume={36},
			pages={191\ndash 213},
			url={https://0-doi-org.pugwash.lib.warwick.ac.uk/10.1007/BF02412838},
			review={\MR{62214}},
		}
		
		\bib{de-giorgi-nuovi-teoremi-1955}{article}{
			author={De Giorgi, Ennio},
			title={Nuovi teoremi relativi alle misure {$(r-1)$}-dimensionali in uno
				spazio ad {$r$} dimensioni},
			date={1955},
			ISSN={0035-5038},
			journal={Ricerche Mat.},
			volume={4},
			pages={95\ndash 113},
			review={\MR{74499}},
		}
		
		\bib{de-giorgi-sulla-proprieta-1958}{article}{
			author={De Giorgi, Ennio},
			title={Sulla propriet\`a isoperimetrica dell'ipersfera, nella classe
				degli insiemi aventi frontiera orientata di misura finita},
			date={1958},
			journal={Atti Accad. Naz. Lincei. Mem. Cl. Sci. Fis. Mat. Nat. Sez. I
				(8)},
			volume={5},
			pages={33\ndash 44},
			review={\MR{0098331}},
		}
	
		\bib{nin}{article}{
		author = {Del Nin, Giacomo},
		title = {Rectifiability of the jump set of locally integrable functions},
		journal = {to appear in Ann. Sc. Norm. Super. Pisa Cl. Sci.  (https://arxiv.org/abs/2001.04675)},
		month = {jan},
		primaryClass = {math.CA},
		adsurl = {https://ui.adsabs.harvard.edu/abs/2020arXiv200104675D},
		adsnote = {Provided by the SAO/NASA Astrophysics Data System}
	}
    
        \bib{DPR16}{article}{
            AUTHOR = {De~Philippis, Guido},
            AUTHOR = {Rindler, Filip},
            TITLE = {On the structure of {$\mathcal A$}-free measures and applications},
            JOURNAL = {Ann. of Math. (2)},
            FJOURNAL = {Annals of Mathematics. Second Series},
            VOLUME = {184},
            YEAR = {2016},
            NUMBER = {3},
            PAGES = {1017--1039},
            ISSN = {0003-486X},
            MRCLASS = {49Q20 (28B20 46A22)},
            MRNUMBER = {3549629},
            MRREVIEWER = {Pei Biao Zhao},
            URL = {https://0-doi-org.pugwash.lib.warwick.ac.uk/10.4007/annals.2016.184.3.10},
        }
        
		
		
		\bib{evans1992}{book}{
			author={Evans, L.~C.},
			author={Gariepy, R.~F.},
			title={Measure theory and fine properties of functions},
			series={Studies in Advanced Mathematics},
			publisher={CRC Press, Boca Raton, FL},
			date={1992},
			ISBN={0-8493-7157-0},
		}
		
		\bib{federer-note-on-gauss-green-1958}{article}{
			author={Federer, Herbert},
			title={A note on the {G}auss-{G}reen theorem},
			date={1958},
			ISSN={0002-9939},
			journal={Proc. Amer. Math. Soc.},
			volume={9},
			pages={447\ndash 451},
			url={https://0-doi-org.pugwash.lib.warwick.ac.uk/10.2307/2033002},
			review={\MR{95245}},
		}
		
		\bib{federer1969geometric-measu}{book}{
			author={Federer, Herbert},
			title={Geometric measure theory},
			series={Die Grundlehren der mathematischen Wissenschaften, Band 153},
			publisher={Springer-Verlag New York Inc., New York},
			date={1969},
			url={https://0-mathscinet-ams-org.pugwash.lib.warwick.ac.uk/mathscinet-getitem?mr=0257325},
			review={\MR{0257325}},
		}
		
		\bib{Federer1972_slices_and_potentials}{article}{
			author={Federer, Herbert},
			title={Slices and potentials},
			date={1972},
			ISSN={0022-2518},
			journal={Indiana Univ. Math. J.},
			volume={21},
			pages={373\ndash 382},
		}
		
		\bib{Fleming_Rishel_60}{article}{
			author={Fleming, Wendell~H.},
			author={Rishel, Raymond},
			title={An integral formula for total gradient variation},
			date={1960},
			ISSN={0003-889X},
			journal={Arch. Math. (Basel)},
			volume={11},
			pages={218\ndash 222},
			url={https://0-doi-org.pugwash.lib.warwick.ac.uk/10.1007/BF01236935},
			review={\MR{114892}},
		}
	
\bib{FM99}{article}{
	author={Fonseca, Irene},
	author={M\"{u}ller, Stefan},
	title={$\scr A$-quasiconvexity, lower semicontinuity, and Young measures},
	journal={SIAM J. Math. Anal.},
	volume={30},
	date={1999},
	number={6},
	pages={1355--1390},
	issn={0036-1410},
	review={\MR{1718306}},
	doi={10.1137/S0036141098339885},
}
		
	
		
		\bib{hajlasz_on_appr_diff_bd_96}{article}{
			author={Haj{\l}asz, Piotr},
			title={On approximate differentiability of functions with bounded
				deformation},
			date={1996},
			ISSN={0025-2611},
			journal={Manuscripta Math.},
			volume={91},
			number={1},
			pages={61\ndash 72},
			url={https://0-doi-org.pugwash.lib.warwick.ac.uk/10.1007/BF02567939},
			review={\MR{1404417}},
		}
    
		
		\bib{kohn1979new-estimates-f}{book}{
			author={Kohn, R.~V.},
			title={New estimates for deformations in terms of their strains},
			date={1979},
			url={http://0-gateway.proquest.com.pugwash.lib.warwick.ac.uk/openurl?url_ver=Z39.88-2004&rft_val_fmt=info:ofi/fmt:kev:mtx:dissertation&res_dat=xri:pqdiss&rft_dat=xri:pqdiss:8003789},
			note={Thesis (Ph.D.)--Princeton University},
			review={\MR{2630218}},
		}
	

		
	

		
	
		
		\bib{Raita_critical_Lp}{article}{
			author={Rai\c{t}\u{a}, Bogdan},
			title={Critical {${\rm L}^p$}-differentiability of
				{${\mathrm{BV}}^{\mathbb A}$}-maps and canceling operators},
			date={2019},
			ISSN={0002-9947},
			journal={Trans. Amer. Math. Soc.},
			volume={372},
			number={10},
			pages={7297\ndash 7326},
			url={https://0-doi-org.pugwash.lib.warwick.ac.uk/10.1090/tran/7878},
			review={\MR{4024554}},
		}
		
		\bib{smith1}{article}{
			author={Smith, K.~T.},
			title={Inequalities for formally positive integro-differential forms},
			date={1961},
			ISSN={0002-9904},
			journal={Bull. Amer. Math. Soc.},
			volume={67},
			pages={368\ndash 370},
			url={https://0-doi-org.pugwash.lib.warwick.ac.uk/10.1090/S0002-9904-1961-10622-8},
			review={\MR{142895}},
		}
    
        \bib{smith2}{article}{
            AUTHOR = {Smith, K. T.},
            TITLE = {Formulas to represent functions by their derivatives},
            JOURNAL = {Math. Ann.},
            FJOURNAL = {Mathematische Annalen},
            VOLUME = {188},
            YEAR = {1970},
            PAGES = {53--77},
            ISSN = {0025-5831},
            MRCLASS = {35.24 (26.00)},
            MRNUMBER = {282046},
            MRREVIEWER = {M. Jodeit, Jr.},
            URL = {https://0-doi-org.pugwash.lib.warwick.ac.uk/10.1007/BF01435415},
        }
    
    

    \bib{spectorVS}{article}{
    	author={Spector, Daniel},
    	author={Van Schaftingen, Jean},
    	title={Optimal embeddings into Lorentz spaces for some vector
    		differential operators via Gagliardo's lemma},
    	journal={Atti Accad. Naz. Lincei Rend. Lincei Mat. Appl.},
    	volume={30},
    	date={2019},
    	number={3},
    	pages={413--436},
    	issn={1120-6330},
    	review={\MR{4002205}},
    }
		
		\bib{van-schaftingen2013limiting-sobole}{article}{
			author={Van~Schaftingen, Jean},
			title={Limiting {S}obolev inequalities for vector fields and canceling
				linear differential operators},
			date={2013},
			ISSN={1435-9855},
			journal={J. Eur. Math. Soc.},
			volume={15},
			number={3},
			pages={877\ndash 921},
			url={https://0-mathscinet-ams-org.pugwash.lib.warwick.ac.uk/mathscinet-getitem?mr=3085095},
			review={\MR{3085095}},
		}
	
		
		\bib{Vol_pert_1967}{article}{
			author={Vol'pert, Aizik~Isaakovich},
			title={The spaces bv and quasilinear equations},
			date={1967},
			journal={Matematicheskii Sbornik},
			volume={115},
			number={2},
			pages={255\ndash 302},
		}
		
	\end{biblist}
\end{bibdiv}
    \end{document}